\documentclass[11pt,a4paper,oneside,reqno]{amsart}

\usepackage[margin=1in]{geometry}
\usepackage{setspace}
\usepackage{graphicx}

\usepackage{amsfonts,amscd,amssymb,amsmath,amsthm,bbm,hyperref}
\usepackage{centernot,caption,subcaption,color,url}

\newtheorem{theorem}{Theorem}[section]
\newtheorem{lemma}[theorem]{Lemma}

\newtheorem{conjecture}[theorem]{Conjecture}
\newtheorem{corol}[theorem]{Corollary}

\theoremstyle{definition}

\newtheorem{claim}[theorem]{Claim}
\newtheorem{remark}[theorem]{Remark}

\newcommand{\E}{\mathbb E}

\newcommand{\1}{\mathbbm{1}}

\newcommand{\Prob}{\mathbb{P}}

\author{Maksim Zhukovskii}
\address{School of Computer Science, University of Sheffield, UK}
\email{m.zhukovskii@sheffield.ac.uk}

\title{Sharp thresholds for spanning regular subgraphs}
\date{}

\begin{document}

\maketitle

\begin{abstract}
We prove that $(1+o(1))\sqrt{e/n}$ is the sharp threshold for the appearance of the square of a Hamilton cycle in $G(n,p)$, confirming the conjecture of Kahn, Narayanan, and Park. We also find the exact asymptotics of the threshold for the emergence of a spanning subgraph isomorphic to a fixed  graph $F$ for a wide family of $d$-regular graphs $F$.  This family includes almost all $d$-regular graphs.
\end{abstract}

\section{Introduction}

A graph property is called {\it increasing} if it is closed under the addition of edges. If $\mathcal{Q}$ is a non-trivial increasing property and $\mathbf{G}\sim G(n,p)$ is a binomial random graph, then, for every fixed $n$, the function $\mathbb{P}(\mathbf{G}\in\mathcal{Q})$ increases in $p$. In particular, there exists a unique solution $p_c(\mathcal{Q})$ of the equation $\mathbb{P}(\mathbf{G}\in\mathcal{Q})=1/2$, which is called the {\it probability threshold for $\mathcal{Q}$}. In 1987, Bollob\'{a}s and Thomason~\cite{BT_87} proved that, for any non-trivial increasing property $\mathcal{Q}$, whp\footnote{With high probability, that is, with probability tending to 1 as $n\to\infty$.} $\mathbf{G}\in\mathcal{Q}$ if $p\gg p_c(\mathcal{Q})$\footnote{For positive sequences $a=(a_n,n\in\mathbb{N})$ and $b=(b_n,n\in\mathbb{N})$, we write $a\ll b$ when $\lim_{n\to\infty}a_n/b_n=0$. In this case, we also write $b\gg a$.} and  whp $\mathbf{G}\notin\mathcal{Q}$ if $p\ll p_c(\mathcal{Q})$.

Since the original paper of Erd\H{o}s and R\'{e}nyi~\cite{ER} the task of determining the asymptotic behaviour of  $p_c(\mathcal{Q})$ for increasing properties $\mathcal{Q}$ has been a central topic in probabilistic combinatorics. %Since then, extensive research has been conducted on the asymptotic behavior of  $p_c(\mathcal{Q})$ for increasing properties $\mathcal{Q}$.
  While the asymptotic order of the probability threshold has been determined for many natural increasing graph properties, a general solution remains unknown,  and determining the {\it exact asymptotics} of $p_c(\mathcal{Q})$ is even more challenging. In this paper we address the latter question, explicitly posed in~\cite[Question 2]{Perkins}. We provide an answer for a class of increasing properties generated by $d$-regular graphs and resolve a notable conjecture in this area regarding the asymptotics of $p_c$ for the appearance of the square of a Hamilton cycle. In particular, our results imply that
 \begin{itemize} 
 \item for (asymptotically) almost all $d$-regular graphs $F$ on $\{1,\ldots,n\}$, the threshold for containing a spanning subgraph isomorphic to $F$ is $(1+o(1))(e/n)^{2/d}$;
 \item the thresholds for containing a square lattice and a triangular lattice are, respectively, $(1+o(1))(e/n)^{1/2}$ and $(1+o(1))(e/n)^{1/3}$;
 \item for every integer $k\geq 2$, the threshold for containing the $k$-th power of a Hamilton cycle is  $(1+o(1))(e/n)^{1/k}$, thus resolving the conjecture of Kahn, Narayanan, and Park~\cite{KNP}.
 \end{itemize}
% Since then, there has been an extensive investigation into asymptotic behaviour of $p_c(\mathcal{Q})$ for increasing properties $\mathcal{Q}$. Although for many natural increasing graph properties, asymptotics of the probability threshold was established, no general solution is known.

\subsection{Expectation threshold.} Although pinning down exact asymptotics for probability thresholds is notoriously difficult, for any increasing property $\mathcal{Q}$ the value of $p_c(\mathcal{Q})$ can be established up to a $\log n$-factor if the so called {\it expectation threshold} is known --- thanks to the remarkable result of Park and Pham~\cite{PP_KK}, which resolved the renowned conjecture of Kahn and Kalai. % that $p_c(\mathcal{Q})$ is close to the expectation threshold.
   Let us recall the definition of the expectation threshold and then state the theorem of Park and Pham. Roughly speaking, the expectation threshold $p_e(\mathcal{Q})$ is the moment when the expected number of graphs generating $\mathcal{Q}$ hits $1/2$. Formally, $p_e(\mathcal{Q})$ is the maximum $p$ such that there exists a set of graphs $\mathcal{Q}'$ satisfying $\sum_{G\in\mathcal{Q}'}p^{|E(G)|}\leq\frac{1}{2}$ and $\mathcal{Q}\subseteq\langle\mathcal{Q}'\rangle$, where $\langle\mathcal{Q}'\rangle$ is the upwards closure of $\mathcal{Q}'$. Due to Markov's inequality, $p_e(\mathcal{Q})\leq p_c(\mathcal{Q})$. Let $n_\mathcal{Q}$ be the number of edges in a largest minimal element of $\mathcal{Q}$.

\begin{theorem}[Park, Pham~\cite{PP_KK}]
There exists a universal constant $C>0$ such that the inequality $p_c(\mathcal{Q})\leq Cp_e(\mathcal{Q})\cdot\log n_\mathcal{Q}$ holds for any increasing property $\mathcal{Q}$.
\label{th:PP_KK}
\end{theorem}

The upper bound in this theorem is tight: there exists an increasing $\mathcal{Q}$ such that $p_c(\mathcal{Q})=\Theta(p_e(\mathcal{Q})\cdot\log n_{\mathcal{Q}})$. For example, let $\mathcal{Q}$ be the property of being Hamiltonian. It is a routine to check that $p_e=\Theta(1/n)$, while $p_c=(1+o(1))\frac{\ln n}{n}$~\cite{AKS,Ham3,Ham1,Ham2,Posa}. %For any $\mathcal{Q}'$ such that $\mathcal{Q}\subseteq\langle\mathcal{Q}'\rangle$, all graphs that do not contain a Hamilton cycle can be excluded from $\mathcal{Q}'$ ...
 It is fair to say that the main remaining challenge is to classify properties based on the value of the ratio $\frac{p_c(\mathcal{Q})}{p_e(\mathcal{Q})}\geq 1$, see, e.g.,~\cite{BBDFP}. In particular, {\it for which properties $\mathcal{Q}$ is $p_c(\mathcal{Q})=(1+o(1))p_e(\mathcal{Q})$?} In this paper, we explore this question for the property of containing a given (unlabelled) regular spanning subgraph.

\subsection{Regular spanning subgraphs.} 
\label{sc:intro_coarse}

Let $d\geq 3$ be a fixed constant. Let $F=F(n)$, $n\in\mathbb{N}$, be a sequence of $d$-regular $n$-vertex graphs. We assume that $dn$ is even, let $V(F(n))=\{1,\ldots,n\}$, and use the shorthand notation $[n]:=\{1,\ldots,n\}$. For $\tilde F\subset F$, we call a vertex $v\in V(\tilde F)$ {\it boundary}, if it has degree strictly less than $d$ in $\tilde F$. We call {\it the vertex boundary of $\tilde F$} the set $\partial_v(\tilde F)$ of all boundary vertices and {\it the edge boundary} the set $\partial_e(\tilde F)$ of all edges of $F$ between $\partial_v(\tilde F)$ and $V(F)\setminus V(\tilde F)$. Let the increasing property $\mathcal{Q}_F$ be generated by the family $\mathcal{F}_n$ of all isomorphic copies of $F$ on $[n]$. For simplicity, we denote $p_c(F):=p_c(\mathcal{Q}_F)$ and $p_e(F):=p_e(\mathcal{Q}_F)$. We also let $X_F$ be the number of graphs from $\mathcal{F}_n$ that are subgraphs of $\mathbf{G}\sim G(n,p)$. Since
$$
\E X_F=|\mathcal{F}_n|p^{dn/2}=\frac{n!}{|\mathrm{Aut}(F)|}p^{dn/2}\leq n!\cdot p^{dn/2},
$$ 
we get that $p_c(F)\geq p_e(F)\geq (1-o(1))(e/n)^{2/d}$. On the other hand, if every proper subgraph of $F$ has edge boundary of size at least $d$, then $p_e(F)=O(n^{-2/d})$ since it is not hard to see that $\mathcal{F}_n$ is {\it $O(n^{-2/d})$-spread}, i.e. for every graph $H\subset F$, 
$$
|\mathcal{F}_n\cap\langle\{H\}\rangle|\leq (Cn^{-2/d})^{|E(H)|}\cdot |\mathcal{F}_n|\quad\text{ for some constant }C=C(d)>0.
$$
%$ since the number of automorphisms of a connected subgraph $\tilde F\subset F$ is at most exponential in $|V(\tilde F)|$, see e.g.~\cite{KLT}.
  Therefore, due to Theorem~\ref{th:PP_KK}, $p_c(F)=O(n^{-2/d}\cdot\log n)$.

Riordan proved~\cite{Riordan} that the logarithmic factor can be removed, i.e. 
\begin{equation}
 p_c(F)=\Theta(p_e(F))=\Theta(n^{-2/d}),
\label{eq:p_c=p_e}
\end{equation} 
if $F$ satisfies a stronger condition on the edge boundary: every $\tilde F\subset F$ with $3\leq|V(\tilde F)|=o(\log n)$ satisfies $|\partial_e(\tilde F)|\geq 2d$. In particular, for powers of a Hamilton cycle\footnote{For a graph $G$, its $k$-th power $G^k$ is obtained by adding to $G$ edges between vertices that are at distance at most $k$ in the graph metric induced by $G$.}, this result implies the following: for every $k\geq 3$, the threshold probability for containing the $k$-th power of a Hamilton cycle equals $\Theta(n^{-1/k})$. However, the proof of Riordan does not work for $k=2$ since the minimum size of the edge boundary of a non-trivial subgraph of the square of a Hamilton cycle equals $d+2=6<2d$. We also note that, for a $d$-regular graph to satisfy the condition of Riordan, $d$ should be at least 4. In~\cite{KO}, K\"{u}hn and Osthus proved that $n^{-1/2+o(1)}$ is the threshold probability for containing the square of a Hamilton cycle and conjectured that the threshold is actually $\Theta(n^{-1/2})$. In~\cite{NS}, Nenadov and \v{S}kori\'{c} proved the upper bound $n^{-1/2}(\log n)^4$, which was improved  to $n^{-1/2}(\log n)^3$ by Fischer, \v{S}kori\'{c}, Steger, and Truji\'{c} in~\cite{FSST}, and to $n^{-1/2}(\log n)^{2}$ in an unpublished work of Montgomery (see~\cite{Frieze}). From Theorem~\ref{th:PP_KK} (as well as from its weaker fractional version~\cite{FKNP}), it follows that the threshold is $O(n^{-1/2}\log n)$. The conjecture of K\"{u}hn and Osthus was eventually resolved by Kahn, Narayanan, and Park in 2020~\cite{KNP}. They further conjectured that $p_c(F)=(1+o(1))\sqrt{e/n}$. This paper resolves the conjecture, see Section~\ref{sc:KNP_resolution}. %Another example of a $d$-regular $F$ that satisfies the condition of Riordan is the square tori $T_{\sqrt{n}\times\sqrt{n}}$. 

On the other hand, if $F$ has many small subgraphs with smaller edge boundaries, then it may happen that $p_c(F)\gg p_e(F)$. For instance, assume that there exists $v=v(n)=o(\log n)$ and a graph $H=H(n)$ on $[v]$ with $\frac{dv}{2}-\frac{d}{2}$ edges such that every vertex of $F$ belongs to a subgraph which is isomorphic to $H$ (clearly, its edge boundary equals $d$). % with the edge boundary of size $d$. In other words, $H$ has at least $\frac{dv}{2}-\frac{d}{2}+\lfloor \omega(v/\log n)\rfloor$ edges.
  Then, an increasing (at most logarithmic) factor arises % since in order to contain a copy of $F$, the random graph should have every vertex inside a copy of $H$
    --- for constant $v$ see the proof in~\cite{Spencer} which can be directly generalised to all $v=o(\log n)$. % by observing that it is sufficient to consider $v=O(\log n)$.  % the {\it 1-density} of this graph equals $d/2$, , i.e. the {\it 1-density} of this graph equals $d/2$ --- see the proof for constant $v$ in~\cite{Spencer}, for $v=O(\log n)$ the proof is similar.
  Actually, if we allow even smaller edge boundaries, then the threshold probability may increase by a power of~$n$, as well as the expectation threshold. For instance, for every $d$-regular strictly 1-balanced graph $H$ of constant size $v$, the threshold probability for containing an $H$-factor equals $\Theta\left(n^{-2/d+2/(vd)}(\ln n)^{2/(dv)}\right)$~\cite{BHKMP,JKV}.

In~\cite{CHL}, Chen, Han, and Luo showed that the proof of Kahn, Narayanan, and Park~\cite{KNP} can be generalised to show that~\eqref{eq:p_c=p_e} holds for all $d$-regular graphs $F$ satisfying $|\partial_e(\tilde F)|\geq d+1$ for every $\tilde F\subset F$ with $3\leq|V(\tilde F)|\leq \varepsilon n$. Notice that, for large enough $|V(\tilde F)|$, the bound on the edge boundary in this result is weaker than the condition of Riordan. We show that the weakest of the two conditions always implies~\eqref{eq:p_c=p_e}.  
%For a $d$-regular graph $F$ on the vertex set $[n]$, let $p^*$ be the unique solution of the equation $\frac{n!}{|\mathrm{Aut}(F)|}p^{dn/2}=1$. \MZ{Define $p^*_F$, $p^e_F$}

\begin{theorem}
Let $d\geq 3$, $\varepsilon>0$ be constants, and let $F=F(n)$ be a sequence of $d$-regular graphs on $[n]$, $n\in\mathbb{N}$, such that $|\partial_e(\tilde F)|\geq d+1$, %-C|V(\tilde F)|/\ln n$,
%\begin{itemize}
%\item  for every $\varepsilon>0$ and all large enough $n$ the number of automorphisms of $F_n$ is less then $e^{\varepsilon n^{2/d}}$; \MZ{do we need it? I believe the bound on the number of automorphisms can be just omitted.}
%\item
  for every $\tilde F\subset F$ with $3\leq|V(\tilde F)|\leq\varepsilon\ln n$.
% $|V(\tilde F)|\geq\frac{2|E(\tilde F)|}{d}+\frac{d+1}{d}$.
%\end{itemize}
 Then~\eqref{eq:p_c=p_e} holds. % is a threshold for containing a copy of $F_n$. % Let $\varepsilon>0.$ If $p>(1+\varepsilon)dp^*$, then whp (assuming that $dn$ is even) $G(n,p)$ contains a copy of $F_n$.
\label{th:main1}
\end{theorem}

As follows from the above discussion, our bound on edge boundaries is tight, that is, in some sense, the condition on the edge boundary is the only obstacle in getting the equality~\eqref{eq:p_c=p_e}. Indeed, for every $3\leq v=o(\log n)$ and every $0\leq\ell\leq d$ (such that $dv-\ell$ is even), there exists a graph $F$ that have a subgraph on $v$ vertices with edge boundary of size $\ell$ such that~\eqref{eq:p_c=p_e} does not hold. The proof of Theorem~\ref{th:main1} uses similar techniques as in~\cite{CHL,KNP} and  does not require accurate spread approximations, for this reason we relegate it to Appendix~\ref{sc:theorem_coarse_proof}.

\begin{remark}
\label{rk:best_possible_upper_bounds}
Our proof allows to get an explicit upper bound on the threshold probability, which we did not try to optimise. However, for graphs with small number of automorphisms, a refined argument gives a better upper bound: Assume, for instance, that, for every proper connected subgraph of $F$, the number of automorphisms that preserve boundary vertices is bounded and that every subgraph on at least 3 and at most $n/2$ vertices has edge boundary of size at least $d+1$. Then $p_c(F)\leq\left(e^{2/d}+\frac{d}{1+\1_{d\text{ is even}}}-1+o(1)\right)n^{-2/d}$. The proof of this bound follows the same lines as the proof of Theorem~\ref{th:main1}, but requires more accurate computations. In particular, in the proof of Theorem~\ref{th:main1}, we do not need very accurate estimates for the number of subgraphs in $F$ and for the number of ways to embed such a subgraph into $F$, provided by Claim~\ref{cl:main} in Section~\ref{sc:spread}. However, the refined bound requires the full power of Claim~\ref{cl:main}.  For instance, for the square of a Hamilton cycle $F$, the refined bound implies that the threshold probability $p_c(F)$ belongs to a fairly tight interval $\left[\frac{\sqrt{e}-o(1)}{\sqrt{n}},\frac{\sqrt{e}+1+o(1)}{\sqrt{n}}\right]$. As we will see further, it actually coincides asymptotically with the left boundary of this interval.
\end{remark}

%It immediately implies that the threshold probability for containing a copy of $F$ equals $p(n)=\Theta(n^{-2/d})$. As we mentioned above, the restriction on edge boundaries is tight --- if we allow subgraphs with edge boundary of size $d$ instead of $d+1$, then the assertion becomes false.\\

%Note that a bound on the number of symmetries can not be omitted --- as soon as the number of automorphisms of $F$ becomes larger than $e^{o(n)}$, $p^*$ becomes larger as well. In particular, $p(n)=(d!\log n)^{\frac{2}{d(d+1)}}n^{-2/(d+1)}$ is a sharp threshold for the existence of a $K_{d+1}$-factor~\cite{Riordan_factors}.\\

\subsection{Sharp thresholds.} Let $\mathbf{G}\sim G(n,p)$. The threshold probability $p_c(\mathcal{Q})$ is called {\it sharp}, if, for every constant $\varepsilon>0$, the following is true: if $p>(1+\varepsilon)p_c(\mathcal{Q})$, then whp $\mathbf{G}\in\mathcal{Q}$; if $p<(1-\varepsilon)p_c(\mathcal{Q})$, then whp $\mathbf{G}\notin\mathcal{Q}$. Friedgut proved~\cite{Friedgut} that all increasing and isomorphism-closed\footnote{In the appendix to this paper, written by Bourgain, an alternative version of this result --- one that does not require the symmetry condition --- is proved.} properties that do not have sharp thresholds are essentially local: with large probability, they are equivalent to the presence of certain property-specific subgraphs of bounded sizes. Omitting technical details, it means that, for some $p=\Theta(p_c(\mathcal{Q}))$, $\mathbf{G}\sim G(n,p)$ sprinkled with $\mathbf{G}'\sim G(n,\varepsilon p)$, for some $\varepsilon>0$, is less likely to have the property $\mathcal{Q}$, than the union of $\mathbf{G}$  with an independently sampled random clique of some bounded size (see, e.g.,~\cite[Theorem 2.3]{Friedgut2}). With this remarkable result, it is fairly straightforward to check that the threshold probabilities in Theorem~\ref{th:main1} are sharp. To the best of our knowledge, there are no general results describing asymptotics of sharp $p_c$ for a wide class of increasing properties (say, for $\mathcal{Q}_F$, where $F$ are $d$-regular). In this paper, we find asymptotics of $p_c(F)$ for most $d$-regular graphs $F$.

%[Sharp vs not sharp]  [Small subgraphs, Friedgut]

% However, they did non settle a right constant in front of $n^{-1/2}$ and conjectured that the right constant is $\sqrt{e}$ and that the threshold $p(n)=\sqrt{e/n}(1+o(1))$ is {\it sharp} (i.e., if $p>(1+\varepsilon)\sqrt{e/n}$, then whp $G(n,p)$ contains the second power of a Hamilton cycle). In this paper, we prove this conjecture and even more: for $d\leq 4$ the requirement from Theorem~\ref{th:main1} guarantees that $p(n)=(1+o(1))\sqrt{e/n}$ is even {\it sharp}; however, for $d\geq 5$ we need to strengthen the bound on the edge boundary to $2d-1$ (note that this is still better than the condition of Riordan).

\begin{theorem}
Let $d\geq 3$ and let $F=F(n)$ be a sequence of $d$-regular graphs on $[n]$, $n\in\mathbb{N}$, such that one of the following two conditions holds.
\begin{enumerate}
\item There exists $\delta\in(0,1/d)$ and $w=w(n)=\omega(1)$ such that 

(a) $|\mathrm{Aut}(F)|\leq e^{o(n^{1-\delta})}$; and 

(b) every $\tilde F\subset F$ with $3\leq|V(\tilde F)|\leq n-3$ has 
$$
|\partial_e(\tilde F)|\geq d+1+\left\lfloor |V(\tilde F)|\frac{w}{\ln n}\right\rfloor\cdot\1_{\ln n/w\leq|V(\tilde F)|\leq n^{1-\delta}}.
$$%, if $3\leq|V(\tilde F)|\leq n-3$. %n^{1-\delta}$, and $|\partial_e(\tilde F)|\geq d+1$, if $n^{1-\delta}<|V(\tilde F)|\leq n-3$.

\item Let $\gamma\in(0,1)$ be a fixed constant, let $C=C(d,\gamma)$ be large enough, and let $D\geq 1$ be a fixed constant.

(a) for every subgraph $\tilde F\subset F$ with $|E(\tilde F)|\leq C n^{2/d}$ and $|\partial_v(\tilde F)|\leq \gamma|V(\tilde F)|$, the number of automorphisms of $\tilde F$ that fix all vertices of $\partial_v(\tilde F)$ is at most $D$;

(b) for every subgraph $\tilde F\subseteq F$ with $|\partial_v(\tilde F)|\leq \gamma|V(\tilde F)|$, the number of automorphisms of $\tilde F$ that fix all vertices of $\partial_v(\tilde F)$ is at most $e^{o(|V(\tilde F)|)}$;

(c) every $\tilde F\subset F$ has $|\partial_e(\tilde F)|\geq 2d$, if $3\leq|V(\tilde F)|\leq n-3$.
\end{enumerate}
Then 
\begin{equation}
p_c(F)=(1+o(1))p_e(F)=(1+o(1))\left(\frac{e}{n}\right)^{2/d}\quad\text{ is sharp}.
\label{eq:sharp}
\end{equation}
%Let $\varepsilon>0.$ If $p>(1+\varepsilon)\left(\frac{e}{n}\right)^{2/d}$, then whp (assuming that $dn$ is even) $G(n,p)$ contains a copy of $F_n$.
\label{th:main2}
\end{theorem}

We note that the condition 2.(c) implies $d\geq 4$. Though the second condition in Theorem~\ref{th:main2} is strictly weaker than the condition of Riordan, observe that the first condition is not such. For instance, a random triangle-free 3-regular graph has linearly many subgraphs on 3 vertices with edge boundary of size $5<6=2d$ --- so, it does not satisfy the condition of Riordan, while it satisfies 1.(a) and 1.(b) whp, see below. % take a 3-lift of a random $3$-regular graph and draw a triangle on every set of three clones. This $d$-regular graph  ($d=5$) has linearly many triangles with edge boundary $9<2d$ --- so, it does not satisfy the condition of Riordan, while it satisfies 1.(a) and 1.(b) whp. 

Note that, for every $d\geq 4$, almost all $d$-regular graphs satisfy the first condition in Theorem~\ref{th:main2} (see~\cite{Bollobas_iso,KSV,MW}). Indeed, for every $d\geq 3$, letting $\mathbf{G}_d$ be a random $d$-regular graph on $[n]$, whp $\mathbf{G}_d$ is asymmetric~\cite{MW}, implying 1.(a). % Since, for every subgraph $\tilde F\subset\mathbf{G}_d$, every automorphism of $\tilde F$ that fix all vertices of $\partial_v(\tilde F)$ can be extended to a non-trivial automorphism of $\mathbf{G}_d$, implying (b) and (c).
 Moreover, whp $\mathbf{G}_d$ has Cheeger constant $\Theta(1)$~\cite{Bollobas_iso} and, whenever $d\geq 4$, for every fixed graph $H$ with at least $(d|V(H)|-d)/2$ edges, whp $\mathbf{G}_d$ does not have a subgraph isomorphic to $H$~\cite{KSV}, implying 1.(b). If $d=3$, then whp there are no subgraphs with $3\leq v\leq n-3$ vertices and the edge boundary of size at most $d$ % at least $\frac{d}{2}v-\frac{d+1}{2}$ edges
   other than $K_3$ and their vertex-complements. Therefore, whp $\mathbf{G}_d$ satisfies the first condition, if it does not contain triangles. Also, $k$-th powers of Hamilton cycles for every $k\geq 3$ and toroidal square grids $T_{m,n/m}\cong C_m\square C_{n/m}$ for $m\geq 4$ satisfy the second condition as well. It is also easy to see that the second condition implies sharp thresholds for square and triangular lattices (by considering their regular completions). It is worth recalling that thresholds for lattices were investigated in the past: the question of which random graphs contain spanning lattices was raised by Venkatesan and Levin~\cite{LV}. Thresholds for lattices, up to a polylog-factor, were proved in~\cite{FdVM} and, up to a constant factor, were proved in~\cite{Riordan} (they also follow from Theorem~\ref{th:main1} and~\cite[Corollary 1.5]{CHL}). Summing up, we get the following corollary that demonstrates that the set of regular graphs $F$ satisfying the restrictions in Theorem~\ref{th:main2} is quite broad.

%[Discussions]

%[Examples]

%Note that, for every $d\geq 4$, almost all $d$-regular graphs are good (see~\cite{Bollobas_iso,KSV,MW}). In particular, if $d\geq 5$, then whp in a random $d$-regular graph on $[n]$ there are no subgraphs with $3\leq v\leq n-3$ vertices and the edge boundary of size at most $2d-2$. 

% If $d=4$, then whp there are no subgraphs with $3\leq v\leq n-3$ vertices and the edge boundary of size at most $d$. % at least $\frac{d}{2}v-\frac{d+1}{2}$ edges.
%  If $d=3$, then whp there are no subgraphs with $3\leq v\leq n-3$ vertices and the edge boundary of size at most $d$ % at least $\frac{d}{2}v-\frac{d+1}{2}$ edges
 %  other than $K_3$ and their vertex-complements. Therefore, a random 3-regular graph is good whp under the condition that it does not contain triangles.

\begin{corol}
The following graphs $F$ satisfy~\eqref{eq:sharp}:
\begin{itemize}
\item $k$-th power of a cycle, $k\geq 3$,
\item toroidal grid $T_{m\times n/m}$, $m\geq 4$ (assuming that $n$ is divisible by $m$),
\item square and triangular lattices with $d=4$ and $d=6$ respectively.
\end{itemize}
Moreover, \eqref{eq:sharp} holds for every $d\geq 4$ and (asymptotically) almost all $d$-regular graphs $F$ on $[n]$ (assuming $dn$ is even) and for (asymptotically) almost all triangle-free 3-regular graphs  $F$ on $[n]$ (assuming $n$ is odd).
\label{corol}
\end{corol}

\begin{remark}
Actually, using our methods, we are able to establish the same sharp threshold for almost all 3-regular graphs --- the condition of the absence of triangles is redundant, since the number of triangles converges in probability to a Poisson random variable~\cite{Wormald}, and so it is bounded in probability. In other words, we may allow $F$ to have a bounded number of subgraphs with a smaller edge boundary. However, we do not want to overload the proof with technical details, and so we formulate Theorem~\ref{th:main2} as well as Corollary~\ref{corol} in their current forms.

Through private communication, we learned that Tam\'{a}s Makai, Matija Pasch, Kalina Petrova, and Leon Schiller are independently working on determining the asymptotics of $p_c(F)$ for $k$-th powers of cycles $F$. They proved~\eqref{eq:sharp} for all $k\geq 4$ and plan to upload this result.
\end{remark}

%Fernandez de la Vega and Manoussakis~\cite{FdVM} proved that $p_c(F)=O((\log n/n)^{2/3})$ for a hexagonal lattice $F$ and Riordan noticed in~\cite{Riordan} that his methods are not applied to 3-regular graphs. Even though Theorem~\ref{th:main1} implies~\eqref{eq:p_c=p_e} for the hexagonal lattice $F$ (as well as~\cite[Corollary 1.5]{CHL}), Theorem~\ref{th:main2} is not applicable in this case. Other 

Examples of graphs that do not satisfy restrictions from Theorem~\ref{th:main2} include the square of a cycle, toroidal grid $T_{3,n/3}$, and, so-called, 2-overlapping 4-cycles. The latter graph, denoted by $C^e_{4,n}$, is obtained from $n/2$ cyclically order copies of $C_4$, where two consecutive $C_4$ overlap in exactly one edge, whereby each $C_4$ overlaps with two copies of $C_4$ in opposite edges. Recall that Kahn, Narayanan, and Park in~\cite{KNP} proved~\eqref{eq:p_c=p_e} for the square of a cycle $F$, resolving the conjecture of K\"{u}hn and Osthus~\cite{KO}. Moreover, using the same method, Espuny D\'{i}az and Person in~\cite{DP} proved~\eqref{eq:p_c=p_e} for $F=C^e_{4,n}$ answering the question of Frieze~\cite{Frieze_C4} (we note that~\eqref{eq:p_c=p_e} for these graphs also follow both from Theorem~\ref{th:main1} and~\cite[Corollary 1.5]{CHL}).

\subsection{Squares of Hamilton cycles.}
\label{sc:KNP_resolution}

In this section, we present results that complement Theorem~\ref{th:main2} and resolve the conjecture of Kahn, Narayanan, and Park. The novelty of its proof technique constitutes our main contribution.

In 2020, Kahn, Narayanan, and Park~\cite{KNP} conjectured that~\eqref{eq:sharp} holds when $F$ is the square of a cycle. Due to the natural barrier in the application of the fragmentation technique that was used to prove~\eqref{eq:p_c=p_e} for the square of a Hamilton cycle~\cite{KNP}, to resolve the conjecture of Kahn and Kalai~\cite{PP_KK}, and to prove the so called spread lemma~\cite{ALWZ} (see details in Section~\ref{sc:intro_strategy}), this conjecture attracted significant interest and was reiterated in~\cite{Frieze,Perkins}. In this paper, we prove the conjecture.

\begin{theorem}
Let $F=F(n)$ be the square of an $n$-cycle on $[n]$. Then $F$ satisfies~\eqref{eq:sharp}.
\label{th:second_power}
\end{theorem}

We then show that the method that we use to prove this conjecture can be used to prove a more general result for all $d$-regular graphs that have ``cyclic'' structure.

\begin{theorem}
Let $d\geq 3$ be a constant, $r=r(n)=o(\log\log n/\log\log\log n)$\footnote{We did not try to optimise the bound on $r$ and we believe that it can be improved using our methods, but not beyond $\log n$.}, and $F=F(n)$ be a sequence of graphs on $[n]$ satisfying the following.  
\begin{itemize}
\item Every subgraph $\tilde F\subset F$ with $3\leq|V(\tilde F)|\leq r$ has $|\partial_e(\tilde F)|\geq d+1$.
\item The bijection that maps each $v\in[n]$ to $v+r$\footnote{Vertices here are treated as elements of the cyclic group $(\mathbb{Z}_n,+)$.} is an automorphism of $F$.
\item Every $u,v\in[n]$, such that the distance between $u$ and $v$ in the cyclic order on $[n]$ is more than $r$, are not adjacent in $F$.
\item There are no automorphisms of $F$ that fix vertices from $[r]$. % and $F[[r]]$ has $O(1)$ automorphisms.
\end{itemize}
Then $F$ satisfies~\eqref{eq:sharp}.
\label{th:second_power_generalisation}
\end{theorem}

In particular,~\eqref{eq:sharp} holds for the two other `naive' cases that Theorem~\ref{th:main2} does not cover --- for $F=C^e_{4,n}$ and $F=T_{3,n/3}$. Indeed, for these two graphs the requirements of Theorem~\ref{th:second_power_generalisation} are satisfied with $r=4$ and $r=3$, respectively. % the hexagonal lattice $F$ as well as for the $F=C^e_{4,n}$.

\subsection{Proof strategy.}
\label{sc:intro_strategy}

The crux of our proofs is the {\it fragmentation trick} that in different forms appeared in many applications. One of them is the famous {\it spread lemma}~\cite{ALWZ} which in particular gives good sunflower bounds~\cite{Tao}; in probabilistic terms the application of the trick for the spread lemma is described in~\cite{MN-WSZ}. This trick was also the main ingredient in the proof of Theorem~\ref{th:PP_KK} in~\cite{PP_KK} and in the proof of Frankston, Kahn, Narayanan, and Park~\cite{FKNP} of Talagrand's fractional version of the conjecture of Kahn and Kalai~\cite{Talagrand}. Our proofs partially rely on {\it typical} fragments, similarly to~\cite{KNP}, rather than {\it minimal} fragments, that were used by Park and Pham in~\cite{PP_KK} to prove Theorem~\ref{th:PP_KK} --- see more details below. 

{\it In order to prove Theorem~\ref{th:second_power}}, we introduce a novel technique of choosing fragments --- we show that the ``cyclic'' structure of squares of Hamilton cycles allows to choose rare fragments that are sparser than typical ones, from sufficiently many copies of $F$.
% Kahn, Narayanan and Park~\cite{KNP} used the ``planted trick'' to prove their results on threshold probabilities as well. In essence, the key idea is to ``plant'' a graph $F$ from the family $\mathcal{F}_n$ and to combine it with the noise produced by $G(n,p)$. Then, it turns out that whp there exists a graph $F'\in\mathcal{F}_n$ which is entirely inside the perturbed planted hyperedge $F\cup G(n,p)$ such that the size of $F'\setminus G(n,p)$ is quite small. This allows to replace $\mathcal{F}_n$ with the set of fragments of $F\in\mathcal{F}_n$ equal to $F'\setminus G(n,p)$, to draw independently edges of another $G(n,p)$ and to apply the same argument once again. If the number of steps in this procedure is bounded by a constant, then we get that the threshold probability has the same order of magnitude as $p^*$.
 The key idea from~\cite{KNP} that proves~\eqref{eq:p_c=p_e} for squares of cycles is as follows. Fix any $F\in\mathcal{F}_n$, before the edges of the uniformly random graph $\mathbf{W}\sim G(n,m)$ are exposed, where $m=\lfloor p{n\choose 2}\rfloor$. Consider $F\cup\mathbf{W}$ and show that whp almost every $F'\subset F\cup\mathbf{W}$ from $\mathcal{F}_n$ is such that $|F'\setminus\mathbf{W}|=O(\sqrt{n})$. It means that we may replace almost every $F\in\mathcal{F}_n$ with its {\it fragment} $F'\setminus\mathbf{W}\subset F$ --- for typical $F'\subset F\cup\mathbf{W}$ --- of size $O(\sqrt{n})$ and expose another $\mathbf{\tilde W}\sim G(n,\tilde m)$ independently where $\tilde m$ is large enough to cover at least one fragment whp. Then whp $\mathbf{W}\cup\mathbf{\tilde W}$ contains a graph from $\mathcal{F}_n$.

 %In some sense, the planted hyperedge `learn' how to find another copy from $\mathcal{F}_n$ in its perturbed version. However, the following effect of overlearning appears: as soon as $F'$ contains an almost closed $H\subset F$ (but maybe several missing intermidiate edges), it might contain the whole fragment as well. It
However, if a fragment contains a subgraph with edge boundary of size $d+2=6$ with $\Omega(\ln n)$ edges --- such a subgraph is exactly the square of a path, we call it {\it closed} --- then $\tilde m=\lceil\delta n\sqrt{n}\rceil$ additional uniformly random edges with a certain bounded from zero $\delta>0$ is not enough to cover it. Unfortunately, we suspect that typical fragments do contain closed subgraphs of size $\Omega(\ln n)$. So, this approach does not allow to get a sharp bound on the threshold. This is the main complication that does not allow to apply directly the fragmentation technique to resolve the conjecture of Kahn, Narayanan, and Park. A careful implementation of the technique gives upper bound on the threshold $(\sqrt{e}+1+o(1))n^{-1/2}$, see Remark~\ref{rk:best_possible_upper_bounds}.
%the straightforward application of this trick results in an overcount in the upper bound for the desired probability and thus in impossibility of applying the original `planted hyperedge' trick since the fragments may contain closed subgraphs of at least logarithmic size, and $\tilde w=\delta n\sqrt{n}$ additional edges with small enough $\delta>0$ is not enough to cover it. %The reason for this is the following: as soon as $F'$ contains an almost closed $H\subset F$ (but maybe one of its intermediate edges $e$), it might contain the edge $e$ as well.

Nevertheless, we show that there exists a large family of fragments that do not contain closed subgraphs of size $\Omega(\ln n)$ and which is still ``well-spread''. %$n^{-\gamma}$-spread for some small enough $\gamma>\frac{1}{2}$.
   A direct approach would be, for a typical $F\in\mathcal{F}_n$ (in the $\mathbf{W}$-measure), to search for a fragment $H\subset F$ that does not have large closed subgraphs --- indeed, since typically, the set of fragments, for a given $F$, is large, we suspect that such a fragment exists. There might be some clever decoupling argument that allows to implement this approach, but we could not find it. Instead, we improve a typical fragment manually: for a fragment $H\subset F$ we distribute evenly vertices of maximal closed subgraphs between them. This is possible due to the following crucial observation. Let $P_1=(u^1_1u^1_2v_1\ldots v_pu^1_3u^1_4)$ and $P_2=(u^2_1u^2_2v_{p+1}\ldots v_{p+p'}u^2_3u^2_4)$ be disjoint closed subgraphs of a fragment $H\subset F\in\mathcal{F}_n$. Then the graph $H'$ obtained from $H$ by applying any permutation to ``internal'' vertices $v_i$ of $P_1\sqcup P_2$ remains a fragment of {\it some other} $F'\in\mathcal{F}_n$. More formally,
\begin{itemize} 
\item let $\pi$ be any permutation on $[p+p']$,
\item let $t\in\{0,1,\ldots,p+p'\}$, and
\item let $H'$ be obtained from $H$ by replacing $P_1$ with $P_1^{\pi}:=(u^1_1u^1_2 v_{\pi(1)}\ldots v_{\pi(t)} u^1_3u^1_4)$ and $P_2$ with $P_2^{\pi}:=(u^2_1u^2_2 v_{\pi(t+1)}\ldots v_{\pi(p+p')} u^2_3u^2_4)$.
\end{itemize} 
Then, if $H\cup\mathbf{W}$ contains a graph from $\mathcal{F}_n$, we get that $H'\cup\mathbf{W}$ contains a graph from $\mathcal{F}_n$ as well. Moreover, $H$ and $H'$ have exactly the same number of edges, vertices, and connected components.

Although it allows to get a graph $H'$ that does not have large closed subgraphs, this graph is no longer a subgraph of $F$. Since the `improving' function that maps each $F$ to the modified fragment may have large pre-images, this modification of fragments may affect spreadness properties of the multiset of fragments. We implement a probabilistic approach to show that a suitable way to modify fragments exists: we distribute vertices among closed subgraphs according to a rule that is described by a perfect matching in a binomial random bipartite graph with an appropriate edge probability, sampled independently of $\mathbf{W}\sim G(n,m)$. Using this rule, we prove that the maximum cardinality of a pre-image of the `improving' function differs by a sufficiently small factor from the maximum cardinality of a pre-image of the original function that maps every $F$ to its fragment. It allows to reduce sizes of fragments to $o(\log n)$ after $O(\log\log n)$ fragmentation rounds, where edges are sprinkled with probability $\varepsilon/(\log\log n\cdot\sqrt{n}))$, for a sufficiently small $\varepsilon>0$.

Another complication is that, in order to implement this improvement, we need to have sufficiently many disjoint maximal closed subgraphs with at least 4 vertices, so that after vertices are distributed evenly between them, each closed subgraph has size $o(\log n)$. Although it seems plausible that typical fragments contain sufficiently many such closed subgraphs whp, we did not find a way to show this. Instead, we force typical fragments to contain a fixed set of $\omega(\sqrt{n}/\log n)$ closed subgraphs with 4 vertices ---{\it diamonds}. This is implemented via restricting the family $\mathcal{F}_n$ to the set of all cycles that contain fixed diamonds on specific positions. Due to symmetry and linearity of expectation, we then are able to extend the multiset of fragments to the entire $\mathcal{F}_n$.

It then remains to cover a fragment of size $o(\log n)$. To show that $G(n,\varepsilon n^{-1/2})$ contains such a fragment whp, it suffices to ensure that the initial fragments --- those collected after the first fragmentation round, before any refinement --- do not contain closed subgraphs of size at least $\sqrt{n}/\ln n$. Although we believe this is typically the case, we do not see a way to decouple the random graph and the random fragment to rigorously establish this. Instead, for every pair $(F\in\mathcal{F}_n,W\in{{[n]\choose 2}\choose m})$, we introduce a non-uniform distribution over the set of all $(F,W)$-fragments with atoms being fragments that have {\it disjoint} sets of inclusion-maximal closed subgraphs of size at least $\sqrt{n}/\ln n$. We show that such a distribution exists for almost all pairs $(F,W)$, and then we sample a fragment from this distribution for every suitable pair $(F,W)$ independently.

{\it The proof of Theorem~\ref{th:main2}} is much more straightforward --- typical fragments are good enough to derive~\eqref{eq:sharp}. Indeed, both conditions exclude closed subgraph of size $\Omega(\log n)$.  We prove the first part of Theorem~\ref{th:main2} via $1/\delta$ fragmentation steps. The proof of the second part  of Theorem~\ref{th:main2} requires only two fragmentations. Actually the usual second moment method works in this case --- for the uniform model instead of the binomial, as in the paper of Riordan~\cite{Riordan}. It is possible to show that $\mathrm{Var}X_F=O((\mathbb{E}X_F)^2)$ and then to apply the powerful results of Friedgut~\cite{Friedgut2,Friedgut}. In particular, this strategy was used by Narayanan and Schacht to determine sharp thresholds for non-loose Hamiltonian cycles in random hypergraphs~\cite{NS}. However, we give the proof of the second part of Theorem~\ref{th:main2} using fragmentation for the sake of convenience and coherence. 

We believe that our analysis is essentially optimal and significant improvements of conditions in Theorem~\ref{th:main2} require new ideas. Indeed, we obtain fairly optimal bounds on the number of graphs $|\mathcal{F}_n\cap\langle I\rangle |$ containing a given set of edges $I$ and on the number of subgraphs of $F$ with fixed numbers of vertices, edges, and components (see Claim~\ref{cl:main}). The main novel ingredient in the proof of Claim~\ref{cl:main} is a very nice property of $d$-regular graphs that do not have non-trivial subgraphs with edge boundary of size at most $d$: for every $v$, there are 1) constantly many closed subgraphs on $v$ vertices that share a vertex, 2) at most linearly many closed subgraphs on $v$ vertices in $F$, see Claim~\ref{cl:count_copies_vertex} and Claim~\ref{cl:total_number_closed}.

%   and the method in its current form cannot be used to weaken the bound on edge boundaries in Theorem~\ref{th:main2} for $d\geq 5$. Our main achievement is that we make a step beyond the usage of the notions of spreadness and superspreadness. We obtain optimal bounds on the number of hyperedges containing a given set of edges $I$ (commonly denoted by $|\mathcal{F}_n\cap\langle I\rangle |$) and on the number of subgraphs of $F$ with a fixed number of vertices, edges and components (see Section~\ref{sc:spread} and Claim~\ref{cl:main}). The main ingredient of the proof of Claim~\ref{cl:main} is a very nice property of $d$-regular graphs satisfying the requirements of Theorem~\ref{th:main1}: for every $v$, there are not too many subgraphs on $v$ vertices with the maximum possible number of edges $\frac{dv}{2}-\left\lceil\frac{d+1}{2}\right\rceil$ (see Section~\ref{sc:properties}). 

{\it Let us finally say a few words about our proof of Theorem~\ref{th:main1}.} Kahn, Narayanan, and Park in~\cite{KNP} noted that the key observation enabling their proof that the threshold for the appearance of the second power of a Hamilton cycle $F$ equals $p_c(F)=\Theta(n^{-1/2})$ is that $\mathcal{F}_n$ is $(1+o(1))\sqrt{e/n}$-spread. They refined the concept of spreadness by incorporating the number of components in a subgraph. This refinement was later distilled by Espuny D\'{i}az and Person in~\cite{DP} under the name of {\it superspreadness} and was used to extend the result of Kahn, Narayanan, and Park to a broader class of spanning subgraphs in $G(n,p)$. Somewhat surprisingly, this property allows to prove the weaker result \cite[Corollary 1.5]{CHL} but does not allow to prove Theorem~\ref{th:main1}, as it treats contributions from connected components of both bounded and growing sizes in the same manner. It is also worth noting that Spiro \cite{Spiro} proposed another generalisation of spreadness, which, in particular, recovers the results from \cite{DP} and served as a key ingredient in the proofs of \cite{CHL}. 

%We apply the same fragmentation trick with a more careful analysis of subgraphs of growing sizes that allows us to prove Theorem~\ref{th:main1}. Since the proof does not introduce any substantially new ideas, we defer it to Appendix~\ref{sc:theorem_coarse_proof}.

In our approach, we employ the same fragmentation technique but with a more refined analysis of subgraphs of growing sizes, which enables us to prove Theorem~\ref{th:main1}. Since the proof does not introduce any substantially new ideas, we defer it to Appendix~\ref{sc:theorem_coarse_proof}.

% \MZ{Add another citation here?} A subsequent work \cite{CHL} further leveraged this concept to establish thresholds for an even wider family of spanning subgraphs. Rather than treating this notion as a black box,

% we apply the same fragmentation trick with a more careful analysis of subgraphs of growing size that allows us to prove Theorem~\ref{th:main1}. Since the proof does not introduce any new ideas, we defer it to Appendix~\ref{sc:theorem_coarse_proof}. %In particular, they answered a question of Frieze asked in~\cite{Frieze_C4} --- they showed that the threshold for appearance of spanning 2-overlapping $4$-cycles (i.e. the copies of $C_4$ are ordered cyclically, two consecutive $C_4$ overlap in exactly one edge, whereby each $C_4$ overlaps with two copies of $C_4$ in opposite edges) equals $\Theta(n^{-2/3})$. Clearly, Theorem~\ref{th:main2} implies that $p(n)=(e/n)^{2/3}$ is a sharp threshold for appearance of spanning 2-overlapping $4$-cycles.

\subsection{Organisation.} Properties of closed subgraphs in $d$-regular graphs $F$ are studied in Section~\ref{sc:properties}. We further use them in Section~\ref{sc:spread} to get tight estimates on the number of subgraphs and their extensions in $F$. The fragmentation trick and the main Lemma~\ref{lm:not_bad}, that allows to apply it, are described in Section~\ref{sc:planted}. Section~\ref{sc:theorem_sharp_proof} is devoted to the proof of Theorem~\ref{th:main2}. Then, in Section~\ref{sc:KNP_conjecture_resolution} we prove Theorem~\ref{th:second_power}. It generalises directly to Theorem~\ref{th:second_power_generalisation}. For the sake of clarity of presentation instead of providing a full proof of the more general result, we give a detailed proof of Theorem~\ref{th:second_power} (in Section~\ref{sc:KNP_conjecture_resolution}), which is more transparent and avoids unnecessary technical details, and then sketch the proof of Theorem~\ref{th:second_power_generalisation} in Section~\ref{sc:generalisation_square_sketch}. In Section~\ref{sc:challenges} we discuss some remaining challenges. The proof of Theorem~\ref{th:main1} is presented in Appendix~\ref{sc:theorem_coarse_proof}. %   Then we prove both theorems in Section~\ref{sc:theorems_proof}. Sections~\ref{sc:cl_proof}~and~\ref{sc:lm_proof} are devoted to the proof of Claim~\ref{cl:main} and the key lemma (Lemma~\ref{lm:main} from Section~\ref{sc:planted}) that validates the application of the planted trick respectively.

\subsection{Notation.} For every positive integer $n$, we denote $[n]:=\{1,\ldots,n\}$. The number of edges in the clique on $[n]$ is denoted by $N=N(n):={n\choose 2}$. We denote by $\partial_e(\tilde F)$ and $\partial_v(\tilde F)$ the edge and the vertex boundaries of a subgraph $\tilde F$ in a fixed graph $F$ (see definitions in Section~\ref{sc:intro_coarse}). For a graph $G$, we denote its minimum degree, its number of vertices, its number of edges, and its number of connected components by $\delta(G)$, $x(G)$, $\ell(G)$, and $c(G)$, respectively. We also denote the automorphism group of $G$ by $\mathrm{Aut}(G)$. For a set $U\subset V(G)$, we let $G[U]$ be the subgraph of $G$ induced by $U$. For a vertex $v\in V(G)$, its degree in $G$ is denoted by $\mathrm{deg}_G(v)$.

\section{Linearly many closed subgraphs}
\label{sc:properties}

In this section, we consider $d$-regular graphs with good enough expansion properties (edge boundaries have sizes at least $d+1$) and prove that they have a limited amount of closed subgraphs. Two main results of this section --- Claims~\ref{cl:count_copies_vertex} and~\ref{cl:total_number_closed} are used in Section~\ref{sc:spread} to prove Claim~\ref{cl:main}. The latter claim is essential in the proof of Theorem~\ref{th:main2}.

%\MZ{explain where we use the staff from this section --- I guess it is only needed in the first part of the second theorem}

Let us call a graph $F$ {\it locally sparse} if the edge boundary of every subgraph $\tilde F\subset F$ with $3\leq|V(\tilde F)|\leq n-3$ is of size at least $d+1$. Clearly $d+1$ can be replaced with $d+2$ for even $d$ since in this case $|\partial_e(\tilde F)|$ cannot be odd. Let $\Delta:=d+1$ for odd $d$ and $\Delta:=d+2$ for even $d$. It is easy to see that the condition $|\partial_e(\tilde F)|\geq\Delta$ holds for all $\tilde F$ with $2\leq |V(\tilde F)|\leq d-1$ just due to the $d$-regularity of $F$.
Let us call an {\it induced} subgraph $\tilde F$ of a locally sparse $d$-regular graph $F$ with the edge boundary of size exactly $\Delta$ {\it closed} (note that a closed subgraph is always connected --- otherwise, it has a connected component with a smaller boundary). % For $j<d$, let us call a vertex $w$ of a connected subgraph $\tilde F\subset F_n$ {\it $j$-free}, if its degree in $\tilde F$ equals $j$; $w$ is simply {\it free}, if it is $j$-free for some $j<d$.\\

Let $F$ be a locally sparse $d$-regular graph on $[n]$.

\begin{claim}
\label{cl:degrees}
Every closed subgraph of $F$ with at least 3 vertices has minimum degree at least~$d/2$.
\end{claim}

\begin{proof}
Assume that $\tilde F$ is a closed subgraph of $F$ with at least 3 vertices and with a vertex $w$ having degree $d'<d/2$. If we remove the vertex $w$ from $\tilde F$, then we get the graph $\tilde F\setminus w$ with edge boundary of size $|\partial_e(\tilde F)|+2d'-d<|\partial_e(\tilde F)|=\Delta$. This contradicts the local sparsity of $F$ when $|V(\tilde F)|\geq 4$. Otherwise it contradicts the fact that a subgraph on 2 vertices has the edge boundary of size at least $2d-2\geq\Delta$.
\end{proof}

 \begin{claim}
For any pair of adjacent vertices $x,y$ in $F$ and for every $3\leq v\leq n-3$, there are at most two closed subgraphs in $F$ on $v$ vertices containing $x$ and not containing $y$.
\label{cl:count_copies_vertex}
 \end{claim}
 
 \begin{proof}
 
Fix adjacent vertices $x,y$ and $3\leq v\leq n-3$.
 
A closed subgraph $\tilde F\subset F$ sends exactly $\Delta$ edges to $F\setminus\tilde F$ implying that $F\setminus \tilde F$ is also closed. Assume that $v\geq n/2$, and that there are at least 3 closed graphs on $v$ vertices that share $x$ and do not contain $y$. Then their complements are closed graphs on $n-v\leq n/2$ vertices that share $y$ and do not contain $x$. Therefore, it suffices to prove the claim for $v\leq n/2$.
 
Let $H_1,H_2$ be different closed subgraphs of $F$ on $v$ vertices that contain $x$ and do not contain $y$. Note that $H_1,H_2$ should have at least one other common vertex since otherwise the degree of $x$ is bigger than $d$ due to Claim~\ref{cl:degrees}. Then $|V(H_1)\cup V(H_2)|\leq n-2$.

Let $H_0=H_1\cap H_2$. Note that $|E(H_0)|\leq \frac{d}{2}|V(H_0)|-\frac{\Delta}{2}$ implying that $|E(H_j)\setminus E(H_0)|\geq \frac{d}{2}|V(H_j\setminus H_0)|$ for both $j=1$ and $j=2$ since $H_1,H_2$ are closed. On the other hand, if, say $|E(H_2)\setminus E(H_0)|>\frac{d}{2}|V(H_2\setminus H_0)|$, then $|E(H_1\cup H_2)|>\frac{d}{2}|V(H_1\cup H_2)|-\frac{\Delta}{2}$ which contradicts the local sparsity of $F$ since $|V(H_1\cup H_2)|\leq n-2$. Therefore, $|E(H_j)\setminus E(H_0)|=\frac{d}{2}|V(H_j\setminus H_0)|$ for both $j=1$ and $j=2$, but then $|E(H_0)|= \frac{d}{2}|V(H_0)|-\frac{\Delta}{2}$, i.e. $H_0$ is closed. 

Then, there are exactly $\Delta$ edges between $H_0$ and $F\setminus H_0$, and one of them is the edge between $x$ and $y$. It means that $H_j\setminus H_0$, $j\in\{1,2\}$, send at most $\Delta-1$ edges (in total) to $H_0$. This may happen only if $|V(H_j\setminus H_0)|=1$ for both $j=1$ and $j=2$. Indeed, $|V(H_1\setminus H_0)|=|V(H_2\setminus H_0)|$. Moreover, the number of edges that $H_j\setminus H_0$ sends to $H_0$ equals
 $$
  |E(H_j)\setminus E(H_0)|-|E(H_j\setminus H_0)|\geq \frac{d}{2}|V(H_j\setminus H_0)|-\left(\frac{d}{2}|V(H_j\setminus H_0)|-\frac{\Delta}{2}\right)=\frac{\Delta}{2}
 $$ 
 whenever $|V(H_j\setminus H_0)|\geq 2$.
 
Assume that there exists a closed graph $H_3\not\subset H_1\cup H_2$ on $v$ vertices that contains $x$ and does not contain $y$. From the above it follows that $|V(H_1)\cap V(H_3)|=|V(H_2)\cap V(H_3)|=v-1$. If $H_0\not\subset H_3$, then $H_3$ has to contain both vertices from $(V(H_1)\cup V(H_2))\setminus V(H_0)$. Therefore, there are at least two vertices in $V(H_0)\setminus V(H_3)$ and then $|V(H_1)\cap V(H_3)|\leq v-2$ --- a contradiction. We get $H_3\cap H_1=H_3\cap H_2=H_0$. Each vertex of $H_j\setminus H_0$, $j\in\{1,2,3\}$, sends at least $\frac{d}{2}$ edges to $H_0$ due to Claim~\ref{cl:degrees}. But then the vertices from $H_j\setminus H_0$ send at least $\frac{3d}{2}\geq\Delta$ edges to $H_0$ --- contradiction again, since there is one additional edge $\{x,y\}$ in the edge boundary of $H_0$. 

Therefore, any other closed graph that contains $x$ and does not contain $y$ should be entirely inside $H_1\cup H_2$. Assume that such a graph $H_3$ exists. Let $w_1\in H_1\setminus H_0$, $w_2\in H_2\setminus H_0$. Clearly, $H_3$ contains $w_1,w_2$ and all but one vertex of $H_0$.   In the same way as above we get that $H_1\cap H_2=H_0$, $H_1\cap H_3$ and $H_2\cap H_3$ are three closed graphs on $v-1$ vertices that contain $x$ and do not contain $y$. These three closed graphs on $v-1$ vertices have the property that none of them is inside the union of the other two --- this is only possible when $v-1=2$, i.e. $v=3$. The only possible closed graph on 3 vertices is a triangle. Moreover, a triangle is closed only when $d=4$. So, $H_1,H_2$ are triangles sharing an edge, but then $H_3$ adds another edge to the union $H_1\cup H_2$ implying that $H_1\cup H_2\cup H_3$ is a 4-clique. We get a contradiction with the local sparsity since the edge boundary of a 4-clique in a 4-regular graph is of size $4<\Delta=6$.  
 \end{proof}
 
From this, it immediately follows that, for every $v$, there are at most $2dn$ closed subgraphs on $v$ vertices in $F$ --- since $F$ is connected, every subgraph contains a vertex that is incident to an edge than does not belong to this subgraph. The following claim gives a slightly better bound.
 
\begin{claim}
Let $k\in\mathbb{N}$, and let $F':=F[[k]]$ be the induced subgraph of $F$ on $[k]$.
For every $3\leq v\leq n-3$, the number of closed subgraphs of $F'$ with $v$ vertices is at most $\frac{2dk}{3}$.
\label{cl:total_number_closed}
\end{claim}

\begin{proof}
Fix a vertex $w$ in $F'$ and let us bound the number --- denoted by $\mu(w)$ --- of closed subgraphs $\tilde F\subset F'$ on $v$ vertices containing $w$ such that $\mathrm{deg}_{\tilde F}w<d$. Due to Claim~\ref{cl:count_copies_vertex}, $\mu(w)\leq 2d$. On the other hand, Claim~\ref{cl:degrees} implies that every closed subgraph has vertex boundary of size at least 3. Letting $f$ to be the number of closed subgraphs in $F'$ on $v$ vertices, by double counting, we get that $3f\leq\sum_{w\in V(F')}\mu (w)\leq 2dk$ as needed.
\end{proof}

\section{Subgraphs and spread}
\label{sc:spread}

%We here follow the notations of Section~\ref{sc:planted}: $F\in\mathcal{F}_n$ is fixed, $\mathbb{F}\in\mathcal{F}_n$ is chosen uniformly at random, and $\pi_{\ell}=\Prob(|\mathbb{F}\cap F|=\ell)$.

In this section we estimate two quantities: 1) the number of subgraphs with given numbers of vertices, edges, and connected components in a fixed $d$-regular locally sparse graph $F$, 2) the number of ways to extend such a subgraph to an isomorphic copy of $F$. It is split into two subsections: In Section~\ref{sc:sub_sharp}, we prove Claim~\ref{cl:main}, that gives fairly tight estimations of both quantities. These estimations are used in the first fragmentation step in the proof of Theorem~\ref{th:main2}. Section~\ref{sc:sub_coarse} proves much more straightforward and coarse bounds that hold for {\it any} $d$-regular graph and that we use in later fragmentation steps as well as in the proof of Theorem~\ref{th:main1} in Appendix~\ref{sc:theorem_coarse_proof}.

Let $F$ be an arbitrary $d$-regular locally sparse graph on $[n]$. Let $\mathcal{F}_n$ be the set of all isomorphic copies of $F$ on $[n]$. Fix $\ell\in[dn/2]$, $c\in[\ell]$, and $x\in\left[\frac{2\ell}{d}+\frac{\Delta}{d}c,\ell+c\right]$. The following quantity will play a crucial role in our proofs: 
$$
\sigma:=\frac{d}{2}x-\left(\ell+\frac{\Delta}{2}c\right).
$$ 
We will call $\sigma$ the {\it excess} of a graph with $x$ vertices, $\ell$ edges, and $c$ connected components. It measures the edit distance to the closest union of $c$ closed subgraphs on $x$ vertices.
%Let $p(\ell,x,c)$ be the probability that the intersection of $\mathbf{F}$ with $F$ is a graph on $x$ vertices with $\ell$ edges and $c$ connected components (we think about graphs as about sets of their edges, so there are no isolated vertices in $|\mathbf{F}\cap F|$). 

\subsection{Sharp estimates}
\label{sc:sub_sharp}

This section is devoted to the proof of the following claim.

\begin{claim}
\label{cl:main}
There exist constants $A^1_{\alpha},A^2_{\alpha}>0$ such that the number of subgraphs in $F$ without isolated vertices with $x$ vertices, $\ell$ edges, and $c$ components is
\begin{equation}
\alpha(\ell,x,c)\leq{n\choose c}{x\choose c}e^{A^1_{\alpha} c+A^2_{\alpha}\sigma} \max_{o\leq (\Delta+2)\sigma}{x\choose o}.
\label{eq:count_subgraphs_general}
\end{equation}
Moreover, there exist constants $A^1_{\beta},A^2_{\beta}>0$ such that, given $H\subset F$ with $x$ vertices, $\ell$ edges, and $c$ components, the number of ways to extend $H$ to a graph from $\mathcal{F}_n$ is at most
\begin{equation}
 \beta(\ell,x,c)=(n-x+c)!e^{A^1_{\beta}c+A^2_{\beta}\sigma}\min\left\{\left(\frac{x}{c}\right)^c(d-1)^x/|\mathrm{Aut}(F)|,1\right\}\max_{o\leq (\Delta+2)\sigma}{x\choose o}.
\label{eq:embed_subgraphs_general}
\end{equation}
\end{claim}

\begin{proof}
We first prove~\eqref{eq:count_subgraphs_general}. There are at most ${x\choose c}{\sigma+c\choose c}$ ways to choose positive integers $\ell_1,\ldots,\ell_{c}$ and $x_1,\ldots,x_{c}$  such that 
\begin{itemize}
\item $\frac{2\ell_i}{d}+\frac{\Delta}{d}\leq x_i\leq\ell_i+1$ for all $i\in[c]$,
\item $\sum_{i=1}^{c}\ell_i=\ell$, and $\sum_{i=1}^{c} x_i=x$.
\end{itemize} 
For every $i\in[c]$, fix such $x_i$ and $\ell_i$, and set
\begin{equation}
 \sigma_i=\frac{d}{2}x_i-\left(\ell_i+\frac{\Delta}{2}\right).
\label{eq:sigma_i_definition}
\end{equation}

%Let us compute the number of ways to choose connected vertex-disjoint subgraphs $R_1,\ldots,R_c$ from $F$ with the respective numbers of edges ($\ell_i$) and vertices ($x_i$). 
 We first choose closed subgraphs of $F$ that correspond to $\sigma_i=0$ one by one. The number of ways to choose the first closed subgraph is at most $\frac{2d}{3}\cdot n$, due to Claim~\ref{cl:total_number_closed} (if $\ell_1=1$, then the number of choices is $\frac{dn}{2}<\frac{2dn}{3}$). In the same way, if some set of $\tilde n$ vertices is already included in the subgraph under construction, then the next closed component can be chosen in at most $\frac{2d}{3}(n-\tilde n)$ ways.

We now switch to not closed components. Note that the $i$-th component $z_i$ has vertex boundary of size $o_i\leq \Delta+2\sigma_i$. Indeed, let $\ell'_i$ be the number of edges in $E(F[V(z_i)])\setminus E(z_i)$, that is the number of missing edges in $z_i$. Then $dx_i=2\ell_i+2\ell'_i+|\partial_e(z_i)|$, implying $|\partial_e(z_i)|+2\ell'_i=\Delta+2\sigma_i$ due to~\eqref{eq:sigma_i_definition}. Since every missing edge in $z_i$ is incident to two boundary vertices, we get that $|\partial_v(z_i)|\leq|\partial_e(z_i)|+2\ell'_i$.  The desired bound on $o_i$ follows. Assume that $\tilde n$ vertices have been already included in the subgraph and we now describe the procedure of choosing the $i$-th component $z_i$:
\begin{enumerate}
\item choose the size of the vertex boundary $o_i\leq \Delta+2\sigma_i\leq (\Delta+2)\sigma_i\leq (\Delta+2)^{\sigma_i}$;
\item choose a set $\mathcal{O}\in{[x_i]\choose o_i}$ that identifies the labels of boundary vertices in the $i$-th component; 
\item choose a vertex $w$ out of the set of remaining $n-\tilde n$ vertices and {\it activate} it --- we treat this vertex as the minimum vertex in the component under construction;
\item at every step $j\geq 1$, {\it consider} the minimum vertex $v_j$ among active vertices:
\begin{itemize}
\item if $j\in\mathcal{O}$ (i.e. $v_j$ should be boundary), then add to the component some set of edges $E_j$ incident to $v_j$ (in at most $2^d$ ways), deactivate $v_j$, and {\it activate} all the vertices incident to edges of the set $E_j$ that have not been considered,
\item if $j\notin\mathcal{O}$, then add to the component all the edges incident to $v_j$, deactivate $v_j$, and {\it activate} all the neighbours of $v_j$ that have not been considered.
\end{itemize}
\end{enumerate}
The set of edges that have been added during this process forms the desired component $z_i$. So, the number of ways to choose the $i$-th component is at most $(n-\tilde n)(\Delta+2)^{\sigma_i}\max\limits_{o_i\leq (\Delta+2)\sigma_i}{x_i\choose o_i}2^{do_i}$.

Eventually we get that the number of {\it ordered} choices of components with parameters $\ell_i,x_i$, $i\in[c]$, in $F$ is at most 
$$
 c!{n\choose c}\left(\frac{2d}{3}\right)^{c}\prod_{i=1}^{c}(\Delta+2)^{\sigma_i}\max\limits_{o_i\leq (\Delta+2)\sigma_i}{x_i\choose o_i}2^{do_i} \leq c!{n\choose c} d^{c} 2^{2(\Delta+1)\sigma} \max_{o\leq (\Delta+2)\sigma}{x\choose o}.
$$
Note that this bound does not depend on the order of the choice of components, thus
\begin{align*}
 \alpha(\ell,x,c) & \leq{n\choose c}{x\choose c}{\sigma+c\choose c}d^{c}\cdot 2^{2(\Delta+1)\sigma} \max_{o\leq (\Delta+2)\sigma}{x\choose o}\\
 &\leq{n\choose c}{x\choose c}(2d)^c \cdot 2^{(2\Delta+3)\sigma} \max_{o\leq (\Delta+2)\sigma}{x\choose o}
\end{align*}
as needed.

Let us now fix $H\subset F$ with $x$ vertices, $\ell$ edges, and $c$ components. Let us bound the number of ways to extend $H$ to an $F'\in\mathcal{F}_n$. We construct such an extension in the following way.

First of all, we add to $H$ all the isolated vertices from $[n]$ that it misses, and then we  forget the labels of all the $n$ vertices of $H$. Fix some $F'\in\mathcal{F}_n$ such that each vertex (but the first one) has a smaller neighbour (in the linear order on $[n]$).  We will compute the number of ways to embed the unlabelled $H$ into $F'$, i.e. the number $\mathrm{mon}(H\to F')$ of monomorphisms $H\to F'$. Clearly, the desired number of extensions is exactly $\frac{\mathrm{mon}(H\to F')}{|\mathrm{Aut}(F')|}$. Let $\mathcal{Z}$ be the set of all $n-x+c$  connected components of $H$. We should compute the number of ways to embed the elements of $\mathcal{Z}$ in $F'$ disjointly. 

Let $z_1,\ldots,z_{n-x+c}$ be an arbitrary ordering of $\mathcal{Z}$ (there are $(n-x+c)!$ ways to order the elements of $\mathcal{Z}$). We embed sequentially each $z_i$ in $F'$ in a way such that all vertices of $z_i$ are bigger than all the $i-1$ minimum vertices of $z_1,\ldots,z_{i-1}$.  At every step $i=1,\ldots,n-x+c$, consider the minimum vertex $\kappa_i$ of $F'$ such that none of the embedded elements of $\mathcal{Z}$ in $F'$ contain this vertex. If $z_i$ is a single vertex, then we assign $\kappa_i$ with $z_i$ and proceed to the next step. Otherwise, we let $\kappa_i$ be the minimum vertex of the embedding of $z_i$ and, then, distinguish between the following cases. %We let $z_i=R_1$ without loss of generality and for the sake of simplicity of notations. 

First, we assume that $z_i$ is closed. If $|V(z_i)|=2$, then there are at most $d$ ways to choose this edge and at most two ways to place it (two rotations). Thus there are at most $2d$ ways to embed $z_i$. If $|V(z_i)|\geq 3$, then there are at most $2d|\mathrm{Aut}(z_i)|$ ways to choose a (labelled) copy of $z_i$ in $F'$ with the minimum vertex $\kappa_i$, due to Claim~\ref{cl:count_copies_vertex}. Indeed, there are at most $2d$ ways to choose a subgraph in $F'$ that is isomorphic to $z_i$: first, choose an edge that is adjacent to $\kappa_i$, and then choose a closed subgraph on $|V(z_i)|$ vertices that does not contain the selected edge. Moreover, 
$$
|\mathrm{Aut}(z_i)|\leq|V(z_i)|(d-1)^{|V(z_i)|},
$$
due to~\cite[Theorem 2]{KLT}. 

Second, let $z_i$ be not closed with $|\partial_v z_i|=o_i$. Choose a set $\mathcal{O}$ from ${[|V(z_i)|]\choose o_i}$ that will identify the labels of boundary vertices in the embedding of $z_i$ into $F'$. {\it Activate} $v_1:=\kappa_i$. At every step $j\geq 1$, {\it consider} the minimum active vertex $v_j$ in $F'$ and
\begin{itemize}
\item if $j\in\mathcal{O}$, then add to the image of $z_i$ under construction some set of edges $E_j$ incident to $v_j$ (in at most $2^d$ ways), deactivate $v_j$, and {\it activate} all the vertices incident to edges of the set $E_j$ that have not been considered,
\item if $j\notin\mathcal{O}$, then add to the image of $z_i$ all the edges incident to $v_j$, deactivate $v_j$, and {\it activate} all the neighbours of $v_j$ that have not been considered.
\end{itemize}
The image is constructed. However, we have not yet mapped the vertices of $z_i$ to the vertices of the image. The number of such mappings $\rho_i$ is exactly the number of automorphisms of $z_i$. As above, it is bounded by $|V(z_i)|(d-1)^{|V(z_i)|}$.

We also notice that every automorphism of every $z_i$ respects the property of a vertex to be boundary. Moreover, any automorphism that preserves boundary vertices of $z_i$ extends trivially to an automorphism of the entire $F'$. Thus, $\prod_{i=1}^c|\mathrm{Aut}(z_i)|\leq |\mathrm{Aut}(F')|\prod_{i=1}^c d^{o_i}$, where $d^{o_i}$ is the upper bound on the number of automorphisms of $z_i$ that preserve all non-boundary vertices.

We conclude that there are at most 
\begin{multline*}
(n-x+c)!(2d)^{c}\min\left\{\left(\frac{x}{c}\right)^c(d-1)^x\prod_{i=1}^{c}{x_i\choose o_i}2^{do_i},|\mathrm{Aut}(F')|\prod_{i=1}^c {x_i\choose o_i}(d2^d)^{o_i}\right\}\\
\leq
(n-x+c)!(2d)^{c}\min\left\{\left(\frac{x}{c}\right)^c(d-1)^x2^{d(2+\Delta)\sigma},|\mathrm{Aut}(F')|(d2^d)^{(2+\Delta)\sigma}\right\}\max_{o\leq (2+\Delta)\sigma}{x\choose o}
\end{multline*}
ways to embed $H$ into $F'$. It remains to divide the final bound by $|\mathrm{Aut}(F')|$.
\end{proof}

\subsection{Coarse estimates}
\label{sc:sub_coarse}

We will also use the following coarse version of the first part of Claim~\ref{cl:main}. For a graph $H$ with maximum degree at most $d$, let us denote by $\mathcal{J}^H_{\ell,x,c}$ the set of all subgraphs $J\subset H$ with $\ell$ edges, $x$ non-isolated vertices, and $c$ connected components (excluding isolated vertices).
\begin{claim}
\label{cl:J_upper_bound}
For every graph $H$ with maximum degree at most $d$, 
%\begin{equation}
$$
|\mathcal{J}^H_{\ell,x,c}|\leq {\min\{|E(H)|,|V(H)|\}\choose c} (16d)^{\ell}.
$$
%\label{eq:J_upper_bound_trivial}
%\end{equation}
\end{claim}
\begin{proof}
We first choose $c$ edges (or vertices, if $|V(H)|<|E(H)|$) in $H$ that belong to different components of $J$. Then, the number of ways to assign to each of the $c$ components its number of edges is at most ${\ell+c-1\choose c-1}\leq 2^{2\ell}$ since $\ell\geq c$. Finally, as soon as, for every $i\in[c]$, a vertex in the $i$-th component of size $\ell_i$ is fixed, the number of ways to explore it sequentially, edge by edge, is at most $(4d)^{\ell_i}$: at every step, first decide whether the current edge is incident to the vertex considered at the previous step or to the next explored but unconsidered vertex --- two choices (the order of vertices in $H$ is arbitrary and fixed in advance). If the second choice is made, consider the minimum vertex in the set of explored unconsidered vertices. Second, choose the new edge incident to the considered vertex in at most $d$ ways. Finally, if a new vertex has been explored, decide whether it has additional edges in $J$ or not (two choices). If it does not have additional edges, remove it from the pool of unconsidered vertices. It completes the proof.% of~\eqref{eq:J_upper_bound_trivial}.
\end{proof}

Moreover, for any $d$-regular graph $F$ on $[n]$ and any $H\subset F$ with $\ell$ edges, $x$ non-isolated vertices, and $c$ connected components (excluding isolated vertices), by counting $\mathrm{mon}(H\to F)$ using the same strategy as in the proof of the second part of Claim~\ref{cl:main}, we get 
\begin{claim}
\label{cl:simple_embeddings}
The number of ways to extend $H$ to an isomorphic copy of $F$ on $[n]$ is at most $e^{O(\ell)}(n-x+c)!/|\mathrm{Aut}(F)|$.
\end{claim}
\begin{proof} 
As in the proof of Claim~\ref{cl:main}, we let $\mathcal{Z}$ be the set of all $n-x+c$  connected components of $H$. We order the elements of $\mathcal{Z}$ and then, at every step $i\in\{1,\ldots,n-x+c\}$, we consider the minimum vertex $\kappa_i$ of $F$ such that none of the embedded elements of $\mathcal{Z}$ contain this vertex and map some vertex of $z_i$ to $\kappa_i$. After that, we embed $z_i$ edge by edge in an arbitrary order that respects connectivity. We then get at most $(n-x+c)!\left(\frac{x}{c}\right)^c d^{\ell}=e^{O(\ell)}(n-x+c)!$ embeddings~as %since
\begin{equation}
\left(\frac{x}{c}\right)^c=e^{c\ln (x/c)}=e^{x\cdot\frac{\ln(x/c)}{x/c}}\leq e^x.
\label{eq:x/c}
\end{equation}
\end{proof}

\section{Fragmentation: arbitrary families}
\label{sc:planted}

This section presents our main tool ---  fragmentation trick via typical fragments. We, first, use it in the proof of Theorem~\ref{th:main2} in Section~\ref{sc:theorem_sharp_proof}, and then we use this tool to show the existence of rare fragments without closed subgraphs of size $\Omega(\log n)$ in our proof of Theorem~\ref{th:second_power} in Section~\ref{sc:KNP_conjecture_resolution}. Although the trick appeared in several papers (in particular, in~\cite{KNP}, see also Section~\ref{sc:intro_strategy}), we describe it here in full detail for the sake of completeness. Here we present it in a slightly different form, without explicitly distinguishing between pathological and non-pathological pairs, which is more transparent for us.

Let $\delta_n=o(1)$ be a slowly decreasing function. Let %Let $d\geq 3$ be a constant, {\bf[Check the conditions below]}
$$
p=p(n)=\Omega(n^{-2/3}), \quad \quad
m=m(n)=pN,\quad  \quad
f=f(n)=O(n).
$$
Let $\mathcal{F}_n$ be an arbitrary multiset of graphs on $[n]$, each graph has $f$ edges. 
Let $\mathcal{B}$ be an arbitrary graph property (not necessarily isomorphism-closed). For $F\in\mathcal{F}_n$ and $W\subset{[n]\choose 2}$, let $\mathcal{M}(F,W)$ be the multiset of all $F'\in\mathcal{F}_n$ such that $F'\subset F\cup W$. Let $\mathcal{M}_{\mathcal{B}}(F,W)$ be the multiset of all $F'\in\mathcal{M}(F,W)$ such that $F'\cap F\in\mathcal{B}$. Let us say that the pair $(F,W)$ is {\it $\mathcal{B}$-bad}, if $|\mathcal{M}_{\mathcal{B}}(F,W)|>\delta_n|\mathcal{M}(F,W)|$. Let $\mathbf{F}$ be a uniformly random element of $\mathcal{F}_n$ sampled independently of a uniformly random $\mathbf{W}\in{{[n]\choose 2}\choose m}$. We shall prove sufficient conditions for $(\mathbf{F},\mathbf{W})$ not being $\mathcal{B}$-bad whp. That would mean that whp we may replace most of $F\in\mathcal{F}_n$ with fragments $F\cap F'$ that do not have the property $\mathcal{B}$. Let $\mathcal{B}_{\ell}$ contain all graphs with the property $\mathcal{B}$ and $\ell$ edges. % of sizes at most $\ell_0$.
For $F\in\mathcal{F}_n$, let  
\begin{equation}
\label{eq:Pi_def}
\Pi^F_{\mathcal{B}}:=\Prob(F\cap \mathbf{F}\in\mathcal{B}).
\end{equation}
\begin{lemma}
\label{lm:not_bad}
If
\begin{equation}
 \max_{F\in\mathcal{F}_n}\sum_{\ell=0}^f\Pi^F_{\mathcal{B}_{\ell}}\left(\left(1+\frac{3f}{m}\right)\frac{N}{m}\right)^{\ell}e^{-f^2/m+f^3/(3m^2)}\leq\delta_n^3,
\label{eq:main_lm}
\end{equation}
then $(\mathbf{F},\mathbf{W})$ is not $\mathcal{B}$-bad with probability at least $1-2\delta_n$.
\end{lemma}

\begin{proof}
For $t\in\{0,1,\ldots,f\}$, let
$$ 
 M(t):=|\mathcal{F}_n|{N-f\choose m-t}/{N\choose m+f-t}
$$
be the expected number of $F\in\mathcal{F}_n$ such that $F$ belongs to a uniformly random subset of ${[n]\choose 2}$ of size $m+f-t$. 

Let $t\in\{0,1,\ldots,f\}$. Each pair $\left\{F\in\mathcal{F}_n, \, W\in{{n\choose 2}\choose m}\right\}$ such that $|\mathcal{M}(F,W)|<\delta_n M(t)$ and $|F\cap W|=t$  can be obtained by choosing a set $A\subset{[n]\choose 2}$ of size $m+f-t$, on the role of $F\cup W$ (in at most ${N\choose m+f-t}$ ways), choosing an $F\subset A$, $F\in\mathcal{F}_n$ (in less than $\delta_n M(t)$ ways), and choosing the intersection $F\cap W$ (in ${f\choose t}$ ways).
We conclude that
\begin{align*}
 \mathbb{P}\left(|\mathcal{M}(\mathbf{F},\mathbf{W})|<\delta_n M(t)\mid|\mathbf{F}\cap\mathbf{W}|=t\right)
 \leq\frac{\delta_n M(t){N\choose m+f-t}{f\choose t}}{|\mathcal{F}_n|{N-f\choose m-t}{f\choose t}}=\delta_n.
\end{align*}
Then,
$$
 \Prob((\mathbf{F},\mathbf{W})\text{ is $\mathcal{B}$-bad},\, |\mathcal{M}(\mathbf{F},\mathbf{W})|<\delta_nM(t)
 \mid|\mathbf{F}\cap\mathbf{W}|=t)\leq\delta_n=o(1).
$$ 
Thus, it is sufficient to prove that, % uniformly over $t$,
$$
 \Prob((\mathbf{F},\mathbf{W})\text{ is $\mathcal{B}$-bad},\, |\mathcal{M}(\mathbf{F},\mathbf{W})|\geq\delta_nM(t)
 \mid|\mathbf{F}\cap\mathbf{W}|=t)\leq\delta_n.
$$
The latter probability is at most
\begin{equation}
\Prob\left(\left|\mathcal{M}_{\mathcal{B}}(\mathbf{F},\mathbf{W})\right|> \delta_n^2M(t)
 \mid|\mathbf{F}\cap\mathbf{W}|=t\right)
 \leq\frac{\mathbb{E}(X \mid| \mathbf{F}\cap\mathbf{W}|=t)}{\delta_n^2M(t)},
\label{eq:bad_to_X_day1}
\end{equation}
where $X=X(\mathbf{F},\mathbf{W})$ counts the number of  $F'\in\mathcal{M}(\mathbf{F},\mathbf{W})$ such that $F'\cap\mathbf{F}\in\mathcal{B}$. Fix $F\in\mathcal{F}_n$, let $\mathbf{W}'_F=\mathbf{W}'_F(t)$ be a uniformly random $(m-t)$-subset of ${[n]\choose 2}\setminus F$, and let $X'_F=X'_F(t)$ be the number of $F'\in\mathcal{M}(F,\mathbf{W}'_F)$ such that $F'\cap F\in \mathcal{B}$. We get
\begin{align*}
 \mathbb{E}(X \mid|\mathbf{F}\cap\mathbf{W}|=t) & =\sum_{F\in\mathcal{F}_n}
 \mathbb{E}(X\cdot \1_{\mathbf{F}=F}\mid |\mathbf{F}\cap\mathbf{W}|=t)
=\sum_{F\in\mathcal{F}_n}\frac{
 \mathbb{E}\left(X\cdot \1_{\mathbf{F}=F,\,|F\cap\mathbf{W}|=t}\right)}{\mathbb{P}(|\mathbf{F}\cap\mathbf{W}|=t)}.
 \end{align*}
Since all $F\in\mathcal{F}_n$ have the same number of edges, $\mathbb{P}(|F\cap\mathbf{W}|=t)$ does not depend on $F\in\mathcal{F}_n$. In particular $\mathbb{P}(|\mathbf{F}\cap\mathbf{W}|=t)=\mathbb{P}(|F\cap\mathbf{W}|=t)$ for every $F\in\mathcal{F}_n$. Therefore,
$$
 \mathbb{E}(X \mid|\mathbf{F}\cap\mathbf{W}|=t) 
 =\sum_{F\in\mathcal{F}_n}
 \mathbb{E}(X\cdot \1_{\mathbf{F}=F}\mid|F\cap\mathbf{W}|=t).
$$ 
Recall that $X(F,W)$ is the number of  $F'\in\mathcal{M}(F,W)$ such that $F'\cap F\in\mathcal{B}$. We get
\begin{align}
 \mathbb{E}(X \mid|\mathbf{F}\cap\mathbf{W}|=t) &=\sum_{F\in\mathcal{F}_n}
 \mathbb{E}(X(F,\mathbf{W})\cdot \1_{\mathbf{F}=F}\mid|F\cap\mathbf{W}|=t)\notag\\
 &=%\sum_{F\in\mathcal{F}_n}\mathbb{E}(X'_F\cdot \1_{\mathbf{F}=F})=
 \sum_{F\in\mathcal{F}_n}\mathbb{E}(X(F,\mathbf{W})\mid|F\cap\mathbf{W}|=t)\cdot\mathbb{P}(\mathbf{F}=F)\\
 &=
 \sum_{F\in\mathcal{F}_n}\mathbb{E}X'_F\cdot\mathbb{P}(\mathbf{F}=F).
\label{eq:X_to_X'_day1}
\end{align}
Then it remains to prove that $\frac{\E X'_F}{M(t)}\leq\delta_n^3$ for all $F\in\mathcal{F}_n$ and all $t$. By the definition of $M(t)$,
\begin{align*}
\frac{\E X'_F}{M(t)}  
&=\sum_{\ell}|\mathcal{F}_n|\frac{\Pi^F_{\mathcal{B_{\ell}}}}{M(t)}
 {m-t\choose f-\ell}/{N-f\choose f-\ell}\\
 &=\sum_{\ell}\Pi^F_{\mathcal{B_{\ell}}}\frac{{m-t\choose f-\ell}/{N-f\choose f-\ell}}{{N-f\choose m-t}/{N\choose m+f-t}}\\
 &=
 \sum_{\ell}\Pi^F_{\mathcal{B_{\ell}}}\frac{{m-t\choose f-\ell}{N\choose m+f-t}}{{N-f\choose m-t}{N-f\choose f-\ell}}.
\end{align*} 
Applying Stirling's approximation, we get
 \begin{align*}
\frac{\E X'_F}{M(t)} 
 &\sim\sum_{\ell}\Pi^F_{\mathcal{B_{\ell}}}
\frac{(m-t)^{2(m-t)}(N-2f+\ell)^{N-2f+\ell}N^N}{(m-t-f+\ell)^{m-t-f+\ell}(m+f-t)^{m+f-t}(N-f)^{2N-2f}}\\
&=\sum_{\ell}\Pi^F_{\mathcal{B_{\ell}}}\left(\frac{N-2f+\ell}{m-t-f+\ell}\right)^{\ell}
\frac{\left(1+\frac{f-\ell}{m-t-f+\ell}\right)^{m-t-f}\left(1-\frac{f}{m-t+f}\right)^{m-t+f}}
{\left(1-\frac{f}{N}\right)^{N}\left(1+\frac{f-\ell}{N-2f+\ell}\right)^{N-2f}}.
\end{align*}
Therefore, using Taylor expansion, we get
 \begin{align*}
\frac{\E X'_F}{M(t)}
&\leq\sum_{\ell}\Pi^F_{\mathcal{B_{\ell}}}\left(\frac{N}{m}\left(1+\frac{t+f-\ell}{m-t-f+\ell}\right)\right)^{\ell}\\
&\quad\quad\quad\quad\quad\quad\quad\times
e^{f-\ell-\frac{(f-\ell)^2}{2m}+\frac{(f-\ell)^2\ell}{2(m-t-f+\ell)^2}+\frac{(f-\ell)^3}{3(m-t-f+\ell)^2}-f-\frac{f^2}{2m}+f-f+\ell+O(1)}\\
&\leq\sum_{\ell}\Pi^F_{\mathcal{B_{\ell}}}\left(\left(1+\frac{3f}{m}\right)\frac{N}{m}\right)^{\ell}e^{-f^2/m+f^3/(3m^2)+O(1)}\leq\delta_n^3=o(1),
\end{align*}
completing the proof.
\end{proof}

\section{Sharp thresholds}
\label{sc:theorem_sharp_proof}

We prove the two parts of Theorem~\ref{th:main2} separately in the next two sections.

\subsection{Better expansion of large subgraphs}
\label{sc:sharp1}

Here, we prove the following.

\begin{theorem}
Let $d\geq 3$, $\delta\in(0,1/d)$, and let $F=F(n)$ be a sequence of $d$-regular graphs on $[n]$, $n\in\mathbb{N}$, such that
\begin{itemize}
\item the number of automorphisms of $F$ is at most $e^{o(n^{1-\delta})}$;
\item for some $w=w(n)=\omega(1)$, every $\tilde F\subset F$ with $3\leq|V(\tilde F)|\leq n^{1-\delta}$ has $|\partial_e(\tilde F)|\geq d+1+\lfloor |V(\tilde F)|w/\log n\rfloor$ and every $\tilde F\subset F$ with $n^{1-\delta}<|V(\tilde F)|\leq n-3$ has $|\partial_e(\tilde F)|\geq d+1$.
\end{itemize}
Let $\varepsilon>0.$ If $p>(1+\varepsilon)\left(\frac{e}{n}\right)^{2/d}$, then, assuming that $dn$ is even, whp  $\mathbf{G}\sim G(n,p)$ contains an isomorphic copy of $F$.
\label{th:main2_1}
\end{theorem}

We note that Theorem~\ref{th:main2_1} indeed implies the first part of Theorem~\ref{th:main2} since $p^e_F\geq(1-o(1))\left(\frac{e}{n}\right)^{2/d}$. Indeed, due to the first requirement, $|\mathrm{Aut}(F)|=e^{o(n)}$, implying 
\begin{equation}
\mathbb{E}X_F=\frac{n^n}{e^{n(1+o(1))}}p^{dn/2}=\Theta(1)\quad
\Rightarrow\quad p=(1-o(1))\left(\frac{e}{n}\right)^{2/d}.
\label{eq:from_part_Th2}
\end{equation} %\MZ{Define $\mathbb{E}X_F$ somewhere}

The rest of the section is devoted to the proof of Theorem~\ref{th:main2_1}. Fix a $d$-regular graph $F$ satisfying the requirements of the theorem.

%\MZ{correct the proof after weakening the condition on the edge boundary}

\subsubsection{Fragmentation: $1/\delta$ iterations}

%Without loss of generality, we may assume that $\lceil 1/\delta\rceil>2d$ and that the number of automorphisms of $F_n$ is $e^{o(n^{1-\delta})}$.

Consider $\lfloor 1/\delta\rfloor$ independent samples 
$$
\mathbf{W}_1\sim G(n,m),\quad m=\left\lfloor(1+\varepsilon)\left(\frac{e}{n}\right)^{2/d}N\right\rfloor,\quad\text{ and }
$$
$$
\mathbf{W}_2,\ldots,\mathbf{W}_{\lfloor 1/\delta\rfloor}\sim G(n,m_0),\quad m_0=\left\lfloor\varepsilon\cdot n^{-2/d}\cdot N\right\rfloor.
$$
Recall that we denote by $\mathcal{F}_n$ the family of all isomorphic copies of $F$ on $[n]$.

\begin{claim}
\label{cl:sharp_day_1}
Whp there exists $F'\subset F\cup\mathbf{W}_1$, $F'\in\mathcal{F}_n$, such that $|F\cap F'|\leq n^{1-\delta}$.
\end{claim}

\begin{proof}
Let $f=\frac{dn}{2}$, and $\mathcal{B}$ be the property of graphs on $n$ vertices to have more than $n^{1-\delta}$ edges. Due to Lemma~\ref{lm:not_bad}, it is sufficient to prove~\eqref{eq:main_lm}. More precisely, due to symmetry, it is sufficient to prove
\begin{equation}
 \sum_{\ell=0}^f\Pi^F_{\mathcal{B}_{\ell}}\left(\left(1+\frac{3f}{m}\right)\frac{N}{m}\right)^{\ell}e^{-f^2/m+f^3/(3m^2)}=o(1),
\label{eq:sufficient_first_fragmentation}
\end{equation}
for the fixed $F\in\mathcal{F}_n$. Fix non-negative integers $c$ and $x$. As in Claim~\ref{cl:J_upper_bound}, we denote by $\mathcal{J}_{\ell,x,c}$ the set of all subgraphs $J\subset F$ with $\ell$ edges, $x$ non-isolated vertices, and $c$ connected components (excluding isolated vertices), and let $p(\ell,x,c):=\mathbb{P}(\mathbf{F}\cap F\in\mathcal{J}_{\ell,x,c})$. Due to Claim~\ref{cl:main},
$$
p(\ell,x,c)\leq\frac{{n\choose c}{x\choose c}e^{A_1 c+A_2\sigma} \max_{o\leq (\Delta+2)\sigma}{x\choose o}^2(n-x+c)!}{|\mathcal{F}_n|},
$$
for some constants $A_1,A_2>0$. Let $\ell>n^{1-\delta}$. Recalling that $x=\frac{2}{d}\ell+\frac{c\Delta}{d}+\frac{2}{d}\sigma$, % in the same way as in~\eqref{eq:factorials_fraction},
 we get
\begin{align}
 \frac{(n-x+c)!}{|\mathcal{F}_n|}  \leq
 e^{o(n^{1-\delta})}\cdot\frac{(n-\frac{2\ell+c(\Delta-d)+2\sigma}{d})!}{n!} \leq e^{o(n^{1-\delta})}\cdot\frac{e^{(x-c)^2/n}}{n^{(2\ell+c+2\sigma)/d}}.
 \label{eq:factorials_fraction_sharp}
\end{align}

We first choose $0<\varepsilon'\ll\varepsilon''$ small enough. From the bound ${n\choose c}\leq\left(\frac{en}{c}\right)^c$, we get % and the equality $x-c=\ell/2+c/2+\sigma/2$, we get
\begin{equation}
 p(\ell,x,c)\leq
 e^{o(n^{1-\delta})}\cdot\frac{\left(\frac{e^{A_1+1} n^{1-1/d}}{c}\right)^c{x\choose c}e^{A_2\sigma}{x\choose o}^2 e^{(x-c)^2/n}}{n^{(2\ell+2\sigma)/d}},
\label{eq:lm_proof_general_bound}
\end{equation}
where $o=o(x)\leq 8\sigma$ is chosen in such a way that ${x\choose o}$ achieves its maximum. We get that $
 \left(\frac{e^{A_1+1}n^{1-1/d}}{c}\right)^c \leq e^{A_1n^{1-1/d}}\leq e^{o(\ell)}.$

%Let us recall that $\ell>\ell_0$, and thus $x>\ell_0/2$. Here $\ell_0=\lfloor C\sqrt{n}\rfloor$, and $C\gg\max\{\frac{1}{\varepsilon'},A_1/\delta\}$ can be chosen arbitrarily large.

We further consider separately several different cases. 

{\bf 1.} If $\sigma>\varepsilon' x$, then 
$$
{x\choose c}{x\choose o}^2e^{A_2\sigma}e^{(x-c)^2/n}< 8^xe^{x}e^{A_2\sigma}<n^{2\sigma/d}.
$$
Therefore,
\begin{equation}
 p(\ell,x,c)\leq \left((1+\varepsilon)/n\right)^{2\ell/d}.
\label{eq:small_p_ell_x_c}
\end{equation}

{\bf 2.} Let $\sigma\leq\varepsilon' x$.

{\bf 2.1.} If $c<\varepsilon' x$ and $x<\varepsilon' n$, then
$$
  {x\choose c}{x\choose o}^2e^{A_2\sigma}e^{(x-c)^2/n}<
 e^{(\varepsilon''/d)\ell}
$$
implying (\ref{eq:small_p_ell_x_c}) as well. 

{\bf 2.2.} If $c<\varepsilon' x$ and $x\geq \varepsilon' n$, then $e^{A_2\sigma} {x\choose c}{x\choose o}^2<e^{(\varepsilon''/d)\ell}.$
Thus, (\ref{eq:lm_proof_general_bound}) implies
$$
 p(\ell,x,c)\leq \frac{e^{(x-c)^2/n}}{n^{2\ell/d}}e^{(\varepsilon''/d+o(1))\ell}\leq (e^{1+\varepsilon''}/n)^{2\ell/d},
$$
since $2\ell/d=x-c\Delta/d-2\sigma/d\geq x(1-5\varepsilon'/2)$. 

{\bf 2.3.} Finally, let $c\geq\varepsilon' x$. Since $x\geq\frac{2}{d}\ell$, we get that $x\gg n^{1-1/d}$. In particular, 
$$
\left(\frac{e^{A_1+1} n^{1-1/d}}{c}\right)^c\leq
e^{e^{A_1} n^{1-1/d}}\leq e^{-10x}.
$$
Since $e^{A_2\sigma}\leq n^{2\sigma/d}$, we get
$$
 p(\ell,x,c)\leq e^{-10x+o(\ell)}\frac{{x\choose c}{x\choose o}^2 e^{(x-c)^2/n}}{n^{2\ell/d}}
\leq e^{-10x+x+o(\ell)}8^x n^{-2\ell/d}\leq n^{-2\ell/d}.
$$
Summing up,
\begin{align*}
\sum_{\ell}\Pi^F_{\mathcal{B}_{\ell}}\left(\left(1+\frac{3f}{m}\right)\frac{N}{m}\right)^{\ell}e^{-f^2/m+f^3/(3m^2)}& \leq
\sum_{\ell>n^{1-\delta}}\sum_{x,c}p(\ell,x,c)\left(\frac{1+o(1)}{1+\varepsilon}\cdot\left(\frac{n}{e}\right)^{2/d}\right)^{\ell}\\
&\leq n^2\sum_{\ell>n^{1-\delta}}\left(\frac{e^{2\varepsilon''/d}+o(1)}{1+\varepsilon}\right)^{\ell}=o\left(\frac{1}{n^2}\right),
\end{align*}
completing the proof of the claim.
\end{proof}

For every $F\in\mathcal{F}_n$ (with a slight abuse of notation, for the sake of simplicity of presentation, we now denote by $F$ an arbitrary graph from $\mathcal{F}_n$) such that there exists $F'$ as in the statement of Claim~\ref{cl:sharp_day_1}, we choose one such $F'$ and put $F\cap F'$ into a multiset $\mathcal{F}^{(1)}_n$. Due to Claim~\ref{cl:sharp_day_1} and Markov's inequality, we get that whp $|\mathcal{F}^{(1)}_n|=(1-o(1))|\mathcal{F}_n|$. We then proceed by induction. Assume that, for $i\in\left[\lceil 1/\delta\rceil-1\right]$, we have a multiset $\mathcal{F}^{(i)}_n$ comprising a single $H\subset F$ from almost every $F\in\mathcal{F}_n$ such that the graph $H\cup\mathbf{W}_1\cup\ldots\cup\mathbf{W}_i$ contains some $F'\in\mathcal{F}_n$ and $|H|=\lfloor n^{1-i\delta}\rfloor$. Let $\mathbf{H}$ be a uniformly random element of $\mathcal{F}^{(i)}_n$.
\begin{claim}
Whp there exists $H'\subset \mathbf{H}\cup\mathbf{W}_{i+1}$, where $H'\in\mathcal{F}^{(i)}_n$, such that 
$|\mathbf{H}\cap H'|\leq \max\left\{n^{1-(i+1)\delta},\ln n\right\}$.
\end{claim}

\begin{proof}
Let $f=\lfloor n^{1-i\delta}\rfloor$, $m=m_0$, and $\mathcal{B}$ be the property of graphs on $n$ vertices to have more than $\max\{n^{1-(i+1)\delta},\ln n\}$ edges. Due to Lemma~\ref{lm:not_bad}, it is sufficient to prove~\eqref{eq:main_lm}.

Fix non-negative integers $c$ and $x$. Recall that, for $H\in \mathcal{F}^{(i)}_n$, we denote by $\mathcal{J}^H_{\ell,x,c}$ the set of all subgraphs $J\subset H$ with $\ell$ edges, $x$ non-isolated vertices, and $c$ connected components (excluding isolated vertices), and denote $p^H(\ell,x,c):=\mathbb{P}(\mathbf{H}\cap H\in\mathcal{J}^H_{\ell,x,c})$. 

Due to Claim~\ref{cl:J_upper_bound} and Claim~\ref{cl:simple_embeddings}, for some constant $A>0$ that does not depend on $\varepsilon$,
\begin{align*}
p^H(\ell,x,c) \leq\frac{{\lfloor n^{1-i\delta}\rfloor\choose c}e^{A\ell}(n-x+c)!/|\mathrm{Aut}(F)|}{|\mathcal{F}^{(i)}_n|}
=(1+o(1))\frac{{\lfloor n^{1-i\delta}\rfloor\choose c}e^{A\ell}(n-x+c)!}{n!}.
\end{align*}
Due to~\eqref{eq:factorials_fraction_sharp} and since $x\leq|V(H)|=o(n)$,
$$
 p^H(\ell,x,c)\leq \frac{\left(\frac{e n^{1-i\delta}}{c}\right)^c e^{(A+o(1))\ell}}{n^{2\ell/d+c/d+2\sigma/d}}. %\leq \frac{\left(\frac{eC^in^{1-(i+1)/d}}{c}\right)^c e^{(A+2)\ell}}{n^{2\ell/d+2\sigma/d}}.
$$
Let $C>0$ be a large constant. Since every subgraph $J\subset F$ with $2\leq|V(H)|\leq n^{1-\delta}$ has $|\partial_e(J)|\geq d+1+\lfloor |V(J)|w/\log n\rfloor$, for a graph $J\in\mathcal{J}^H_{\ell,x,c}$ consisting of $c$ components that have $\ell_1,\ldots,\ell_c$ edges and $x_1,\ldots,x_c$ vertices, we have $\ell_i=\frac{d}{2}x_i-\frac{\Delta}{2}-\sigma_i$, where 
$$
\sigma_i\geq \frac{dC\ell_i}{2\ln n}\cdot I\left(\ell_i\geq\frac{\ln n(1-d\delta)}{dC}\right).
$$
%Indeed, since every subgraph of $J\subset F$ of size at least 2 and at most $n/2$ has edge boundary at least $d+1+\lfloor |V(J)|w/\log n\rfloor$, every set of $\frac{n}{2}\geq x'\geq \frac{\ln n}{200C}$ vertices induces at most $\frac{d}{2}x'-\frac{\Delta}{2}-1$ edges in $F$.
  Therefore, for any admissible triple $(\ell,x,c)$, we get $\ell=\frac{d}{2}x-\frac{\Delta}{2}c-\sigma$, where 
\begin{align}
\sigma=\min_{J\in\mathcal{J}^H_{\ell,x,c}}\sum_{i=1}^c\sigma_i &\geq\min_{J\in\mathcal{J}^H_{\ell,x,c}}\sum_{i:\,\ell_i\geq\frac{\ln n(1-d\delta)}{dC}}\frac{dC\ell_i}{2\ln n}\notag\\
&>\left(\ell-c\cdot\frac{\ln n(1-d\delta)}{dC}\right)\frac{dC}{2\ln n}
=\frac{dC}{2\ln n}\cdot\ell-\frac{1-d\delta}{2}\cdot c.
 \label{eq:sigma_lower}
\end{align}
Therefore,
\begin{align*}
\sum_{\ell}\Pi^H_{\mathcal{B}_{\ell}}\left(\left(1+\frac{3f}{m}\right)\frac{N}{m}\right)^{\ell}& \leq
\sum_{\ell>\max\{n^{1-(i+1)\delta},\ln n\}}\sum_{x,c}p^H(\ell,x,c)\left(\frac{1+o(1)}{\varepsilon}\cdot n^{2/d}\right)^{\ell}\\
&\leq n^2\sum_{\ell>\max\{n^{1-(i+1)\delta},\ln n\}}\frac{\left(\frac{e n^{1-i\delta}}{c}\right)^c }{n^{\delta c+\frac{C}{\ln n}\cdot\ell}}\left(\frac{e^A+o(1)}{\varepsilon}\right)^{\ell}\\
&\leq n^2\sum_{\ell>\max\{n^{1-(i+1)\delta},\ln n\}}\frac{e^{n^{1-(i+1)\delta}}}{e^{C\ell}}\left(\frac{e^A+o(1)}{\varepsilon}\right)^{\ell}=o\left(\frac{1}{n^2}\right),
\end{align*}
whenever $C>2(A+\ln(1/\varepsilon))$, completing the proof.
\end{proof}
By induction, whp we get a multiset $\mathcal{F}^{(\lceil 1/\delta\rceil)}_n$ of size $(1-o(1))|\mathcal{F}_n|$ comprising a single $H\subset F$ from almost every $F\in\mathcal{F}_n$ such that the graph $H\cup\mathbf{W}_1\cup\ldots\cup\mathbf{W}_{\lceil 1/\delta\rceil}$ contains some $F'\in\mathcal{F}_n$ and $|H|=\lfloor\ln n\rfloor$.

\subsubsection{Proof of Theorem~\ref{th:main2_1}}

%It is well known that increasing properties that hold whp in the uniform model hold whp in the respective binomial model as well (see, e.g.,~\cite[Corollary 1.16]{Janson}). 

It is well known that increasing properties that hold whp in the uniform model hold whp in the respective binomial model as well, and vice versa (see, e.g.,~\cite[Corollary 1.16]{Janson}). Let
$$
 \mathbf{G}\sim G(n,p),\quad\text{ where }\quad p=(1+\lceil 1/\delta \rceil\varepsilon)(e/n)^{2/d}.
$$
By~\cite[Corollary 1.16]{Janson}, whp there exists a multiset $\mathcal{F}^{(\lceil 1/\delta\rceil)}_n=\mathcal{F}^{(\lceil 1/\delta\rceil)}_n(\mathbf{G})$ of graphs of size $\lfloor\ln n\rfloor$ comprising a single subgraph $H$ of almost every $F\in\mathcal{F}_n$ so that $H\cup\mathbf{G}$ contains some $F'\in\mathcal{F}_n$, since the existence of such a multiset is an increasing property.

Let $X$ be the number of $H\in\mathcal{F}^{(\lceil 1/\delta\rceil)}_n$ that are subgraphs of $\mathbf{G}'\sim G(n,p'=\varepsilon n^{-2/d})$, sampled independently of $\mathbf{G}$. We get 
$$
\mathbb{E}X=\left|\mathcal{F}_n^{(\lceil 1/\delta\rceil)}\right|\cdot p'^{\lfloor\ln n\rfloor }=(1-o(1))|\mathcal{F}_n|p'^{\lfloor\ln n\rfloor }=\omega(1).
$$
Let $\mathcal{B}$ be the set of non-empty graphs. Due to the definition~\eqref{eq:Pi_def} of $\Pi^H_{\mathcal{B}_{\ell}}=\Pi^H_{\mathcal{B}_{\ell}}(\mathcal{F}^{(\lceil 1/\delta\rceil)}_n)$, 
$$
 \frac{\mathrm{Var}X}{(\mathbb{E}X)^2} \leq\frac{\max_{H\in\mathcal{F}_n^{(\lceil 1/\delta\rceil)}}\sum_{\ell\geq 1}\left(\Pi^H_{\mathcal{B}_{\ell}}\cdot\left|\mathcal{F}_n^{(\lceil 1/\delta\rceil)}\right|\cdot p'^{\lfloor\ln n\rfloor-\ell}\right)}{\mathbb{E}X}=\max_{H\in\mathcal{F}_n^{(\lceil 1/\delta\rceil)}}\sum_{\ell\geq 1}\Pi^H_{\mathcal{B}_{\ell}}p'^{-\ell}.
$$
Therefore, due to Claim~\ref{cl:J_upper_bound}, Claim~\ref{cl:simple_embeddings}, estimates~\eqref{eq:sigma_lower}, and the inequality
\begin{equation} 
\frac{(n-x+c)!}{n!}  \leq\frac{e^{(x-c)^2/n}}{n^{(2\ell+c+2\sigma)/d}},
\label{eq:factorials_ratio_general}
\end{equation}
we get that, for some constant $A>0$, %\MZ{When we apply Claim~\ref{eq:J_upper_bound_trivial}, we have $|V(H)|$ there but use it for $|E(H)|$. As an option, replace $V$ with $E$ in this claim}
\begin{align*}
 \frac{\mathrm{Var}X}{(\mathbb{E}X)^2} &\leq
\sum_{\ell\geq 1}\sum_{x,c}(1-o(1))\frac{{\lfloor\ln n\rfloor\choose c}e^{A \ell}}{n^{2\ell/d+c/d+2\sigma/d}}\left(\frac{1}{\varepsilon}\cdot n^{2/d}\right)^{\ell}\\
&\leq \sum_{\ell\geq 1}\sum_{c\geq 1} O(\ell) \left(\frac{e\lfloor\ln n\rfloor}{cn^{\delta}}\right)^c\left(\frac{e^{A}}{\varepsilon\cdot e^C}\right)^{\ell}=O\left(\frac{\ln n}{n^{\delta}}\right),
\end{align*}
whenever $C>2(A+\ln(1/\varepsilon))$. This completes the proof of Theorem~\ref{th:main2_1}, due to Chebyshev's inequality.

\subsection{Better expansion of small subgraphs}
\label{sc:sharp2}

It remains to prove the second part of Theorem~\ref{th:main2}:

\begin{theorem}
Let $d\geq 4$. Let $\gamma\in(0,1)$ be a fixed constant, let $C=C(d,\gamma)$ be large enough, and let $D\geq 1$ be a fixed constant. Let $F=F(n)$ be a sequence of $d$-regular graphs on $[n]$, $n\in\mathbb{N}$, such that
\begin{itemize}
%\item $F_n$ has $e^{o(n)}$ automorphisms;
\item for every subgraph $\tilde F\subset F$ with $|E(\tilde F)|\leq C n^{2/d}$ and $|\partial_v(\tilde F)|\leq \gamma|V(\tilde F)|$, the number of automorphisms of $\tilde F$ that fix all vertices of $\partial_v(\tilde F)$ is at most $D$; 
\item for every subgraph $\tilde F\subseteq F$ with $|\partial_v(\tilde F)|\leq \gamma|V(\tilde F)|$, the number of automorphisms of $\tilde F$ that fix all vertices of $\partial_v(\tilde F)$ is at most $e^{o(|V(\tilde F)|)}$;
\item every $\tilde F\subset F$ with $3\leq|V(\tilde F)|\leq n-3$ has $|\partial_e(\tilde F)|\geq 2d$.
\end{itemize}
Let $\varepsilon>0.$ If $p>(1+\varepsilon)\left(\frac{e}{n}\right)^{2/d}$, then whp (assuming that $dn$ is even) $\mathbf{G}\sim G(n,p)$ contains an isomorphic copy of $F$. % \MZ{make a note somewhere that it implies sharp threshold. It is not obvious here}
\label{th:main2_2}
\end{theorem}

We notice that Theorem~\ref{th:main2_2} indeed implies the second part of Theorem~\ref{th:main2} since $p^e_F\geq(1-o(1))\left(\frac{e}{n}\right)^{2/d}$. Indeed, due to the second requirement applied to $\tilde F=F$ with $\partial_v(\tilde F)=\varnothing$, we get $|\mathrm{Aut}(F)|=e^{o(n)}$, implying~\eqref{eq:from_part_Th2}. The rest of the section is devoted to the proof of Theorem~\ref{th:main2_2}. Fix a $d$-regular graph $F$ satisfying the requirements of the theorem.

\subsubsection{Fragmentation: two iterations}

Consider two independent samples 
$$
\mathbf{W}_1\sim G(n,m),\,\,\, m:=\left\lfloor(1+\varepsilon)\left(\frac{e}{n}\right)^{2/d}N\right\rfloor,
\quad\,\text{ and}
\quad\,
\mathbf{W}_2\sim G(n,m_0),\,\,\,
m_0:=\left\lfloor\varepsilon\cdot n^{-2/d}\cdot N\right\rfloor.
$$
Let $\mathcal{F}_n$ be the family of all isomorphic copies of $F$ on $[n]$.

\begin{claim}
\label{cl:sharp_2_day_1}
Whp there exists $F'\subset F\cup\mathbf{W}_1$ satisfying the following: 
$F'\in\mathcal{F}_n$, $F\cap F'$ has at most $\sqrt{\ln n}$ non-isolated edges, and $|F\cap F'|\leq C n^{2/d}$.
\end{claim}

\begin{proof}
Let $f=\frac{dn}{2}$, and let a graph $G$ on $n$ vertices has the property $\mathcal{B}$ if either $G$ has more than $C n^{2/d}$ edges or $G$ has more than $\sqrt{\ln n}$ non-isolated edges. Due to Lemma~\ref{lm:not_bad} and due to symmetry, it is sufficient to prove~\eqref{eq:sufficient_first_fragmentation}.

In contrast to the proof of Claims~\ref{cl:coarse_day_1}~and~\ref{cl:sharp_day_1}, here we will distinguish the contribution to $\Pi^F_{B_{\ell}}$ between components consisting of a single edge and all the others. Fix non-negative integers $c$, $c'$, and $x$. Denote by $\mathcal{J}_{\ell,x,c,c'}$ the set of all subgraphs $J\subset F$ with $\ell$ edges, $x$ non-isolated vertices, $c$ isolated edges, and $c'$ other connected components, excluding isolated vertices. Denote by $p(\ell,x,c,c')$ the probability that $\mathbf{F}\cap F\in\mathcal{J}_{\ell,x,c,c'}$. We now refer to the proof of Claim~\ref{cl:main} and consider separately connected components of some $J\in\mathcal{J}_{\ell,x,c,c'}$ consisting of single edges. The number of ways to embed $c$ disjoint edges into $F$ is at most 
\begin{align*}
n(n-2)\ldots(n-2(c-1)) \cdot d^c &=2^c\frac{(n/2)!}{(n/2-c)!}d^c\\
&\leq 2\sqrt{\frac{n}{n-2c+1}}\cdot \frac{(2d)^c(n/2)^c}{e^c}\left(1+\frac{c}{n/2-c}\right)^{n/2-c}\\
&\leq 2\sqrt{\frac{n}{n-2c+1}}\cdot (dn)^c=:\varphi_1.
\end{align*}
We also note that $J$ does not have closed components. After we identify the image $z'_i$ in $F$ of a non-trivial connected component $z_i$ of $J$ with $x_i$ vertices and $o_i$ boundary vertices, the number of ways to map the vertices of $z_i$ to the vertices of the image is bounded by $\left(D\cdot\1_{\ell_i\leq Cn^{2/d}}+e^{o(x_i)}\cdot\1_{\ell_i>Cn^{2/d}}\right)\cdot d^{o_i}$ whenever $o_i\leq \gamma x_i$ (since every automorphism of $z'_i$ acts bijectively on the vertex boundary). If $o_i> \gamma x_i$, then we may bound the number of automorphisms by 
$$
x_i(d-1)^{x_i}\leq  (o_i/\gamma)(d-1)^{o_i/\gamma}<(2d)^{o_i/\gamma},
$$
in the same way as in the proof of Claim~\ref{cl:main}.  %; 2) the trivial extension of an automorphism of $z'_i$ that fixes all free vertices to the entire $F$ is an automorphism of $F$, and so there are at most $D$ such automorphisms.

Let $\sigma:=\frac{d}{2}(x-2c)-(\ell-c+\frac{\Delta}{2}\cdot c')$ be the excess of the union of components of $J$ with more than 2 vertices. Following the proof of the second part of Claim~\ref{cl:main}, we get that there exist constants $A^1_{\beta},A^2_{\beta}>0$ such that the number of ways to embed $J$ into $F$ is at most 
$$
\beta(\ell,x,c,c')=\varphi_1\cdot(n-x+c')!e^{A^1_{\beta}c'+A^2_{\beta}\sigma}\left(1+e^{o(\ell)}\cdot\1_{\ell>Cn^{2/d}}\right)
\times\max_{o\leq (\Delta+2)\sigma}{x-2c\choose o}%\left(\frac{o\cdot\ln n}{c'C}\right)^{c'}
(2d)^{o/\gamma}.
$$
%if $\ell\leq Cn^{2/d}$, \MZ{add the case $Cn^{2/d}<\ell<\varepsilon' n$ separately --- seems that we need an additional assumption here?} and at most ${dn/2\choose c}c!2^c(n-x+c')!e^{A^1_{\beta}c'+A^2_{\beta}\sigma}\max_{o\leq (\Delta+2)\sigma}{x-2c\choose o}\cdot e^{o(n)}$, otherwise. \MZ{this is when $\ell>\varepsilon' n$} %, due to Claim~\ref{cl:small_number_of_global_automorphisms}.

Now, let us bound the number of ways to choose a subgraph $J\subset F$ such that $J\in\mathcal{J}_{\ell,x,c,c'}$. The number of ways to choose $c>0$ disjoint edges in $F$ is at most 
$$
\frac{n(n-2)\ldots(n-2(c-1))\cdot d^c}{c!\cdot 2^c}\leq\frac{\varphi_1}{\sqrt{c}\cdot(2c/e)^c}=:\varphi_2.
$$ 
According to the proof of the first part of Claim~\ref{cl:main}, there exist constants $A^1_{\alpha},A^2_{\alpha}>0$ such that the number of ways to choose the other $c'$ components is at most
$$
\alpha(\ell,x,c,c')={n-2c\choose c'}{x-2c\choose c'}e^{A^1_{\alpha} c'+A^2_{\alpha}\sigma} \max_{o\leq (\Delta+2)\sigma}{x-2c\choose o}.
$$
Letting $\varphi_2:=1$ when $c=0$, we conclude that, for all $c\geq 0$ and for some constants $A_1,A_2>0$,
\begin{multline}
\label{eq:th2_p_general_bound}
p(\ell,x,c,c') \leq\frac{\varphi_2\cdot\alpha(\ell,x,c,c')\beta(\ell,x,c,c')}{|\mathcal{F}_n|\cdot|\mathrm{Aut}(F)|}\\
\leq\frac{\varphi_1\varphi_2{n-2c\choose c'}{x-2c\choose c'}e^{A_1 c'+A_2\sigma}(n-x+c')!}{n!}\left(1+e^{o(\ell)}\cdot\1_{\ell>Cn^{2/d}}\right)\max_{o\leq (\Delta+2)\sigma}{x-2c\choose o}^2.%\left(\frac{o\cdot\ln n}{c'C}\right)^{c'}
%(2d)^{o/\gamma}.
\end{multline}
%if $\ell\leq Cn^{2/d}$, and 
%\MZ{add the case $Cn^{2/d}<\ell<\varepsilon' n$ separately}
%$$
%p(\ell,x,c,c')\leq e^{o(n)}\cdot \frac{{dn/2\choose c}^2c! 2^c{n-2c\choose c'}{x-2c\choose c'}e^{A_1 c'+A_2\sigma}(n-x+c')!\max_{o\leq (\Delta+2)\sigma}{x-2c\choose o}^2}{|\mathcal{F}_n|},
%$$
%otherwise. \MZ{This is when $\ell<\varepsilon' n$}

Note that every connected $\tilde F\subset F$ with $3\leq|V(\tilde F)|\leq n-3$, has $2d+2|E(\tilde F)|\leq d|V(\tilde F)|$. Therefore, the $i$-th connected component of $J$ with $\ell_i\geq 2$ and $x_i\leq n-3$ has 
$$
\sigma_i\geq \frac{2d-\Delta}{2}\geq\frac{d-2}{2}.
$$
In particular, for $J$ consisting of such connected components, we get 
$$
\sigma=\sum_i\sigma_i\geq\frac{2d-\Delta}{2}\cdot c'.
$$
%\geq\frac{c'}{2}.
%\MZ{Note that the above holds true only when $x_i\leq n-3$. Do we use it otherwise?}
 Since
\begin{equation}
x-c'=2\ell/d+(2-2/d)c+2\sigma/d+\frac{\Delta-d}{d}c'\geq
2\ell/d+(2-2/d)c+c',
\label{eq:n-x+c'}
\end{equation}
in a similar way as in~\eqref{eq:factorials_fraction_sharp}, we get that
\begin{equation}
 \frac{(n-x+c')!}{n!}\leq\frac{e^{\frac{(x-c')^2}{n}%+O\left(\frac{(x-c')^3}{n^2}\right)}
 }}{n^{x-c'}}\leq
 \frac{e^{\frac{(x-c')^2}{n}%+O\left(\frac{(x-c')^3}{n^2}\right)}
 }}{n^{2\ell/d+(2-2/d)c+c'+2(\sigma-(d-\Delta/2)c')/d}}.
\label{eq:factorials_fraction_sharp_3}
\end{equation}
We then conclude that
\begin{align}
 \frac{\varphi_1\varphi_2(n-x+c')!}{n!} &\leq
 4c^{-1/2}\frac{n}{n-2c+1}\cdot\frac{\left(\frac{ed^2n^2}{2c}\right)^c}{n^{2\ell/d+(2-2/d)c+c'+2(\sigma-(d-\Delta/2)c')/d}}\cdot e^{\frac{(x-c')^2}{n}},&\quad\text{ if }c>0;\notag\\
 \label{eq:varphi1varphi2_c=0}
 \frac{\varphi_1\varphi_2(n-x+c')!}{n!} &\leq
 \frac{2}{n^{2\ell/d+c'+2(\sigma-(d-\Delta/2)c')/d}}\cdot e^{\frac{(x-c')^2}{n}},&\quad\text{ if }c=0.
 %\cdot\frac{e^{-\frac{2c^2}{n}+\frac{(x-c')^2}{2n}+O\left(\frac{(x-c')^3}{n^2}\right)}}{n^{2\ell/d+(2-2/d)c+c'+2(\sigma-(d-\Delta/2)c')/d}}\\
% &\leq
% c^{-1/2}(1+o(1))\frac{n}{n-2c}\cdot\left(\frac{ed^2n^{2/d}}{2c}\right)^c\cdot\frac{e^{\frac{(x+2c-c')(x-2c-c')}{2n}+O\left(\frac{(x-c')^3}{n^2}\right)}}{n^{2\ell/d+c'+2(\sigma-(d-\Delta/2)c')/d}}.
\end{align}
In particular, when $c>0$, we get %The latter implies {\bf [consider separately case $c=0$]}
 %, for $\ell\leq Cn^{2/d}$,
\begin{align}
\label{eq:varphi1varphi2_tight}
\frac{\varphi_1\varphi_2(n-x+c')!}{n!}e^{-\frac{(nd/2)^2}{m}}\left(\frac{N}{m}\right)^{c}
&\leq
 \frac{4n}{n-2c+1}\cdot\frac{c^{-1/2}\left(\frac{ed^2n^2n^{2/d}}{4cm}\right)^c}{n^{2\ell/d+c'+2(\sigma-(d-\Delta/2)c')/d}}\cdot e^{-\frac{(nd/2)^2}{m}+\frac{(x-c')^2}{n}}\\%+\frac{(x+2c-c')(x-2c-c')}{2n}+O\left(\frac{(x-c')^3}{n^2}\right)}\\
&\leq
\frac{4n}{n-2c}\cdot\frac{c^{-1/2}}{n^{2(\ell-c)/d+c'+2(\sigma-(d-\Delta/2)c')/d}}\cdot e^{\frac{(x-c')^2}{n}}.%\frac{(x+2c-c')(x-2c-c')}{2n}+O\left(\frac{(x-c')^3}{n^2}\right)}.
\label{eq:phiphi}
\end{align}
%\MZ{It does not seem that we need to compute things so carefully when $d\geq 4$.}
The last expression is $O(1/\sqrt{c})$ when $\ell=c=O(\sqrt{n})$ (and thus $\sigma=c'=0$ and $x=2c$), and, as we will see soon, it gives the main contribution to the sum in~\eqref{eq:sufficient_first_fragmentation}.

Let us recall that we have to prove that % \MZ{sometimes we have the bound $\ell>Cn^{2/d}$ and sometimes --- $x>Cn^{2/d}$. Be consistent}
\begin{equation}
 \left(\sum_{\ell>Cn^{2/d}}\sum_{x,c,c'}+\sum_{\ell\leq Cn^{2/d}}\sum_{x-2c>\sqrt{\ln n}}\sum_{c'}\right)p(\ell,x,c,c')e^{-\frac{(nd/2)^2}{m}}\left(\frac{N}{m}\right)^{\ell}\left(1+\frac{3dn}{2m}\right)^{\ell}=o(1).
\label{eq:sharp_2_day_1_conclusion_0}
\end{equation}
We choose $0<\varepsilon'\ll\varepsilon''\ll\varepsilon$ small enough and consider separately three cases: $\ell\leq Cn^{2/d}$, $Cn^{2/d}<\ell\leq\varepsilon' n$, and $\ell>\varepsilon' n$. % Notice that we may choose $C$ large, depending on $\varepsilon'$. %\MZ{Introduce $\varepsilon'$}
%\MZ{Check carefully --- now $C$ cannot depend on $\varepsilon$! Looks good but if not --- maybe remove $C$ from the statement of Theorem 5 and add the existential quantifier over $C$ into the statement of Claim 10? Or make different constants?}

%\begin{enumerate}

{\bf 1.} $\ell\leq Cn^{2/d}$. Since $d\geq 4$, the equalities $n\ell=O(m)$ and $(x-c')^2=O(n)$ hold. In particular, $\left(1+\frac{3dn}{2m}\right)^{\ell}e^{(x-c')^2/n}=O(1)$. Combining~\eqref{eq:th2_p_general_bound}, \eqref{eq:varphi1varphi2_c=0}, and~\eqref{eq:phiphi}, we get % \MZ{Here we need $d\geq 4$!}
\begin{multline}
 %\sum_{\ell\leq Cn^{2/d}}\sum_{x-2c>\sqrt{\ln n}}\sum_{c'}
 p(\ell,x,c,c')e^{-\frac{(nd/2)^2}{m}}\left(\frac{N}{m}\right)^{\ell}\left(1+\frac{3dn}{2m}\right)^{\ell}\\
 =O\left(\frac{{n-2c\choose c'}{x-2c\choose c'}e^{A_1 c'+A_2\sigma}}{c'+2(\sigma-(d-\Delta/2)c')/d}\left(
 \frac{{x-2c\choose o}^2\left(\frac{N}{m}\right)^{\ell-c}}{\sqrt{c}\cdot n^{2(\ell-c)/d}}\cdot\1_{c>0}+
 \frac{{x\choose o}^2\left(\frac{N}{m}\right)^{\ell}}{n^{2\ell/d}}e^{-\frac{(nd/2)^2}{m}}\cdot\1_{c=0}\right)\right),
 \label{eq:sharp_2_day_1_conclusion_1}
\end{multline}
where $o\leq(\Delta+2)\sigma$ maximises the above expression. 

{\bf 1.1.} If $\sigma<\varepsilon'(x-2c)$, then, since $c'\leq\frac{2}{2d-\Delta}\sigma$ and $\ell\geq x-c-c'$, we get, for $c>0$,
\begin{align}
\frac{{n-2c\choose c'}{x-2c\choose c'}e^{A_1 c'+A_2\sigma}{x-2c\choose o}^2\left(\frac{N}{m}\right)^{\ell-c}}{\sqrt{c}\cdot n^{2(\ell-c)/d+c'}}&\leq
\frac{{n-2c\choose c'}e^{\varepsilon''\cdot (x-2c)}}{\sqrt{c}\cdot n^{c'}e^{2(\ell-c)/d}}\leq
\frac{\left(\frac{e^{1+2/d}}{c'}\right)^{c'}e^{\varepsilon''\cdot (x-2c)}}{\sqrt{c}\cdot e^{2(x-2c)/d}}\notag\\
&<\left(\frac{e^{1+2/d}}{c'}\right)^{c'}e^{-\frac{x-2c}{d}}\cdot c^{-1/2}.
\label{eq:case1.1-first_try}
\end{align}
This bound is not sufficient to prove that summation over all $x,c$ with the same value of $x-2c$ is small. Nevertheless, letting $\tau:=\frac{d^2n^2}{4m}$ and $c=:\tau+t$, we get
\begin{align}
\sum_{c=0}^{(1+\varepsilon')\tau}\left(\frac{ed^2n^{2}}{4c m}\right)^c&=
\sum_{t}\left(\frac{e\tau}{\tau+t}\right)^{\tau+t}=e^{\tau}\sum_{t}e^t\left(\frac{1}{1+t/\tau}\right)^{\tau+t}\notag\\
&\leq
e^{\tau}\sum_{t}e^{t-(\tau+t)t/\tau+(t/\tau)^2(\tau+t)/2}
\leq
e^{\tau}\sum_{t}e^{-(1-\varepsilon')t^2/(2\tau)}<3\sqrt{\tau}\cdot e^{\tau}.
\label{eq:integral_normal}
\end{align}
Moreover, for $c\notin[(1-\varepsilon')\tau,(1+\varepsilon')\tau]$, we get 
\begin{equation}
\label{eq:c-not_typical}
 \left(\frac{ed^2n^2}{4cm}\right)^ce^{-\frac{(nd/2)^2}{m}}=e^{-\Omega(n^{2/d})}.
\end{equation}

We are now ready to estimate the contribution of the considered values of parameters $\ell,x,c,c'$ to the left-hand side of~\eqref{eq:sharp_2_day_1_conclusion_0}. Let us again recall the inequality~\eqref{eq:th2_p_general_bound}.  In order to refine the bound~\eqref{eq:case1.1-first_try}, we will use the bound~\eqref{eq:varphi1varphi2_tight} instead of~\eqref{eq:phiphi}. Recalling that $c'\leq\frac{2}{2d-\Delta}\sigma$ and arguing in a similar way as in~\eqref{eq:case1.1-first_try}, we get that any summand in the left-hand side of~\eqref{eq:sharp_2_day_1_conclusion_0} that corresponds to $c>0$ is at most
\begin{multline*}
 \frac{\varphi_1\varphi_2{n-2c\choose c'}{x-2c\choose c'}e^{A_1 c'+A_2\sigma}(n-x+c')!}{n!}{x-2c\choose o}^2e^{-\frac{(nd/2)^2}{m}}\left(\frac{N}{m}\right)^{\ell}\left(1+\frac{3dn}{2m}\right)^{\ell}\\
 =O\left(\frac{c^{-1/2}e^{A_1 c'+A_2\sigma}}{n^{2(\ell-c)/d+c'}}\cdot \left(\frac{e(nd/2)^2}{cm}\right)^c e^{-\frac{(nd/2)^2}{m}} {n-2c\choose c'}{x-2c\choose c'}{x-2c\choose o}^2\left(\frac{N}{m}\right)^{\ell-c}\right)\\
 =O\left(\left(\frac{e^{1+2/d}}{c'}\right)^{c'}e^{-\frac{x-2c}{d}}\cdot c^{-1/2}\left(\frac{e(nd/2)^2}{cm}\right)^c\cdot e^{-\frac{(nd/2)^2}{m}}\right).
\end{multline*}
Then, for every fixed $c'$ and $x-2c$, from~\eqref{eq:integral_normal} we derive%{\bf [expand]}
$$
\sum_{c\in[(1-\varepsilon')\tau,(1+\varepsilon')\tau]} p(\ell,x,c,c')e^{-\frac{(nd/2)^2}{m}}\left(\frac{N}{m}\right)^{\ell}\left(1+\frac{3dn}{2m}\right)^{\ell}=O\left(\left(\frac{e^{1+2/d}}{c'}\right)^{c'}\cdot e^{-(x-2c)/d}\right).
$$
Due to~\eqref{eq:c-not_typical}, we get that $c\notin[(1-\varepsilon')\tau,(1+\varepsilon')\tau]$ contribute $e^{-\Omega(n^{2/d})}$ to the left-hand side of~\eqref{eq:sharp_2_day_1_conclusion_0}, implying
$$
\sum_{c\geq 0} p(\ell,x,c,c')e^{-\frac{(nd/2)^2}{m}}\left(\frac{N}{m}\right)^{\ell}\left(1+\frac{3dn}{2m}\right)^{\ell}=O\left(\left(\frac{e^{1+2/d}}{c'}\right)^{c'}\cdot e^{-(x-2c)/d}\right)+e^{-\Omega(n^{2/d})}.
$$
Since for fixed $c$ and fixed $x-2c$, the value of $\ell$ ranges between $c+(x-2c-c')$ and $c+(x-2c-c')\frac{d}{2}$, we get
\begin{align}
\sum_{\ell\leq Cn^{2/d}} \sum_c\sum_{x,c':\,\sigma<\varepsilon'(x-2c)} & p(\ell,x,c,c')e^{-\frac{(nd/2)^2}{m}}\left(\frac{N}{m}\right)^{\ell}\left(1+\frac{3dn}{2m}\right)^{\ell}\notag\\
&\leq \sum_{c'}\sum_{x-2c>\sqrt{\ln n}}\frac{d}{2}(x-2c)\left(O\left(\left(\frac{e^{1+2/d}}{c'}\right)^{c'} e^{-(x-2c)/d}\right)+e^{-\Omega(n^{2/d})}\right)\notag\\
&=e^{-\Omega(\sqrt{\log n})}.
\label{eq:sharp_2_day_1_conclusion_2}
\end{align}
%where the latter sum is over all $c'$ and $x-2c>\sqrt{\ln n}$ such that $\sigma<\varepsilon'(x-2c)$.
 
%\MZ{The problem here is that it does not sum up to $o(1)$ when summing over all $x,c$ with fixed $x-2c$. So, we have to consider a $\log n$-ball around $c$ that maximises the expression.}
{\bf 1.2.} If $\sigma\geq\varepsilon'(x-2c)$ and $c'\leq\frac{1}{2d-\Delta}\sigma$, then
\begin{align}
\frac{{n-2c\choose c'}{x-2c\choose c'}e^{A_1 c'+A_2\sigma}{x-2c\choose o}^2\left(\frac{N}{m}\right)^{\ell-c}}{n^{2(\ell-c)/d+c'+2(\sigma-(d-\Delta/2)c')/d}}&\leq
\frac{{n-2c\choose c'}e^{O(x-2c)}}{n^{c'+\sigma/d}}\notag\\
&\leq
\frac{\left(\frac{e}{c'}\right)^{c'}e^{O(x-2c)}}{n^{\Omega(x-2c)}}=n^{-\Omega(x-2c)}.
\label{eq:sharp_2_day_1_conclusion_3}
\end{align}

{\bf 1.3.} Finally, if $\sigma\geq\varepsilon'(x-2c)$ and $c'\geq\frac{1}{2d-\Delta}\sigma$, then
\begin{align*}
\frac{{n-2c\choose c'}{x-2c\choose c'}e^{A_1 c'+A_2\sigma}{x-2c\choose o}^2\left(\frac{N}{m}\right)^{\ell-c}}{n^{2(\ell-c)/d+c'+2(\sigma-(d-\Delta/2)c')/d}}&\leq
\frac{{n-2c\choose c'}e^{O(x-2c)}}{n^{c'}}\\
&\leq \left(\frac{e}{c'}\right)^{c'}e^{O(x-2c)}=(c')^{-c'/2}(x-2c)^{-\Omega(x-2c)}.
\end{align*}
Applying here exactly the same argument as in the case {\bf 1.1}, we get
\begin{equation}
\sum p(\ell,x,c,c')e^{-\frac{(nd/2)^2}{m}}\left(\frac{N}{m}\right)^{\ell}\left(1+\frac{3dn}{2m}\right)^{\ell}
\leq e^{-\Omega\left(\sqrt{\log n}\cdot \log\log n\right)},
\label{eq:sharp_2_day_1_conclusion_4}
\end{equation}
where the sum is over all $\ell\leq Cn^{2/d}$, $x$, $c$, and $c'$, such that $x-2c>\sqrt{\ln n}$, $\sigma\geq \varepsilon'(x-2c)$, and $c'\geq\frac{1}{2d-\Delta}\sigma$.

{\bf 2.} $Cn^{2/d}\leq \ell\leq\varepsilon' n$. From~\eqref{eq:th2_p_general_bound}~and~\eqref{eq:phiphi}, we get
\begin{multline*}
 %\sum_{\ell>Cn^{2/d}}\sum_{x,c,c'}
 p(\ell,x,c,c')e^{-\frac{(nd/2)^2}{m}}\left(\frac{N}{m}\right)^{\ell}\left(1+\frac{3dn}{2m}\right)^{\ell}\\
 \leq %\sum_{\ell>Cn^{2/d}}\sum_{x,c,c'}
 e^{\varepsilon''\ell}\cdot\left(\frac{ed^2n^{2/d}}{2ce^{2/d}}\right)^c\cdot e^{-\frac{(nd/2)^2}{m}}\cdot\frac{{n-2c\choose c'}{x-2c\choose c'}e^{A_1 c'+A_2\sigma}{x-2c\choose o}^2\left(\frac{N}{m}\right)^{\ell-c}}{n^{2(\ell-c)/d+c'+2(\sigma-(d-\Delta/2)c')/d}},
\end{multline*}
where $o\leq(\Delta+2)\sigma$ maximises the above expression. If $c\leq\frac{1}{4}x$, then $x-2c>\sqrt{\ln n}$, meaning that the argument from the previous case can be applied here as well, implying exactly the same bound. If $c>\frac{1}{4}x$, then the latter expression is at most
\begin{equation}
 e^{-x\frac{\ln C}{10}}\cdot\frac{{n-2c\choose c'}{x-2c\choose c'}e^{A_1 c'+A_2\sigma}{x-2c\choose o}^2\left(\frac{N}{m}\right)^{\ell-c}}{n^{2(\ell-c)/d+c'+2(\sigma-(d-\Delta/2)c')/d}}
 <e^{-x\frac{\ln C}{10}}\cdot\frac{8^xe^{(A_1+A_2)x}}{n^{2(\ell-c)/d}\left(\frac{m}{N}\right)^{\ell-c}}\leq  e^{-x\frac{\ln C}{20}},
 \label{eq:sharp_2_day_1_conclusion_5}
\end{equation}
whenever $\ln C>20(A_1+A_2+3)$.

{\bf 3.} Finally, let $\ell>\varepsilon' n$.

{\bf 3.1.} We separately consider the case that either $x\geq n-2$, $c'=1$, and $c=0$, or $x'=n$, $c'=1$, and $c=1$. In this case, we cannot apply~\eqref{eq:n-x+c'}. Nevertheless, since $c$ and $c'$ are bounded, we have the following bound, due to~\eqref{eq:th2_p_general_bound}:
$$
 p(\ell,x,c,c')e^{-\frac{(nd/2)^2}{m}}\left(\frac{N}{m}\right)^{\ell}\left(1+\frac{3dn}{2m}\right)^{\ell}
 \leq 
 e^{\varepsilon''\ell}\cdot e^{\frac{(x-c')^2}{n}}\cdot\frac{e^{A_2\sigma}{x-2c\choose o}^2\left(\frac{N}{m}\right)^{\ell}}{n^{2\ell/d+2\sigma/d}}.
$$
The latter bound is $n^{-\Omega(n)}$ when $\sigma>\varepsilon' x$. On the other hand, if $\sigma\leq\varepsilon' x$, then $e^{A_2\sigma}{x-2c\choose o}^2<e^{\varepsilon''\ell}$ and $x=2\ell/d+2\sigma/d+O(1)$. Thus,
\begin{align}
 p(\ell,x,c,c')e^{-\frac{(nd/2)^2}{m}}\left(\frac{N}{m}\right)^{\ell}\left(1+\frac{3dn}{2m}\right)^{\ell}
 &\leq 
 e^{2\varepsilon''\ell}\cdot e^{x}\left(\frac{e^{-2/d}}{1+\varepsilon}\right)^{\ell}\notag\\
 &\leq 
 e^{3\varepsilon''\ell}\cdot e^{2\ell/d}\left(\frac{e^{-2/d}}{1+\varepsilon}\right)^{\ell}=
 \left(\frac{e^{3\varepsilon''}}{1+\varepsilon}\right)^{\ell}<e^{-\varepsilon\ell/2}.
\end{align}

{\bf 3.2.} Otherwise, we can apply~\eqref{eq:n-x+c'}.  %Then $\ell\geq\frac{x}{2}>\frac{\varepsilon'}{2}n$. Here, instead of the bound
 Due to~\eqref{eq:th2_p_general_bound}~and~\eqref{eq:factorials_fraction_sharp_3}, %we will use the following:
%$$
% \frac{(n-x+c')!}{n!}\leq\frac{e^{\frac{(x-c')^2}{n}}}{n^{x-c'}}=\frac{e^{\frac{(x-c')^2}{n}}}{n^{2\ell/d+(2-2/d)c+c'+2(\sigma-(d-\Delta/2)c')/d}}.
%$$
%Then
\begin{multline}
 p(\ell,x,c,c')e^{-\frac{(nd/2)^2}{m}}\left(\frac{N}{m}\right)^{\ell}\left(1+\frac{3dn}{2m}\right)^{\ell}\\
 \leq 
 e^{\varepsilon''\ell}\cdot e^{\frac{(x-c')^2}{n}}\cdot\left(\frac{ed^2n^{2/d}}{2ce^{2/d}}\right)^c\cdot\frac{{n-2c\choose c'}{x-2c\choose c'}e^{A_1 c'+A_2\sigma}{x-2c\choose o}^2\left(\frac{N}{m}\right)^{\ell-c}}{n^{2(\ell-c)/d+c'+2(\sigma-(d-\Delta/2)c')/d}},
 \label{eq:sharp_2_first_day_last_case}
\end{multline}
where $o\leq(\Delta+2)\sigma$ maximises the above expression. If $c'>\varepsilon' x$, then $n^{c'}=n^{\Omega(n)}$ and $n^{2(\ell-c)/d}>\left(\frac{N}{m}\right)^{\ell-c}$, implying that the right-hand side of~\eqref{eq:sharp_2_first_day_last_case} is $n^{-\Omega(n)}$. If $c'\leq\varepsilon' x$ and $\sigma>(2d-\Delta)\varepsilon' x$, then $n^{2(\sigma-(d-\Delta/2)c')/d}=n^{\Omega(n)}$, implying the bound $n^{-\Omega(n)}$ as well. Finally, if $(d-\Delta/2)c'\leq \sigma\leq (2d-\Delta)\varepsilon' x$, then consider separately $c<x/\sqrt{\ln n}$ and $c\geq x/\sqrt{\ln n}$. In the latter case, $\left(\frac{ed^2n^{2/d}}{2ce^{2/d}}\right)^c=e^{-\Omega(n\sqrt{\ln n})}$, implying the bound $e^{-\Omega(n\sqrt{\ln n})}$ for the right-hand side of~\eqref{eq:sharp_2_first_day_last_case}. If $c<x/\sqrt{\ln n}$, then $\ell-c=(1-o(1))\ell$. Thus, recalling~\eqref{eq:n-x+c'} and that $\left(\frac{ed^2n^{2/d}}{2ce^{2/d}}\right)^c=e^{O(n^{2/d})}=e^{o(\ell)}$,
\begin{align}
 p(\ell,x,c,c')e^{-\frac{(nd/2)^2}{m}}\left(\frac{N}{m}\right)^{\ell}\left(1+\frac{3dn}{2m}\right)^{\ell}&\leq
 \frac{e^{2\varepsilon''\ell+2\ell/d+(2-2/d)c+2\sigma/d+\frac{\Delta-d}{d}c'}}{((1+\varepsilon)e^{2/d})^{\ell-c}}\notag\\
 &\leq
  \frac{e^{3\varepsilon''\ell+2\ell/d}}{((1+\varepsilon)e^{2/d})^{\ell}}=\frac{e^{3\varepsilon''\ell}}{(1+\varepsilon)^{\ell}}<e^{-\varepsilon\ell/2}.
  \label{eq:sharp_2_day_1_conclusion_6}
\end{align}

From~\eqref{eq:sharp_2_day_1_conclusion_1}--\eqref{eq:sharp_2_day_1_conclusion_6}, we get that the left-hand side in~\eqref{eq:sharp_2_day_1_conclusion_0} is at most
\begin{multline*}
 e^{-\Omega(\sqrt{\log n})}+\sum_{x-2c>\sqrt{\ln n}}\sum_{c',\ell} n^{-\Omega(x-2c)}+e^{-\Omega\left(\sqrt{\log n}\cdot \log\log n\right)}+\\
 +\sum_{Cn^{2/d}<\ell\leq \varepsilon'n}\sum_{x,c,c'}e^{-x\cdot(\ln C)/20}+\sum_{\ell>\varepsilon'n}\sum_{x,c,c'}\left(e^{-\varepsilon\ell/2}+e^{-\Omega(n\sqrt{\ln n})}\right)=o(1),
\end{multline*}
completing the proof.
\end{proof}

For every $F\in\mathcal{F}_n$ such that there exists $F'$ as in the statement of Claim~\ref{cl:sharp_2_day_1}, we choose one such $F'$ and put $F\cap F'$ into a multiset $\mathcal{F}^*_n$. Let $\mathbf{H}$ be a uniformly random element of $\mathcal{F}^*_n$.
\begin{claim}
Whp there exists $H'\subset \mathbf{H}\cup\mathbf{W}_2$, $H'\in\mathcal{F}^*_n$, such that $|\mathbf{H}\cap H'|\leq \ln n$.
\end{claim}

\begin{proof}
Let $f=\lfloor Cn^{2/d}\rfloor$, $m=m_0$, and $\mathcal{B}$ be the property of graphs on $n$ vertices to have more than $\ln n$ edges. Due to Lemma~\ref{lm:not_bad}, it is sufficient to prove~\eqref{eq:main_lm}.

Fix non-negative integers $c$ and $x$. Recall that, for $H\in \mathcal{F}^{*}_n$, we denote by $\mathcal{J}^H_{\ell,x,c}$ the set of all subgraphs $J\subset H$ with $\ell$ edges, $x$ non-isolated vertices, and $c$ connected components (excluding isolated vertices). Denote by $p^H(\ell,x,c)$ the probability that $\mathbf{H}\cap H\in\mathcal{J}^H_{\ell,x,c}$. 

Since $H$ consists of isolated edges and at most $\sqrt{\ln n}$ other edges, we get that $\ell\leq c+\sqrt{\ln n}$. Due to Claim~\ref{cl:J_upper_bound} and Claim~\ref{cl:simple_embeddings}, for some constant $A>0$ that does not depend on $\varepsilon$,
$$
p^H(\ell,x,c) \leq\frac{{\lfloor Cn^{2/d}\rfloor\choose c}e^{A\ell}(n-x+c)!/|\mathrm{Aut}(F)|}{|\mathcal{F}^{*}_n|}=(1+o(1))\frac{{\lfloor Cn^{2/d}\rfloor\choose c}e^{A\ell}(n-x+c)!}{n!}.
$$
Since $d\geq 4$ and since every non-trivial subgraph of $F$ has edge boundary of size at least $2d-2$,
$$
 p^H(\ell,x,c)\stackrel{\eqref{eq:factorials_ratio_general}}\leq \frac{\left(\frac{e Cn^{2/d}}{c}\right)^c e^{(A+o(1))\ell}}{n^{2\ell/d+(1-2/d)c}}\leq \frac{\left(\frac{e^{1+A+o(1)} C}{c}\right)^c e^{(A+o(1))\sqrt{\ln n}}}{n^{2\ell/d}}. %\leq \frac{\left(\frac{eC^in^{1-(i+1)/d}}{c}\right)^c e^{(A+2)\ell}}{n^{2\ell/d+2\sigma/d}}.
$$
Therefore,
\begin{align*}
\sum_{\ell}\Pi^H_{\mathcal{B}_{\ell}}\left(\left(1+\frac{3f}{m}\right)\frac{N}{m}\right)^{\ell}& \leq
\sum_{\ell>\ln n}\sum_{x,c}p^H(\ell,x,c)\left(\frac{1+o(1)}{\varepsilon}\cdot n^{2/d}\right)^{\ell}\\
&\leq n^2\sum_{\ell>\ln n}\left(\frac{e^{1+A} C}{\ell-\sqrt{\ln n}}\right)^{\ell-\sqrt{\ln n}} e^{A\sqrt{\ln n}}\left(\frac{1+o(1)}{\varepsilon}\right)^{\ell}=o\left(\frac{1}{n^2}\right),
\end{align*}
completing the proof.
\end{proof}

\subsubsection{Proof of Theorem~\ref{th:main2_2}}

%\MZ{Add a remark somewhere that second moment + Friedgut is enough to prove it.}

Let $\mathbf{G}\sim G(n,p)$, where $p=(1+2\varepsilon)(e/n)^{2/d}$. By~\cite[Corollary 1.16]{Janson}, whp there exists a multiset $\mathcal{F}^*_n=\mathcal{F}^*_n(\mathbf{G})$ of graphs of size $\lfloor\ln n\rfloor$ comprising a single subgraph $H$ of almost every $F\in\mathcal{F}_n$ so that $H\cup\mathbf{G}$ contains some $F'\in\mathcal{F}_n$ and $H$ has at most $\sqrt{\ln n}$ non-isolated edges. 

%Let $X$ be the number of $H\in\mathcal{F}^{(\lceil 1/\delta\rceil)}_n$ that belong to $\mathbf{G}'\sim G(n,\varepsilon n^{-2/d})$, sampled independently from $\mathbf{G}$. We get 

%Letting $p=(1+\lceil 1/\delta \rceil\varepsilon)(e/n)^{2/d}$, by~\cite[Corollary 1.16]{Janson}, we get that, for $\mathbf{G}\sim G(n,p)$ and every $H\in\mathcal{F}^*_n=\mathcal{F}^*_n(\mathbf{G})$, whp $H\cup\mathbf{G}$ contains some $F\in\mathcal{F}_n$. 

Let $X$ be the number of $H\in\mathcal{F}^*_n$ that are subgraphs of $\mathbf{G}'\sim G(n,p'=\varepsilon n^{-2/d})$, sampled independently of $\mathbf{G}$. We get 
$$
\mathbb{E}X=\left|\mathcal{F}_n^*\right|\cdot p'^{\lfloor\ln n\rfloor }=(1-o(1))|\mathcal{F}_n|\cdot p'^{\lfloor\ln n\rfloor }=\omega(1).
$$
Let $\mathcal{B}$ be the set of non-empty graphs. Due to the definition~\eqref{eq:Pi_def}, the fact that every $H\in\mathcal{F}^*_n$ has at most $\sqrt{\ln n}$ non-isolated edges,~Claim~\ref{cl:J_upper_bound}, Claim~\ref{cl:simple_embeddings}, and estimate~\eqref{eq:factorials_ratio_general}, for some constant $A>0$, %\MZ{When we apply Claim~\ref{eq:J_upper_bound_trivial}, we have $|V(H)|$ there but use it for $|E(H)|$. As an option, replace $V$ with $E$ in this claim}
\begin{align*}
 \frac{\mathrm{Var}X}{(\mathbb{E}X)^2} &\leq\frac{\max_{H\in\mathcal{F}_n^*}\sum_{\ell\geq 1}\left(\Pi^H_{\mathcal{B}_{\ell}}\cdot\left|\mathcal{F}_n^*\right|\cdot p'^{\lfloor\ln n\rfloor-\ell}\right)}{\mathbb{E}X}=\max_{H\in\mathcal{F}_n^*}\sum_{\ell\geq 1}\Pi^H_{\mathcal{B}_{\ell}}p'^{-\ell}\\
& \leq
\sum_{\ell\geq 1}\sum_{x,c}(1-o(1))\frac{{\lfloor\ln n\rfloor\choose c}e^{A \ell}}{n^{2\ell/d+(1-2/d)c}}\left(\frac{1}{\varepsilon}\cdot n^{2/d}\right)^{\ell}\\
&\leq \sum_{\ell\geq 1}\sum_{c\geq 1} O(\ell) \left(\frac{e^{A+1}\lfloor\ln n\rfloor}{\varepsilon c\sqrt{n}}\right)^c e^{A\sqrt{\ln n}}\left(\frac{1}{\varepsilon}\right)^{\sqrt{\ln n}}=n^{-1/2+o(1)}.
\end{align*}
This completes the proof of Theorem~\ref{th:main2_2}, due to Chebyshev's inequality.

\begin{remark}
As noted in Section~\ref{sc:intro_strategy} (and evident from the current proof), we have $\mathrm{Var}X_F=O(({\sf E}X_F)^2)$. Consequently, Theorem~\ref{th:main2_2} follows from Paley--Zygmund inequality and \cite[Theorem 2.3]{Friedgut2}. However, this holds only when there are very few subgraphs whose edge boundary is smaller than $2d$. In particular, for $d=3$, there are at least linearly many subgraphs that do not satisfy this condition. Specifically, any pair of incident edges forms a subgraph on three vertices with an edge boundary of at most five. This observation prevents the direct application of both Riordan's second moment argument~\cite{Riordan} and the fragmentation method in its current form, necessitating new ideas.  
\end{remark}

\section{Proof of Theorem~\ref{th:second_power}}
\label{sc:KNP_conjecture_resolution}

The aim of this section is to prove the following.

\begin{theorem}
For every $\varepsilon>0$, whp $\mathbf{G}\sim G\left(n,(1+\varepsilon)\sqrt{e/n}\right)$ contains the square of a Hamilton cycle.
\label{th:square}
\end{theorem}

Theorem~\ref{th:square} immediately implies Theorem~\ref{th:second_power}.

%For a graph $G$, we denote by $c(G)$, $\ell(G),$ and $x(G)$ the number of connected components, edges, and vertices in $G$, respectively. 

Everywhere in this section $F$ is the square of a cycle on $[n]$ and $\mathcal{F}_n$ is the set of all isomorphic copies of $F$ on $[n]$.  Let us recall that we call a subgraph $H\subset F\in\mathcal{F}_n$ {\it closed} if it is the square of a path.

The proof of Theorem~\ref{th:square} is organised as follows. In Section~\ref{sc:day_0}, we define a subfamily of $\mathcal{F}_n$ whose typical fragments contain sufficiently many disjoint powers of paths, and thus is amenable to improving fragments: we fix $\omega(\sqrt{n}/\ln n)$ {\it diamonds} (4-vertex graphs obtained by removing a single edge from a 4-clique) and include into the family all $F\in\mathcal{F}_n$ that contain these diamonds at equal distances. Then we prove that typical\footnote{In the measure induced by $\mathbf{G}$.} graphs from this subfamily have fragments of size $O(\sqrt{n})$. In Section~\ref{sc:improvement}, we improve fragments, chosen from every typical $F\in\mathcal{F}_n$ according to certain distribution that is defined later in Section~\ref{sc:distribution_of_fragments}, by moving large closed subgraphs `inside' the diamonds. In the same section, we describe a probabilistic approach to assign closed subgraphs to the diamonds. It is designed in a way so that the obtained multiset of improved fragments have ``small spreads'' --- see Claim~\ref{cl:pre-images-size}. We state and prove this {\it reconstruction claim} in Section~\ref{sc:smoothed_reconstruction}. One of the main ingredients of this proof is the novel strategy of selecting a fragment, which is presented in Section~\ref{sc:distribution_of_fragments}. 

Due to Claim~\ref{cl:pre-images-size}, in the second fragmentation round presented in Section~\ref{sc:day_1} we get fragments of size $O(n^{0.3})$. In Section~\ref{sc:day_2}, we show that the usual fragmentation technique that relies on typical fragments, equipped with the reconstruction claim, allows to reduce the size of fragments to $o(\log n)$ in $O(\log\log n)$ rounds, where in each round we expose $O(pn^2/\log\log n)$ edges. Finally, in Section~\ref{sc:day_3}, we complete the proof: we show that sprinkling $\varepsilon p {n\choose 2}$ extra edges is enough to cover at least one remaining fragment. In this last step, the reconstruction claim is one of the main ingredients, as well. Crucially, the second part of the reconstruction claim --- the inequality~\eqref{eq:S_bound_small_fragments} --- is used to cover small fragments in this last step of the proof, and the proof of~\eqref{eq:S_bound_small_fragments} essentially relies on the distribution of a random fragment, presented in Section~\ref{sc:distribution_of_fragments}.

%\MZ{Some additional text explaining the trick: we take the set of all pairs $(F,F')$ and sparsify it according to some construction and a bit of randomness}
 
%Nevertheless, we may improve fragments so that all powers of paths have size $o(\ln n)$ in the following way. \MZ{Or at most $\delta\ln n$?}

\subsection{Day 0: getting fragments of size $O(\sqrt{n})$}
\label{sc:day_0}

Let 
$$
\mathbf{W}\sim G(n,m),\quad\text{ where }\quad m=\left\lfloor(1+\varepsilon)\sqrt{e/n}\cdot N\right\rfloor.
$$
In what follows, for the sake of convenience, $\mathcal{F}_n$ is the family of all rooted oriented squares of cycles on $[n]$, that is, each cycle $F\in\mathcal{F}_n$ has a specified vertex --- its root --- and is oriented, starting from the root, in one of the two directions. Clearly, $|\mathcal{F}_n|=n!$ and each $F\in\mathcal{F}_n$ induces a linear order on $[n]$; we denote the respective permutation acting on $[n]$ by $\pi_F$.

 Let $w=w(n)=\omega(1)$ be a slowly increasing function. Let 
\begin{equation}
\label{eq:d-def}
\chi:=\left\lfloor\frac{w \sqrt{n}}{\ln n}\right\rfloor
\end{equation}
and let $n=(n_1+4)+\ldots+(n_{\chi}+4)$, where $n_1,\ldots,n_{\chi}$ are positive integers that differ by at most 1. Fix a tuple $\overrightarrow{D}=(D_1,\ldots,D_{\chi})$ of disjoint diamonds ($K_4$ minus an edge) $D_1,\ldots,D_{\chi}$ on $[n]$. We then consider the family $\mathcal{F}_n(\overrightarrow{D})\subset\mathcal{F}_n$ of all $F$ such that, % \MZ{$F$ should be rooted?}
\begin{itemize}
\item for every $i\in[\chi]$, $D_i\subset F$,
\item the order of the diamonds as well as the order of vertices in the diamonds respects $\pi_F$ and the first vertex of $D_1$ coincides with the root of $F$, % the order of the vertices in $F$,
\item for every $i\in[\chi]$,  the number of vertices between $D_i$ and $D_{i+1}$ in $F$ equals $n_i$ (in the cyclic order, i.e. we here set $D_{\chi+1}:=D_1$). 
\end{itemize}
% --- each one is induced by four consecutive vertices --- that are almost evenly spread in $F$, i.e. all consecutive pairs of diamonds $D_i,D_{i+1}$ are at distance  
%\begin{equation}
% d_F(D_i,D_{i+1})=\frac{n}{d}+O(n^{1/2-\delta})=\Theta\left(\sqrt{n}\ln n\right).
%\label{eq:evenly_spread_diamonds}
%\end{equation}
%We assume that the order of the diamonds as well as the order of vertices in the diamonds is aligned with the order of the vertices in $F$. We then restrict the family $\mathcal{F}_n$ to the family $\mathcal{F}(F)$ of all $F'\in\mathcal{F}_n$ such that $\mathcal{D}(F)\cap\mathcal{D}(F')\neq\varnothing$, i.e. some tuples of evenly spread diamonds $(D_1,\ldots,D_d)$ and $(D'_1,\ldots,D'_d)$ in $F$ and $F'$, respectively, coincide. % that contain $D_1(F),\ldots,D_d(F)$ as ``ordered'' subgraphs that are evenly spread as well, in accordance to~\eqref{eq:evenly_spread_diamonds}: for all $F'\in\mathcal{F}_n(F)$ and every $i\in[d]$,
%\begin{enumerate}
%\item the order of vertices in $D_i(F)$ induced by $F'$ coincides with the order induced by $F$; 
%\item $d_{F'}(D_i(F),D_{i+1}(F))=\frac{n}{d}+O(n^{1/2-\delta})$.
%\end{enumerate}
%In particular, $F\in\mathcal{F}(F)$.

Fix $F\in\mathcal{F}_n(\overrightarrow{D})$ and $W\in{{[n]\choose 2}\choose m}$. {\it An $(F,W)$-fragment} is a subgraph $H$ of some $F'\in\mathcal{F}_n(\overrightarrow{D})$, such that $F'\subset F\cup W$ and $H=F\cap F'$. We will show that whp there exists an $(F,\mathbf{W})$-fragment $H$ of size $O(\sqrt{n})$. In particular, this fragment has to contain all diamonds $D_1,\ldots,D_{\chi}$. This will allow us to refine the fragment $H$ by spreading evenly the edges of large closed subgraphs in $H$ among the diamonds and, thus, to get rid of all the closed subgraphs of size $\Omega(\ln n)$.

We first expose the edges of $\mathbf{W}$ that belong to $D_1\cup\ldots\cup D_{\chi}$ and treat them as fixed. Since $|D_1\cup\ldots\cup D_{\chi}|=5\chi=o(\sqrt{n})$, whp $\mathbf{W}\cap(D_1\cup\ldots\cup D_{\chi})$ is empty. So, we may further condition on this event, or, in other words, assume that $\mathbf{W}$ is a uniformly random $m$-element subset of ${[n]\choose 2}\setminus(D_1\cup\ldots\cup D_{\chi})$. We also let $\mathcal{F}^*_n(\overrightarrow{D})$ be the set of all $F\setminus(D_1\cup\ldots\cup D_{\chi})$, $F\in\mathcal{F}_n(\overrightarrow{D})$.

For $F\in\mathcal{F}^*_n(\overrightarrow{D})$ and $W\subset{[n]\choose 2}\setminus(D_1\cup\ldots\cup D_{\chi})$, let $\mathcal{M}(F,W)$ be the set of all $F'\in\mathcal{F}^*_n(\overrightarrow{D})$ such that $F'\subset F\cup W$. In particular, $F\in\mathcal{M}(F,W)$. An easy counting argument shows that typically $\mathcal{M}(F,\mathbf{W})$ is large.  Let 
$$
 M:=M(t)=|\mathcal{F}^*_n(\overrightarrow{D})|{N-2n\choose m-t}/{N-5\chi\choose m+(2n-5\chi)-t}
$$
be the expected number of $F'\in\mathcal{F}^*_n(\overrightarrow{D})$ such that $F'$ belongs to a uniformly random subset of ${[n]\choose 2}\setminus(D_1\cup\ldots\cup D_{\chi})$ of size $m+(2n-5\chi)-t$. Notice that $M(t)$ is actually large for every $t$:
\begin{align*}
 M & =
 \frac{n!}{e^{O(\sqrt{n}w)}}\frac{(m-t+(2n-5\chi))^{m-t+(2n-5\chi)}\left(N-2n\right)^{N-2n}}{(m-t)^{m-t}{(N-5\chi)}^{N-5\chi}} \\
 & =
 \frac{n^n}{e^{n+O(\sqrt{n}w)}}\left(\frac{m-t}{N-5\chi}\right)^{2n-5\chi}
 \left(1-\frac{2n-5\chi}{m+(2n-5\chi)-t}\right)^{-m-(2n-5\chi)+t}\left(1+\frac{2n-5\chi}{N-2n}\right)^{-N+2n}\\
  &=
 \frac{(1+\varepsilon)^{2n}}{e^{O(\sqrt{n}w)}}\left(\frac{m-t}{m}\cdot\frac{N}{N-5d}\right)^{2n}e^{\Omega(\sqrt{n})}=(1+\varepsilon)^{(2-o(1))n}.
\end{align*}

We claim that $|\mathcal{M}(F,\mathbf{W})|=\Omega_p(M)$ and that typically most $(F,\mathbf{W})$-fragments have size $O(\sqrt{n})$. Although the conclusion of Lemma~\ref{lm:not_bad} is sufficient for our goals, below we state its stronger version which holds for symmetric families $\mathcal{F}_n$ (i.e. when all graphs in the family are isomorphic). In this version, we replace a random element of $\mathcal{F}^*_n(\overrightarrow{D})$ with any fixed element. Since the proofs are almost identical, we postpone them to Appendix~\ref{sc:appendix_claims}.

\begin{claim}
\label{cl:many_planted_copies}
Let $F\in\mathcal{F}^*_n(\overrightarrow{D})$ and let $\mathbf{W}$ be a uniformly random $m$-element subset of ${[n]\choose 2}\setminus(D_1\cup\ldots\cup D_{\chi})$. Let $\delta(n)\to 0$ as $n\to\infty$. Then, for every $t\in\{0,1,\ldots,2n-5\chi\}$,
$$
 \mathbb{P}\left(|\mathcal{M}(F,\mathbf{W})|<\delta(n)M \mid | F\cap\mathbf{W}|=t\right)\leq\delta(n).
$$ 
\end{claim}

Let $C>0$ be large enough. Set $\ell_0=\lfloor C\sqrt{n}\rfloor$. Let $\mathcal{B}$ be the property of graphs on $n$ vertices to have more than $\ell_0$ edges. We claim that~\eqref{eq:main_lm} holds.
\begin{claim}Sample a uniformly random $\mathbf{F}\in\mathcal{F}^*_n(\overrightarrow{D})$. For $\ell\in\mathbb{N}$, let 
\begin{equation}
\Pi_{\ell}:=\Prob(F\cap \mathbf{F}\text { has }\ell\text{ edges}).
\label{eq:pi_ell_definition}
\end{equation}
Then
\begin{equation}
\sum_{\ell\geq \ell_0+1}\Pi_{\ell}\cdot e^{-\frac{4n^2}{m}}\left((1+o(1))\frac{N}{m}\right)^{\ell}=\exp(-\Omega(\sqrt{n})).
\label{eq:lm}
\end{equation}
\label{cl:lm:main}
\end{claim}
The proof of Claim~\ref{cl:lm:main} relies on estimations similar to those provided in Claim~\ref{cl:main} (see Claim~\ref{cl:main_square} in Section~\ref{sc:lm:main_proof}). However, we cannot apply Claim~\ref{cl:main} here directly because of the fixed diamonds, that slightly affect our calculations. The proof occupies  Section~\ref{sc:lm:main_proof}. Despite obvious similarities and repetitions, we decided to keep the entire proof of Claim~\ref{cl:main_square} to make the proof of Theorem~\ref{th:square} self-contained.

For $\ell\in\mathbb{N}$, let $\mathcal{M}_{\ell}(F,W)$ be the set of all $F'\in\mathcal{M}(F,W)$ such that $|F'\cap F|>\ell$. Let us say that $(F,W)$ is {\it $\ell$-bad}, if $|\mathcal{M}_\ell(F,W)|>\frac{1}{\sqrt{n}}|\mathcal{M}(F,W)|$. In a similar way as in the proof of Lemma~\ref{lm:not_bad}, we derive from Claim~\ref{cl:lm:main} that $(F,\mathbf{W})$ is not $\ell_0$-bad whp. 
\begin{claim}
Let $F\in\mathcal{F}^*_n(\overrightarrow{D})$ and let $\mathbf{W}\subset{[n]\choose 2}\setminus(D_1\cup\ldots\cup D_{\chi})$ be a uniformly random $m$-subset. Then $(F,\mathbf{W})$ is not $\ell_0$-bad with probability at least $1-\exp(-\Omega(\sqrt{n}))$.
\label{cl:not_bad}
\end{claim}
That means that whp we may replace most of $F\in\mathcal{F}_n(\overrightarrow{D})$ with fragments of sizes at most $\ell_0$ (and each fragment contains all $D_i(F)$). The proof of Claim~\ref{cl:not_bad} is presented in Appendix~\ref{sc:appendix_claims}.

Let $Y_{\overrightarrow{D}}$ be the number of $F\in\mathcal{F}_n(\overrightarrow{D})$ that have an $(F,\mathbf{W})$-fragment of size at most $\ell_0+5\chi$. From Claim~\ref{cl:many_planted_copies} and Claim~\ref{cl:not_bad}, we get that $\E Y_{\overrightarrow{D}}=|\mathcal{F}_n(\overrightarrow{D})|(1-o(1))$. By Markov's inequality, for every $\delta>0$,
\begin{align}
 \Prob\left(Y_{\overrightarrow{D}}<(1-\delta)|\mathcal{F}_n(\overrightarrow{D})|\right) &=\Prob\left(|\mathcal{F}_n(\overrightarrow{D})|-Y_{\overrightarrow{D}}>\delta|\mathcal{F}_n(\overrightarrow{D})|\right)\leq\frac{\mathbb{E}\left(|\mathcal{F}_n(\overrightarrow{D})|-Y_{\overrightarrow{D}}\right)}{\delta|\mathcal{F}_n(\overrightarrow{D})|}=o(1).
\label{eq:expect_Y_partial}
\end{align}
Therefore, whp $Y_{\overrightarrow{D}}\geq(1-o(1))|\mathcal{F}_n(\overrightarrow{D})|$.

\subsection{Improving fragments}
\label{sc:improvement}

For a {\it typical} $F\in\mathcal{F}_n(\overrightarrow{D})$, let $H=H(F)=\mathbf{F}'\cap F$ be an $(F,\mathbf{W})$-fragment, where $\mathbf{F}'$ is sampled according to a certain distribution $\mathcal{Q}_{F,\mathbf{W}}$ over $\mathcal{F}_n(\overrightarrow{D})$ that will be defined later in Section~\ref{sc:distribution_of_fragments}. We will need this specific distribution, rather than the uniform distribution over all fragments, in order to prove that, on the last day, at least one fragment can be covered whp (see Section~\ref{sc:day_3}). In particular, this distribution has the following property: whp $|\mathbf{F'}\cap F|\leq \ell_0+5\chi$, where $\mathbf{F}'\sim\mathcal{Q}_{F,\mathbf{W}}$.
 %$(F,W)$ such that there is an $(F,W)$-fragment of size at most $\ell_0+5\chi$, we have that $\mathbb{P}(|\mathbf{F'}\cap F|\leq \ell_0+5\chi)=1$, for $\mathbf{F}'\sim\mathcal{Q}_{F,W}$.
   In this section, we only need this particular property and do not need any other details about the distribution $\mathcal{Q}_{F,W}$. The definition of $\mathcal{Q}_{F,W}$ is quite cumbersome, thus we delay it, so that we do not interrupt the flow of this section.  % with the property that lengths of inclusion-maximal closed subgraphs of $H$ are distributed as evenly as possible.
 Therefore, we may assume that $H$ has size at most $\ell_0+5\chi$, for almost all pairs $(F,W)$.

Recall that all graphs from $\mathcal{F}_n(\overrightarrow{D})$ contain every $D_i$, $i\in[\chi]$. Since $H=F\cap F'$ for some $F'\in\mathcal{F}_n(\overrightarrow{D})$, we get that every $D_i$ is a subgraph of $H$ as well. Moreover, since the distance in $F$ between any two $D_i$ is bigger than $2C\sqrt{n}$, we get that the connected components of $H$ that contain different $D_i$ are disjoint. Consider all inclusion-maximal closed subgraphs in $H$ with at least $10C\ln n/w$ vertices. % that do not overlap with any of the diamonds $D_i(F)$.
   For each such square of a path $v_1\ldots v_h$, we ``cut'' it into pieces of length $\mu:=\left\lfloor10 C\ln n/w\right\rfloor$ in the following way: $v_1v_2$; $v_3\ldots v_{2+\mu}$; $v_{3+\mu}\ldots v_{2+2\mu}$; $\ldots$. We let the last piece $v_{3+j\mu}v_{4+j\mu}\ldots$ to be of length at least 2 and at most $\mu+1$. %When we ``cut'' a closed subgraph, we always preserve 3 edges to the right of the first two vertices of a piece. For example, the piece $v_3\ldots v_{2+\mu}$ has all $2\mu-3$ edges induced by these vertices, as well as edges $\{v_1,v_3\}$, $\{v_2,v_3\}$, $\{v_2,v_4\}$. 
  From every such inclusion-maximal closed subgraph, remove all pieces (i.e. all the listed vertices as well as all the edges that contain at least one of them) but the first and the last one and glue the first and the last so that they make up the square of a path $(v_1v_2 v_{3+j\mu}v_{4+j\mu}\ldots)$. Let $\mathcal{P}_H=\{P_1,\ldots,P_s\}$ be the set of all pieces of length $\mu$ that we have removed ($P_1,\ldots,P_s$ respect the order $\pi_F$ induced by $F$). Clearly, the number of pieces $s$ is less than $2\ell_0/\mu<\chi/3$, where $\chi$ is defined in~\eqref{eq:d-def}. %In the same way, the number of diamonds $D_i$ that were inside closed subgraphs of length $\mu$ is also less than $d/3$.

  We then choose different $j_1,\ldots,j_s\in[\chi]$, according to some rule, that will be explained later. In particular, any $D_{j_i}$ and any removed piece do not belong to the same connected component of the fragment. %, and let $D'_1,\ldots,D'_s$ be the closest $s$ different diamonds that were not inside closed subgraphs of length $\mu$. More formally, $D'_1$ is the closest free diamond to $P_1$, and, for every $i\in\{2,\ldots,s\}$, $D'_i$ is the closest free diamond to $P_i$ that does not coincide with any of $D'_1,\ldots,D'_{i-1}$.
   For every $i\in[s]$, we insert the piece $P_i=(v^i_1\ldots v^i_{\mu})$ between the first two vertices $u_1^{j_i},u_2^{j_i}$ and the last two vertices $u_3^{j_i},u_4^{j_i}$ of the diamond $D_{j_i}=(u_1^{j_i} u_2^{j_i} u_3^{j_i} u_4^{j_i})$ making up the square of a path $(u_1^{j_i} u_2^{j_i} v^i_1 \ldots v^i_{\mu} u_3^{j_i} u_4^{j_i})$. We get a {\it smoothed fragment} $H'$ with exactly the same number of edges as $H$ has (see Figure~\ref{fig:smooth}). The crucial observation is that $H'\cup\mathbf{W}$ still has a subgraph from $\mathcal{F}_n$. Moreover, this  smoothed fragment satisfies the important `smoothness' requirement: all powers of paths that it contains have length $o(\ln n)$. % at most $4\varepsilon\ln n$. \MZ{Change $4\varepsilon$ to $o(1)$? Anyone should be other $\varepsilon$...}
   
\begin{figure}[h]
	\begin{center}
		\includegraphics[scale=.75]{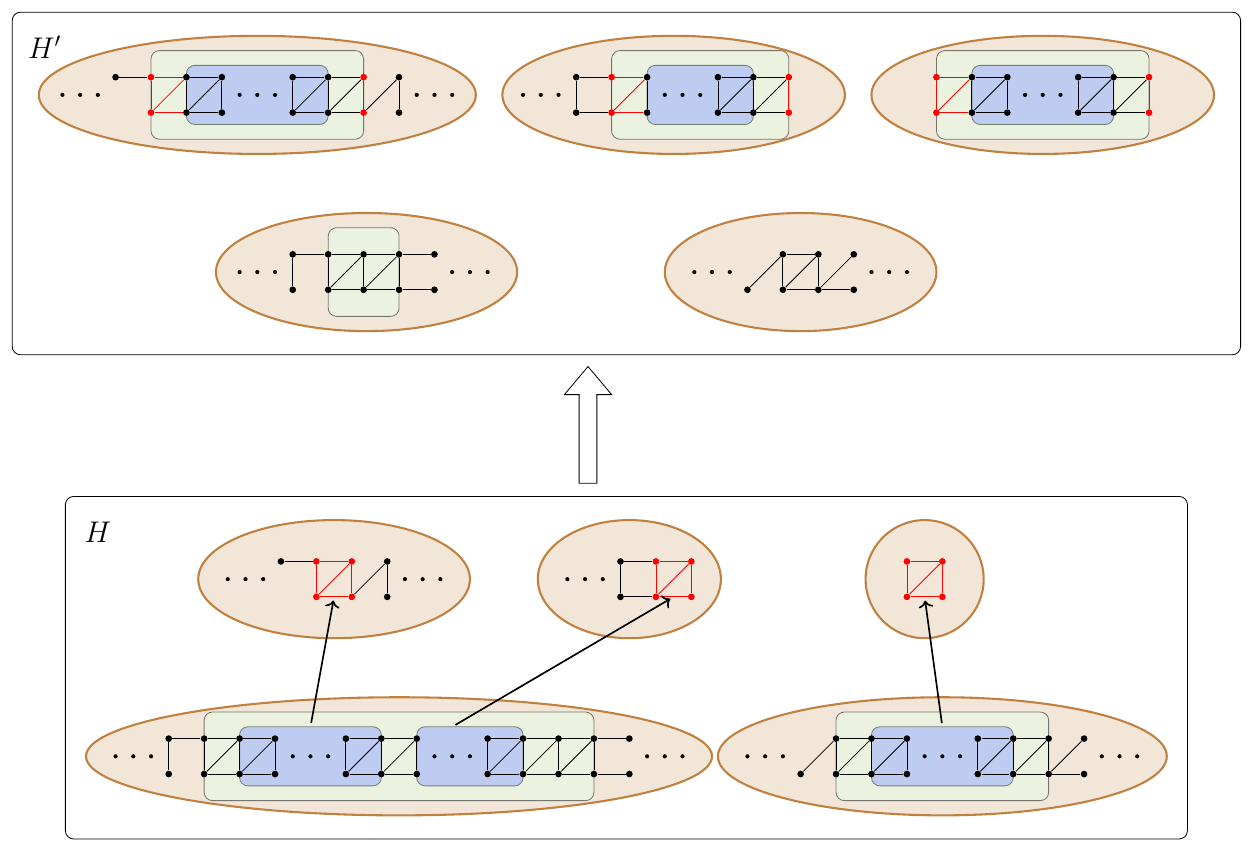}
	\end{center}
\caption{Construction of a smoothed fragment $H'$: Connected components of $H$ and $H'$ are represented by brown ovals; green squares indicate closed subgraphs with at least 6 vertices; blue squares represent pieces that were moved ``inside'' red diamonds.}\label{fig:smooth}
\end{figure}

 So, this modification of fragments allows to avoid problematic large closed subgraphs. However, it may reduce the total number of {\it different} fragments, that will eventually affect moments calculation. Let $\mathcal{H}_n=\left\{H(F), \, F\in\mathcal{F}_n(\overrightarrow{D})\right\}$ be the ($\mathbf{W}$-random) multiset of original $(F,\mathbf{W})$-fragments (we choose one fragment per almost every $F$)\footnote{More precisely, the distribution of this multiset is induced by the product measure $\times_{F}\mathcal{Q}_{F,\mathbf{W}}$.}, and let $\mathcal{H}'_n$ be the set of smoothed $(F,\mathbf{W})$-fragments. Although each $H\in\mathcal{H}_n$ maps to its smoothed version $H'\in\mathcal{H}'_n$, this correspondence $\varphi$ is not necessarily injective. Luckily, every pre-image is not large if we choose the rule of how $P_i$ are matched to $D_{j_i}$ {\it randomly}. We now explain how this matching is chosen. Consider a binomial random bipartite graph $\mathbf{B}$ with parts $V=[\chi]$ and $U=[n]$ and with edge probability $\beta:=n^{-1/3}/w$. % \MZ{For $v\in[n]$, we let $\pi_F(v)$ be the position of $v$ in $F$, i.e. $\pi_F\in S_n$ is a permutation identified by $F$.}
  Let $V'\subset V$ be the set of all indices of diamonds that do not belong to a component of $H$ that contains some $P_i$. We know that $|V'|\geq\frac{2}{3}|V|$. Then, we insert every $P_i=(v^i_1\ldots v^i_{\mu})$ into some $D_{j_i}$ so that
\begin{itemize}
\item $j_i\in V'$, i.e. $D_{j_i}$ does not belong to a connected component that also contains a removed piece,
\item there exists an edge between $j_i\in V$ and $\pi_F(v^i_1)\in U$ in $\mathbf{B}$ (recall that $v^i_1$ is the first vertex of $P_i$).
\end{itemize}
Clearly, to make it possible for all $P_i$ simultaneously, we need the graph $\mathbf{B}[V'\cup\{\pi_F(v^i_1),\,i\in[s]\}]$ to have a matching that covers all vertices $\pi_F(v^i_1),\,i\in[s]$ (recall that $s<2\chi/3\leq|V'|$). This is typically possible due to the following claim.
 
\begin{claim}
Let $(\mathbf{F},\mathbf{W},\mathbf{F}')$ be a random vector independent of $\mathbf{B}$, where 
\begin{itemize}
\item $\mathbf{F}$ is a uniformly random element of $\mathcal{F}_n(\overrightarrow{D})$, 
\item $\mathbf{W}$ is a uniformly random graph from ${{[n]\choose 2}\setminus(D_1\cup\ldots\cup D_{\chi})\choose m}$ chosen independently of $\mathbf{F}$, and 
\item $\mathbf{F}'\sim\mathcal{Q}_{\mathbf{F},\mathbf{W}}$ is a random subgraph of $\mathbf{F}\cup\mathbf{W}$ that belongs to $\mathcal{F}_n(\overrightarrow{D})$.
\end{itemize}
Consider the fragment $H:=\mathbf{F}\cap\mathbf{F}'$, the respective pieces $P_i=(v^i_1\ldots v^i_{\mu})$, $i\in[s]$, and the set $V'\subset V$, defined above. Then whp the graph $\mathbf{B}[V'\cup\{\pi_{\mathbf{F}}(v^i_1),\,i\in[s]\}]$ has a matching that covers all vertices $\pi_{\mathbf{F}}(v^i_1),\,i\in[s]$.
\label{cl:random-bipartite}
\end{claim}

\begin{proof}
Let $\Gamma$ be the set of all $(F,W,F')$ such that 
\begin{enumerate}
\item $\mathbb{P}(\mathbf{F}=F,\mathbf{W}=W,\mathbf{F}'=F')>0$;
\item $|F\cap F'|\leq\ell_0+5\chi$. 
\end{enumerate}
For $(F,W,F')\in\Gamma$ and a bipartite graph $B$  on $V\times U$, let us call the tuple $(B,F,W,F')$ {\it well-mixed}, if $B[V'\cup\{\pi_F(v^i_1),\,i\in[s]\}]$ has a matching that covers all vertices $\pi_F(v^i_1),\,i\in[s]$, where $v^i_1$ are the first vertices of the pieces of $H:=F\cap F'$. For a fixed set $U'\subset U$ of size exactly $|V'|$, whp there exists a perfect matching between $V'$ and $U'$ in $\mathbf{B}[V'\cup U']$ (see, e.g.,~\cite[Theorem 4.1]{Janson}). Therefore, recalling that the bound $|F\cap F'|\leq\ell_0+5\chi$ implies that the number of pieces that we moved in $H$ is less than $2\chi/3\leq |V'|$, we get that, for any {\it fixed} $(F,W,F')\in\Gamma$, whp $(\mathbf{B},F,W,F')$ is well-mixed. Also, due to Claim~\ref{cl:many_planted_copies}, Claim~\ref{cl:not_bad}, and the described property of the distribution $\mathcal{Q}_{F,W}$, whp $(\mathbf{F},\mathbf{W},\mathbf{F}')\in\Gamma$. Then,
\begin{align*}
 \mathbb{P} & ((\mathbf{B},\mathbf{F},\mathbf{W},\mathbf{F}')\text{ is well-mixed})\\
 &\geq\sum_{(F,W,F')\in\Gamma}
 \mathbb{P}((\mathbf{B},F,W,F')\text{ is well-mixed}\mid\mathbf{F}=F,\mathbf{W}=W,\mathbf{F}'=F')\\
 &\quad\quad\quad\quad\quad\quad\quad\quad\quad\quad\quad\quad
 \quad\quad\quad\quad\quad\quad
 \times\mathbb{P}(\mathbf{F}=F,\mathbf{W}=W,\mathbf{F}'=F')\\
 &=
 \sum_{(F,W,F')\in\Gamma}
 \mathbb{P}((\mathbf{B},F,W,F')\text{ is well-mixed})\times\mathbb{P}(\mathbf{F}=F,\mathbf{W}=W,\mathbf{F}'=F')\\
 &=(1-o(1))\mathbb{P}((\mathbf{F},\mathbf{W},\mathbf{F}')\in\Gamma)=1-o(1),
\end{align*}
completing the proof. 
 \end{proof}
From Claim~\ref{cl:random-bipartite}, we immediately get that there exists a bipartite graph $B$ on $V\times U$ so that every vertex in $V$ has degree $(1+o(1))n^{2/3}/w$ (due to the Chernoff bound, say) and that whp $B[V'\cup\{\pi_{\mathbf{F}}(v^i_1),\,i\in[s]\}]$ has a matching that covers all vertices $\pi_{\mathbf{F}}(v^i_1),\,i\in[s]$, that correspond to $H=\mathbf{F}\cap\mathbf{F}'$, where $(\mathbf{F},\mathbf{F}')$ is defined in the statement of the claim. {\it Fix such a bipartite graph~$B$.}
 
Having this rule of matching the diamonds with the pieces of closed subgraphs, we will prove in Section~\ref{sc:smoothed_reconstruction} that $\varphi$ is close enough to an injection. In what follows, we will consider only those $(F,W)$ that generate a fragment that satisfies the matching rule. Formally, for $F\in\mathcal{F}_n(\overrightarrow{D})$ and $W\in{{[n]\choose 2}\setminus(D_1\cup\ldots\cup D_{\chi})\choose m}$, we say that $F'\in\mathcal{F}_n(\overrightarrow{D})$ is {\it $(F,W)$-nice}, if
\begin{itemize}
\item $F'\subset F\cup W$,
\item $|F'\cap F|\leq\ell_0+5\chi$,
\item the graph $B[V\cup\{\pi_F(v^i_1),\,i\in[s]\}]$ has a matching that covers all vertices $\pi_F(v^i_1),\,i\in[s]$, defined by the fragment $H=F\cap F'$.
\end{itemize}
Let $\Sigma$ be the (random) set of all pairs $\left(F\in\mathcal{F}_n(\overrightarrow{D}),W\in{{[n]\choose 2}\setminus(D_1\cup\ldots\cup D_{\chi})\choose m}\right)$ such that the random fragment $\mathbf{F}'\sim\mathcal{Q}_{F,W}$ is $(F,W)$-nice, where all fragments are chosen independently. Due to the choice of $B$ and due to Claims~\ref{cl:many_planted_copies},~\ref{cl:not_bad},~and~\ref{cl:random-bipartite}, whp $(\mathbf{F},\mathbf{W})\in\Sigma$ (in the $(\mathbf{F},\mathbf{W},\mathbf{F}')$-measure). Indeed,
\begin{align}
\mathbb{P}((\mathbf{F},\mathbf{W})\in\Sigma)  =
\mathbb{P}(\mathbf{F}'\text{ is $(\mathbf{F},\mathbf{W})$-nice})=1-o(1),
\label{eq:most_pairs_are_nice}
\end{align}
where $(\mathbf{F},\mathbf{W},\mathbf{F}')$ is defined in the statement of Claim~\ref{cl:random-bipartite}.

We now notice that the there exists a natural bijection between {\it disjoint} families $\mathcal{F}_n(\overrightarrow{D})$, $\mathcal{F}_n(\overrightarrow{D'})$ for different tuples of diamonds $\overrightarrow{D}$ and $\overrightarrow{D'}$. In particular, for $Y:=\sum_{\overrightarrow{D}}Y_{\overrightarrow{D}}$\footnote{Recall that the random variable $Y_{\overrightarrow{D}}$ is defined in Section~\ref{sc:day_0} as the number of $F\in\mathcal{F}_n(\overrightarrow{D})$ that have an $(F,\mathbf{W})$-fragment of size at most $\ell_0+5\chi$.} we get that 
$$
 \mathbb{E}Y=\sum_{\overrightarrow{D}}\E Y_{\overrightarrow{D}}=(1-o(1))\sum_{\overrightarrow{D}}|\mathcal{F}_n(\overrightarrow{D})|=(1-o(1))|\mathcal{F}_n|.
$$
In the same way as in~\eqref{eq:expect_Y_partial}, we get that whp  $Y\geq(1-o(1))|\mathcal{F}_n|$. Since any $F\in\mathcal{F}_n$ belongs to $\mathcal{F}_n(\overrightarrow{D})$ for some uniquely defined $\overrightarrow{D}$, we get the definition of $\Sigma$ naturally extends to pairs $(F,W)$ for an arbitrary $F\in\mathcal{F}_n$. In particular, we require $(F,W)\in\Sigma$ to satisfy $W\cap(D_1\cup\ldots\cup D_{\chi})=\varnothing$, where $\overrightarrow{D}=\overrightarrow{D}(F)=(D_1,\ldots,D_{\chi})$ is defined by $F$. Let a uniformly random element $\mathbf{F}$ of $\mathcal{F}_n$,  uniformly random elements $\mathbf{F}_{\overrightarrow{D}}\in\mathcal{F}_n(\overrightarrow{D})$, a uniformly random $m$-subset $\mathbf{W}\in{{[n]\choose 2}\choose m}$, and uniformly random $m$-subsets $\mathbf{W}_{\overrightarrow{D}}\in{{[n]\choose 2}\setminus(D_1\cup\ldots\cup D_{\chi})\choose m}$  be chosen independently.  Due to~\eqref{eq:most_pairs_are_nice},
\begin{align*}
 \mathbb{P}((\mathbf{F},\mathbf{W})\in\Sigma) 
 &=\sum_{\overrightarrow{D}}\mathbb{P}((\mathbf{F},\mathbf{W})\in\Sigma\mid\mathbf{F}\in\mathcal{F}_n(\overrightarrow{D}))\cdot\mathbb{P}(\mathbf{F}\in\mathcal{F}_n(\overrightarrow{D}))\\
 &=\sum_{\overrightarrow{D}}\mathbb{P}((\mathbf{F}_{\overrightarrow{D}},\mathbf{W})\in\Sigma)\cdot\mathbb{P}(\mathbf{F}\in\mathcal{F}_n(\overrightarrow{D})).
\end{align*} 
Due to symmetry, for a fixed $\overrightarrow{D}$, we get
 \begin{align}
 \label{eq:Sigma_W}
 \mathbb{P}((\mathbf{F},\mathbf{W})\in\Sigma) 
 &=\mathbb{P}((\mathbf{F}_{\overrightarrow{D}},\mathbf{W})\in\Sigma,\mathbf{W}\cap(D_1\cup\ldots\cup D_{\chi})=\varnothing)\notag\\
    &=(1-o(1))\cdot\mathbb{P}((\mathbf{F}_{\overrightarrow{D}},\mathbf{W})\in\Sigma\mid\mathbf{W}\cap(D_1\cup\ldots\cup D_{\chi})=\varnothing)\notag\\
      &=(1-o(1))\cdot\mathbb{P}((\mathbf{F}_{\overrightarrow{D}},\mathbf{W}_{\overrightarrow{D}})\in\Sigma)\stackrel{\eqref{eq:most_pairs_are_nice}}=1-o(1).
\end{align}
%In particular, for asymptotically almost all $W$, whp $(\mathbf{F},W)\in\Sigma$. 
In Section~\ref{sc:smoothed_reconstruction}, we will show that we may control the number of ``pre-images'' of any subgraph of a smoothed $(F,W)$-fragment for a $(F,W)\in\Sigma$ whp, in the product measure $\mathcal{Q}:=\times\mathcal{Q}_{F,W}$, where the product is over $F\in\mathcal{F}_n$ and $W\in{{[n]\choose 2}\choose m}$. In order to make it work, we first define the measure $\mathcal{Q}_{F,W}$ in Section~\ref{sc:distribution_of_fragments}.

\subsection{Distribution of fragments}
\label{sc:distribution_of_fragments}

Let $\delta<1/2$ and let $\chi'\leq n^{\delta}$ be a non-negative integer. Fix disjoint diamonds $D_1,\ldots,D_{\chi+\chi'}$ on $[n]$. Let $\overrightarrow{D}'=(D_1,\ldots,D_{\chi+\chi'})$ be the concatenation of $\overrightarrow{D}=(D_1,\ldots,D_{\chi})$ with $(D_{\chi+1},\ldots,D_{\chi+\chi'})$, where $\overrightarrow{D}$ is as in Section~\ref{sc:day_0}. We also fix some specific different positions $m_1,\ldots,m_{\chi'}\in[n]$ and denote $\overrightarrow{m}:=(m_1,\ldots,m_{\chi'})$. We then consider the family $\mathcal{F}^{m_1,\ldots,m_{\chi'}}_n(\overrightarrow{D}')\subset\mathcal{F}_n(\overrightarrow{D})$ of all $F\in\mathcal{F}_n(\overrightarrow{D})$ such that, for every $i\in[\chi']$,
\begin{enumerate}
\item the leftmost vertex of the diamond $D_{\chi+i}$ occupies the position $m_i$ in the ordering~$\pi_F$~of~$[n]$;
\item the diamond $D_{\chi+i}$ belongs to $F$ and the order of its vertices respects $\pi_F$.
\end{enumerate}
 For $W\in{{[n]\choose 2}\choose m}$, we call {\it a $(\overrightarrow{D}',\overrightarrow{m})$-grounded $(F,W)$-fragment} any graph $F'\cap F$ such that $F'\in\mathcal{F}^{m_1,\ldots,m_{\chi'}}_n(\overrightarrow{D})$ and $F'\subset F\cup W$. The proof of the following claim is identical to the proof of Claim~\ref{cl:not_bad}; therefore, we omit it.

\begin{claim}
Let $F\in\mathcal{F}^{m_1,\ldots,m_{\chi'}}_n(\overrightarrow{D}')$ and let $\mathbf{W}\subset{[n]\choose 2}$ be a uniformly random $m$-subset. Then, with probability at least $1-\exp(-\Omega(\sqrt{n}))$ in the $\mathbf{W}$-measure, we have that $|\mathbf{F}'\cap F|\leq\ell_0+5\chi$ whp in the $\mathbf{F}'$-measure, where $\mathbf{F}'$ is a uniformly random $(\overrightarrow{D}',\overrightarrow{m})$-grounded $(F,\mathbf{W})$-fragment.
\label{cl:grounded_fragments}
\end{claim}

We now fix $F\in\mathcal{F}_n(\overrightarrow{D})$, $W\in{{[n]\choose 2}\choose m}$, and define explicitly a random fragment $\mathbf{F}^*=\mathbf{F}^*(F,W)$ that has the desired distribution $\mathcal{Q}_{F,W}$. Let $\mathbf{F}'$ be a uniformly random subgraph of $F\cup W$ from $\mathcal{F}_n(\overrightarrow{D})$ satisfying $|\mathbf{F}'\cap F|\leq\ell_0+5\chi$. If such a subgraph does not exist, we set $\mathbf{F}^*=F$. Otherwise, if $\mathbf{F}'\cap F$ does not have closed subgraphs with at least $\sqrt{n}/\ln n$ vertices, we let $\mathbf{F}^*=\mathbf{F}'$. If neither of the above two conditions is satisfied, set $\mathbf{F}^{(1)}:=\mathbf{F}'$. 

%Let $C_1,\ldots,C_{\chi'}$ be all the inclusion-maximal closed subgraphs of $F\cap\mathbf{F}^{(1)}$ of length at least $\sqrt{n}/\ln n$. If $\chi'>n^{\delta}$, we let $\mathbf{F}^*=F$. Otherwise, for every $i\in[\chi']$, let $D_{\chi+i}:=(u^i_1u^i_2u^i_3u^i_4)$ be the diamond in $F$ that crosses $C_i$ in its leftmost vertices $u^i_3,u^i_4$. We also let $m_i$ be the position of $u^i_1$  in the ordering $\pi_F$ of $[n]$. We now resample $\mathbf{F}'$ uniformly at random from the set of all $(\overrightarrow{D}',\overrightarrow{m})$-grounded $(F,W)$-fragments. If $|\mathbf{F}'\cap F|>\ell_0+5\chi$, we set $\mathbf{F}^*=F$. Otherwise, if $\mathbf{F}'$ does not have closed subgraphs with at least $\sqrt{n}/\ln n$ vertices, we let $\mathbf{F}^*=\mathbf{F}'$. If neither of the above two conditions is satisfied, set $\mathbf{F}^{(2)}:=\mathbf{F}'$. 

We then proceed by induction. At step $1\leq j\leq n^{\delta}/\ln^2 n$, we have that either $\mathbf{F}^*$ has been defined, or graphs $\mathbf{F}^{(1)},\ldots,\mathbf{F}^{(j)}$ have been sampled. If $j+1>n^{\delta}/\ln^2 n$, then halt and let $\mathbf{F}^*$ be chosen uniformly at random from $\{\mathbf{F}^{(1)},\ldots,\mathbf{F}^{(j)}\}$. Otherwise, let $C_1,\ldots,C_{\chi'}$ be all the inclusion-maximal closed subgraphs of length at least $\sqrt{n}/\ln n$ that belong to one of the fragments $F\cap\mathbf{F}^{(1)},\ldots,F\cap\mathbf{F}^{(j)}$. %If $\chi'> j\ln^2 n$, we let $\mathbf{F}^*=F$. Otherwise, 
 For every $i\in[\chi']$, let $D_{\chi+i}:=(u^i_1u^i_2u^i_3u^i_4)$ be the diamond in $F$ that crosses $C_i$ in its leftmost vertices $u^i_3,u^i_4$ and let $m_i$ be the position of $u^i_1$  in the ordering $\pi_F$ of $[n]$.  Resample $\mathbf{F}'$ uniformly at random from $\mathcal{F}^{m_1,\ldots,m_{\chi'}}_n(\overrightarrow{D}')$ under the condition that $|\mathbf{F}'\cap F|\leq\ell_0+5\chi$. If such a graph does not exist, we set $\mathbf{F}^*=F$. Otherwise, if $\mathbf{F}'\cap F$ does not have closed subgraphs with at least $\sqrt{n}/\ln n$ vertices, we let $\mathbf{F}^*=\mathbf{F}'$. If neither of the above two conditions is satisfied, set $\mathbf{F}^{(j+1)}:=\mathbf{F}'$.

\begin{figure}[h]
	\begin{center}
		\includegraphics[scale=.85]{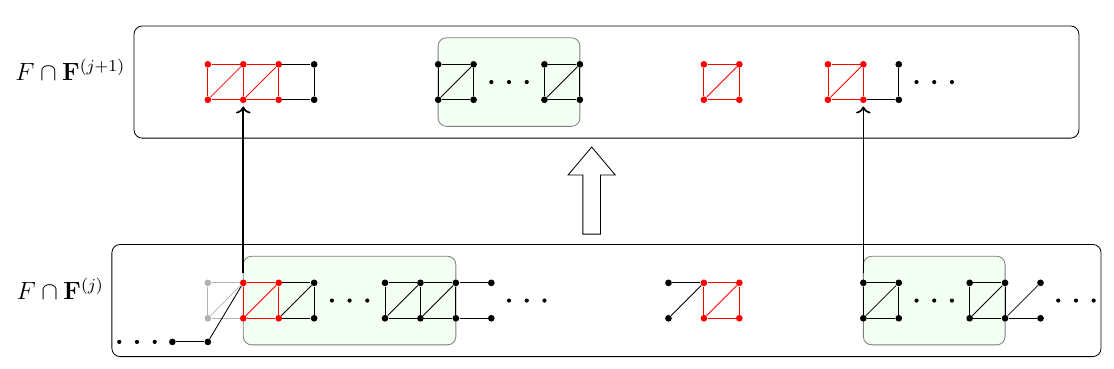}
	\end{center}
\caption{Induction step in the process of construction of $\mathbf{F}^*$: Green squares represent inclusion-maximal closed subgraphs of length at least $\sqrt{n}/\ln n$; their leftmost vertices become rightmost vertices of the diamonds planted in $\mathbf{F}^{(j+1)}$. The order of vertices in the planted diamonds respects $\pi_F$ --- in particular, grey edges of $F$ do not coincide with black edges of $\mathbf{F}^{(j)}$; the former grey edges are coloured red in $\mathbf{F}^{(j+1)}$.}\label{fig:...}
\end{figure}

This completes the construction of $\mathbf{F}^*$. Let us now prove that, in particular, its distribution satisfies the required property that we used in Section~\ref{sc:improvement} --- whp $|F\cap\mathbf{F}^*|\leq\ell_0+5\chi$. Let $\mathbf{W}\subset{[n]\choose 2}$ be a uniformly random $m$-subset. Due to Claim~\ref{cl:not_bad}, with probability at least $1-\exp(-\Omega(\sqrt{n}))$, in the $\mathbf{W}$-measure, after the first step $j=1$, we either proceed to the next step or halt and output $\mathbf{F}^*$ that satisfies 
\begin{itemize}
\item[P1] $|\mathbf{F}^*\cap F|\leq\ell_0+5\chi$,
\item[P2] $\mathbf{F}^*$ does not have closed subgraphs with at least $\sqrt{n}/\ln n$ vertices.
\end{itemize}
Moreover, if we proceed to step $j=2$, then with probability at least $1-\exp(-\Omega(\sqrt{n}))$, in the $\mathbf{W}$-measure, the algorithm samples $\mathbf{F}^{(1)}$ that has $|\mathbf{F}^{(1)}\cap F|\leq\ell_0+5\chi$ (deterministically) and, therefore, the fragment $\mathbf{F}^{(1)}\cap F$ has less than $\ln^2 n$ inclusion-maximal closed subgraphs of length at least $\sqrt{n}/\ln n$.

Then, for an arbitrary $1\leq j\leq n^{\delta}/\ln^2 n$, assume that with probability at least $1-\exp(-\Omega(\sqrt{n}))$ in the $\mathbf{W}$-measure, at step $j$ either the algorithm halts and outputs $\mathbf{F}^*$ that satisfies P1 and P2, or it samples $\mathbf{F}^{(1)},\ldots,\mathbf{F}^{(j)}$ such that,
\begin{itemize}
\item for every $i\in[j]$, $\mathbf{F}^{(i)}$ satisfies P1 and, therefore, $|F\cap \mathbf{F}^{(i)}|$ has less than $\ln^2 n$ inclusion-maximal closed subgraphs of length at least $\sqrt{n}/\ln n$ and,
\item crucially, the $j$ sets of all inclusion-maximal closed subgraphs of length at least $\sqrt{n}/\ln n$ in $F\cap \mathbf{F}^{(1)},\ldots, F\cap \mathbf{F}^{(j)}$ do not overlap.
\end{itemize}
These sets identify the diamonds $C_1,\ldots,C_{\chi'}$, where $\chi'<j\ln^2 n<n^{\delta}$ whp. Due to Claim~\ref{cl:grounded_fragments} and the union bound over the choice of $\chi'$ diamonds, with probability at least $1-\exp(-\Omega(\sqrt{n}))$ in the $\mathbf{W}$-measure, we get that either the algorithm halts at step $j+1$ and outputs $\mathbf{F}^*$ that satisfies P1 and P2, or it samples at this step $\mathbf{F}^{(j+1)}$ that satisfies $|\mathbf{F}^{(j+1)}\cap F|\leq\ell_0+5\chi$ and proceeds to the next step.

We conclude that whp, in the $\mathbf{W}$-measure, the algorithm outputs $\mathbf{F}^*\sim\mathcal{Q}_{F,\mathbf{W}}$ that satisfies property P1, as required. % $|\mathbf{F}'\cap F|\leq\ell_0+5\chi$.  Moreover, by the definition of this distribution,
 Moreover, the algorithm is designed in such a way that whp, in the $\mathbf{W}$-measure, 
\begin{itemize}
\item[P3]
 for every set $\mathcal{X}$ of disjoint closed subgraphs of $F$ on at least $\sqrt{n}/\ln n$ vertices, the $\mathcal{Q}_{F,\mathbf{W}}$-probability that the set of inclusion-maximal closed subgraphs of $\mathbf{F}^*\cap F$ on at least $\sqrt{n}/\ln n$ vertices contains $\mathcal{X}$ is at most $\left\lfloor\frac{\ln^2 n}{n^{\delta}}\right\rfloor$.
\end{itemize}
  We call a pair $(F,W)$ with $\mathcal{Q}_{F,W}$ satisfying P3 {\it separating}. The property of being separating will be applied in Section~\ref{sc:day_3} to show that we may cover at least one fragment of size $o(\log n)$ whp.

\subsection{Reconstruction from a smoothed fragment}
\label{sc:smoothed_reconstruction}

%Let us fix any $W\in{{[n]\choose 2}\choose m}$. 

For every pair $\left(F\in\mathcal{F}_n,W\in{{[n]\choose 2}\choose m}\right)$,  let $\mathbf{F}'\sim\mathcal{Q}_{F,W}$ be generated independently (mutually over all such pairs $(F,W)$). Recall the definition of $\Sigma$ from Section~\ref{sc:improvement}. For $W\in{{[n]\choose 2}\choose m}$, let $\Sigma_W$ be the set of all $F\in\mathcal{F}_n$ such that $(F,W)\in\Sigma$ is separating. Due to~\eqref{eq:Sigma_W} and the conclusion of Section~\ref{sc:distribution_of_fragments}, $(\mathbf{W}\times\mathcal{Q})$-whp, for (asymptotically) almost all $F$, we have that $F\in\Sigma_{\mathbf{W}}$.

%Now, for every pair $(F,W)$ such that $F\in\Sigma_W$, let $\mathbf{F}'\sim\mathcal{Q}_{F,W}$ be generated independently (mutually over all such pairs $(F,W)$).   Let $\Sigma_W$ be the set of all $F\in\mathcal{F}_n$ such that $(F,W)\in\Sigma$ is separating. Due to~\eqref{eq:Sigma_W} and the conclusion of Section~\ref{sc:distribution_of_fragments}, whp, for (asymptotically) almost all $F$, we have that $F\in\Sigma_{\mathbf{W}}$. 

%Now, for every pair $(F,W)$ such that $F\in\Sigma_W$, let $\mathbf{F}'\sim\mathcal{Q}_{F,W}$ be generated independently (mutually over all such pairs $(F,W)$).

 We now fix $W\in{{[n]\choose 2}\choose m}$ and, for every $F\in\mathcal{F}_n$, generate $\mathbf{F}'\sim\mathcal{Q}_{F,W}$. If $F\in\Sigma_W$, we let % such that $(\mathbf{F},W)\in\Sigma$ whp.
% For every $F\in\Sigma_W$, consider an $(F,W)$-nice $F':=F'(F)\in\mathcal{F}_n$.
  $H(F):=F\cap\mathbf{F}'$, and let $H'(F)$ be the smoothed version of $H(F)$ (with respect to the matching rule $B$). For brevity, we denote $\ell_{\chi}:=\ell_0+5\chi.$
 
\begin{claim}
\label{cl:pre-images-size}
%For $F\in\mathcal{F}_n$ and $F'\in\mathcal{F}(F)$ so that
%\begin{itemize}
%\item $|F'\cap F|\leq\ell_0$,
%\item $B[V\cup\{v^i_1,\,i\in[s]\}]$ has a matching that covers all vertices $v^i_1,\,i\in[s]$ defined by $F\cap F'$,
%\end{itemize}
%Let $\mathbf{F}$ be a uniformly random element of $\mathcal{F}_n$. Fix $W$ such that $(\mathbf{F},W)\in\Sigma$ whp.
%  Let $\Sigma_W$ be the set of all $F\in\mathcal{F}_n$ such that $(F,W)\in\Sigma$.
  % Let $\mathcal{H}'_n:=\{H'(F):\, F\in\Sigma_W\}$. Let $\varphi$ map $H(F)$ to $H'(F)$ for every $F\in\Sigma_W$. 
 For any graph $S$, let $\mathcal{S}(S)$ be the (random) set of all graphs $F\in\Sigma_W$ such that $H'(F)$ contains $S$ as a subgraph. Let
$c:=c(S)$, $\ell:=\ell(S)$, $x:=x(S)$, and $\sigma:=\sigma(S)=2x(S)-3c(S)-\ell(S)$. Then
\begin{equation}
 |\mathcal{S}(S)|\leq 64^{\ell}(n-x+c)!\max\left\{n^{c/6}{\ell_{\chi}\choose \min\{\sigma,\lfloor\ell_{\chi}/2\rfloor\}},{\ell_{\chi}\choose \min\{c+\sigma,\lfloor\ell_{\chi}/2\rfloor\}}\right\} \quad\text{ almost surely}.
\label{eq:S_bound_large_fragments}
\end{equation}
Moreover, the bound
\begin{equation}
 |\mathcal{S}(S)|\leq 64^{\ell}(n-x+c)!\frac{(2\ell_{\chi})^{c+\sigma}}{\min\left\{(\ln n)^{c+\sigma},n^{1/4}\right\}}
\label{eq:S_bound_small_fragments}
\end{equation}
holds with $\mathcal{Q}$-probability $1-e^{-\Omega(n^{1/6})}$ (recall that $\mathcal{Q}=\times\mathcal{Q}_{F,W}$).
%for all $H'\in\mathcal{H}'_n$ and $S'\subset H'$. \MZ{Should be updated}
%Let $S'$ be a graph such that $S'\subset H'$ for some $H'\in\mathcal{H}'_n$. Let $\mathcal{S}(S')$ be the set of all graphs $S$ such that there exist $H\in\mathcal{H}_n$ satisfying the following: $S\subset H$ and the restriction of $\varphi$ to $S$ maps $S$ to $S'$. Let $\mathcal{I}_c(S')$ be the set of isomorphism classes of graphs from $\mathcal{S}(S')$ with exactly $c$ connected components. Then whp, for all $S'\subset H'$, $H'\in\mathcal{H}'_n$, $|\mathcal{I}_{c'}(S')|\leq C^{c(S')+c'}$ for every $c'\geq c(S')$, and $\mathcal{I}_{c'}(S')=\varnothing$ for every $c'<c(S')$.
\end{claim}

%\MZ{Some explanation why it is good}

%We defer the proof of Claim~\ref{cl:pre-images-size} until Section~\ref{sc:proof-claim-pre-images-size}. This claim is the crucial part of our argument and its proof relies on certain non-trivial properties of typical fragments that we discuss in Section~\ref{sc:fragments-properties}. The main observation that allows to prove Claim~\ref{cl:pre-images-size} is that a typical fragment have sizes of connected components evenly distributed. In particular, most inclusion-maximal closed subgraphs in $H=F\cap F'$ have bounded sizes and the total number of edges in closed subgraphs of lengths at least $\mu$ is at most $n^{1/2-\delta}$ for some $\delta=\delta(\varepsilon)>0$. Therefore,  there at most $O(n^{1/2-\delta}/\ln n)$ pieces in $H$ that were moved. It implies that, for every moved piece $P$, the diamond that it moved into is at distance at most $n^{1-\delta}$ from $P$ in $F$.

%\section{Proof of Claim~\ref{cl:pre-images-size}}
%\label{sc:proof-claim-pre-images-size}

\begin{proof}%[Proof of Claim~\ref{cl:pre-images-size}.]
%For brevity, set $c:=c(S)$, $\ell:=\ell(S)$, $x:=x(S)$, and $\sigma:=\sigma(S)$.
  Let $T_1,\ldots,T_{c}$ be the connected components of $S$. We fix an arbitrary root in every connected component of $S$ and order edges in every connected component in a way such that edges that are closer to the root appear earlier. Let 
\begin{itemize}
%\item $\gamma$ be an order on the set of connected components in $S'$; 
\item $\alpha:E(S)\to[4]$ be an arbitrary function;
\item $A\subset V(S)$ be a set such that each $T_i$ has at most two vertices in common with $A$;
\item $A'\subset V(S)$ be an arbitrary set. % be a set such that each $T_i$ has at most two vertices in common with $A$.
\end{itemize}
The tuple $(\alpha,A,A')$ identifies two non-negative integers $a,a'$ such that $a\leq\min\{c,\chi\}$ and $a+a'\leq c+\sigma$ that will be explicitly defined later. We also let
\begin{itemize}
\item $\tau_1:[a]\to[\lfloor 2n^{2/3}/w\rfloor]$, $\tau_2:[a]\to[\lfloor \ln n/4\rfloor]$ %, $\tau_3:[\sigma]\to[\ell_{\chi}]$
  be arbitrary functions;
% be two sets of vertices from $S'$ such that each connected components has at most one element from $A_1$ and at most one element from $A_2$;
\item $\rho:[a]\to[\chi]$ be an injection;
\item $f$ be a decomposition of $\ell_{\chi}$ into $a'$ positive integers;
\item $\pi:[n-x+c-a]\to[n-x+c-a]$ be a bijection.
\end{itemize}
Let us show that the tuple $\mathbf{x}:=(A,A',\alpha,\tau_1,\tau_2,f,\rho,\pi)$ identifies a graph $F\in \Sigma_W$ such that $S\subset H'(F)$. As we will see, the latter implies the desired assertion. Below, we describe the algorithm of reconstruction of~$F$ from $\mathbf{x}$.

% there exist $H\in\mathcal{H}_n$ and $S\subset H$ such that the bijection $V(S)\to V(\tilde S)$ that preserves the order of vertices is an isomorphism between $S$ and $\tilde S$. Indeed, first, 
\begin{enumerate}
%\item Order the connected components of $S'$ according to $\gamma$. 

\item Reconstruct the order of vertices (and their positions in the improved fragment $H'(F)$, subject to the positions of roots) in every component of $S$ according to $\alpha$: For every $i\in[c]$, start from the root of $T_i$ and follow the edges according to the fixed order; at every step, one of the ends of the current edge $e$ is already placed, so there are at most four possibilities to place this edge, and then $\alpha(e)$ defines its position. With slight abuse of notation, in what follows we denote by $\alpha$ the reconstructed order of vertices in every component of $S$.

\item Choose the pieces in $S$ that were moved in $H(F)$: For every connected component $T_i\subset S$, if %it contains at least two vertices of $A_1$ or at least two vertices of $A_2$, then halt since there is no $F$ for the given tuple $\mathbf{x}$. If 
 $|A\cap T_i|\leq 1$, then $T_i$ does not contain vertices of a piece that was moved. Otherwise, let $A\cap V(T_i)=\{v_i,u_i\}$, where $v_i<_{\alpha}u_i$, according to the linear order $\alpha$ of vertices in $T_i$, identified at the previous step. Let the unique ``piece'' $P_i\subset T_i$ have all edges of $T_i$ induced by the set of vertices that are $\alpha$-between $v_i$ and $u_i$.
 
 \item Choose pairs of edges in $S$ that were glued in $H(F)$: For every connected component $T_i\subset S$, let $A'\cap T_i=\{v^i_1,\ldots,v^i_{t_i}\}$, where $v^i_1<_{\alpha}\ldots<_{\alpha}v^i_{t_i}$. We also assume that $v^i_{t_i}$ does not equal to the $\alpha$-biggest vertex in $T_i$. We treat each vertex $v^i_j$ as the biggest vertex of $S$ that belongs to the smaller edge of the $j$-th pair of edges that was glued. For $j\in[t_i+1]$, let $P^i_j$ be the subgraph of $T_i$ induced by the set of all vertices that are $\alpha$-bigger than $v^i_{j-1}$ and do not $\alpha$-exceed $v^i_j$. Here, $v^i_0=0$ and $v^i_{t_i+1}$ is the $\alpha$-biggest vertex of $T_i$. % in the linear order of vertices in $T_i$.
 
%  if %it contains at least two vertices of $A_1$ or at least two vertices of $A_2$, then halt since there is no $F$ for the given tuple $\mathbf{x}$. If 
% $|A\cap T_i|=1$, then let $A\cap V(T_i)=\{v_i\}$ and let $P'_i$ be the subgraph of $S$ induced by the set of all vertices that do not exceed $v_i$ in the linear order of vertices in $T_i$.

\item Let $\mathcal{I}$ be the set of all $i\in[c]$ such that $P_i$ is defined and $P_i\neq T_i$. We also let $a=|\mathcal{I}|$. For every $i\in\mathcal{I}$, we choose the index $\eta(i)$ of the respective diamond in $F$ according to the injection $\rho$.  Then, we move pieces $P_i$,  $i\in\mathcal{I}$, to their original places in $F$. For every $i\in\mathcal{I}$, the position $\eta(i)$ of the respective diamond in $[\chi]$ is identified. Therefore, the position of the first vertex of the entire piece in $F$ containing $P_i$ is adjacent to $\eta(i)$ in $B$, where $\eta(i)$ has degree at most $(1+o(1))n^{2/3}/w$. Then $v_i$ is at most $(0.1\cdot\ln n)$-far from this vertex in $F$ and we may identify its position in $F$ according to $(\tau_1(i),\tau_2(i))$.

\item Let $\mathcal{I}'$ be the set of all $i\in[c]$ such that $A'\cap T_i$ is non-empty and $v^i_{t_i}$ is not the $\alpha$-biggest vertex of $T_i$. We also let $a'=\sum_{i\in\mathcal{I}'}|A'\cap T_i|$, i.e. $a'=\sum_{i\in\mathcal{I}'}t_i$. We note that $a+a'\leq c+\sigma$ since 1) the moved pieces and the glued edges do not belong to the same connected component in the smoothed fragment $H'(F)$, according to its construction, 2) different pairs of glued edges do not belong to the same inclusion-maximal closed subgraph of the fragment $H(F)$. For every $i\in\mathcal{I}'$ and every $j\in[t_i]$, we choose the distance between the biggest vertex in $P^i_j$ and the smallest vertex in $P^i_{j+1}$, according to $f$, that identifies the positions of all $P^i_2,\ldots,P^i_{t_i+1}$, $i\in\mathcal{I}'$, as soon as the positions of $P^i_1$ are fixed.  %Note that $f$ is indeed sufficient: since pairs of different glued edges do not belong to the same inclusion-maximal closed subgraph of the fragment $H(F)$, we have that every $v_i^j$, $j\geq 2$, contributes at least one missing edge in $\ell$ to $\sigma$. Thus, $f$ identifies the position of every $P^i_2,\ldots,P^i_{t_i+1}$, as soon as the position of $P^i_1$ is fixed.

%It identifies the position of every $T_i\setminus P'_i$ as soon as the position of $P'_i$ is fixed. However, if $|A\cap T_i|\geq 2$, there still can be glued edges in $T_i$, and then the distances between these glued edges in $F$ are not yet reconstructed. Nevertheless, since these pairs of different glued edges do not belong to the same inclusion-maximal closed subgraph of the fragment $H(F)$, we have that every $v_i^j$, $j\geq 2$, contributes at least one missing edge in $\ell(S)$ to $\sigma(S)$. Therefore, we can reconstruct these distances, according to $\tau_3$.

%$\frac{d!}{(d-z)!}(n-x+c-z)!\leq \frac{d^z}{(n-x+c-z)^z}(n-x+c)!$

\item Since there is no component in $S$ that contains deleted pieces and glued pairs of edges simultaneously, we have that $\mathcal{I}\cap\mathcal{I}'=\varnothing.$ Therefore, it remains to order all the remaining $n-x+c-a$ connected components of $S$ (including the remaining singletons) according to the permutation $\pi$ and complete the reconstruction of $F$.

\end{enumerate}

Thus, $|\mathcal{S}(S)|$ is at most the number of all possible tuples $\mathbf{x}$ which, in turn, does not exceed
\begin{align*}
 \max_{a\leq\min\{c,\chi\},\,a+a'\leq c+\sigma} & 4^{\ell}  2^{2x}\left(\frac{n^{2/3} \ln n }{ 2w}\right)^a\frac{\chi!}{(\chi-a)!}{\ell_{\chi}\choose a'}(n-x+c-a)! \\
 &\leq 64^{\ell}(n-x+c)!\max_{a\leq c, \, a+a'\leq c+\sigma}\left(\frac{\chi n^{2/3} \ln n/(2w)}{n-x+c-a}\right)^{a}{\ell_{\chi}\choose a'}.
\end{align*} 
Due to~\eqref{eq:d-def},
$$
 \left(\frac{\chi n^{2/3} \ln n/(2w)}{n-x+c-a}\right)^{a}\leq 
 \left(\frac{\chi n^{2/3} \ln n /(2w)}{n-x}\right)^{a}\leq \left(\frac{n^{7/6}}{2(n-2\ell_0-4\chi)}\right)^{a}\leq n^{a/6}.
$$
For a fixed $a+a'=:a''$, the function $n^{a/6}{\ell_{\chi}\choose a''-a}$ is convex in $a$, implying that
$$
 \max_{a\leq c} n^{a/6}{\ell_{\chi}\choose a''-a}=\max\left\{n^{c/6}{\ell_{\chi}\choose a''-c},{\ell_{\chi}\choose a''}\right\}.
$$
Therefore, 
\begin{align*} 
|\mathcal{S}(S)|\leq 64^{\ell}(n-x+c)!\max\left\{n^{c/6}{\ell_{\chi}\choose \min\{\sigma,\lfloor\ell_{\chi}/2\rfloor\}},{\ell_{\chi}\choose \min\{c+\sigma,\lfloor\ell_{\chi}/2\rfloor\}}\right\}.
\end{align*}
%Due to~\eqref{eq:d-def},
%\begin{align*}
% |\mathcal{S}(S)|&\leq 64^{\ell}(n-x+c)!\max\left\{\left(\frac{n^{7/6}}{2(n-2\ell_0-4\chi)}\right)^{c},{\ell_{\chi}\choose \min\{c+\sigma,\lfloor\ell_{\chi}/2\rfloor\}}\right\}\\
% &\leq 64^{\ell}(n-x+c)!\max\left\{n^{c/6},{\ell_{\chi}\choose \min\{c+\sigma,\lfloor\ell_{\chi}/2\rfloor\}}\right\}.
%\end{align*}
The proof of~\eqref{eq:S_bound_large_fragments} is completed.

It remains to prove that with sufficiently large probability the factor ${\ell_{\chi}\choose a'}$ can be replaced with $\frac{(2\ell_{\chi})^{a'}}{\min\{(\ln n)^{a'},n^{1/4}\}}$. Indeed, if this is the case, then, since, for a fixed $a+a'=:a''$, the function $n^{a/6}\frac{\ell_{\chi}^{a''-a}}{\min\{(\ln n)^{a''-a},n^{1/4}\}}$ decreases in $a$, we get
$$
 \max_{a\leq c} n^{a/6}\frac{(2\ell_{\chi})^{a''-a}}{\min\{(\ln n)^{a''-a},n^{1/4}\}}=\frac{(2\ell_{\chi})^{a''}}{\min\{(\ln n)^{a''},n^{1/4}\}}\leq
 \frac{(2\ell_{\chi})^{c+\sigma}}{\min\{(\ln n)^{c+\sigma},n^{1/4}\}},
$$
as required in~\eqref{eq:S_bound_small_fragments}.

%Clearly, for $a\leq c$
%$$
%\left(\frac{\chi n^{2/3} \ln n/(2w)}{n-x+c-a}\right)^{a}<n^{c/6}<\frac{\ell_{\chi}^{c+\sigma}}{(\ln n)^{c+\sigma}+n^{1/4}}.
%$$
%Also, if $c+\sigma>\lfloor\ell_{\chi}/2\rfloor$, then 
%$$
% {\ell_{\chi}\choose \min\{c+\sigma,\lfloor\ell_{\chi}/2\rfloor\}}={\ell_{\chi}\choose \ell_{\chi}/2}<2^{\ell_{\chi}}<
% \frac{\ell_{\chi}^{\ell_{\chi}/2}}{(\ln n)^{\ell_{\chi}}+n^{1/4}}< \frac{\ell_{\chi}^{c+\sigma}}{(\ln n)^{c+\sigma}+n^{1/4}},
%$$
%since $c+\sigma<\ell_{\chi}$.
%  Therefore, it remains to prove that with sufficiently large probability the factor ${\ell_{\chi}\choose \min\{c+\sigma,\lfloor\ell_{\chi}/2\rfloor\}}$ in~\eqref{eq:S_bound_large_fragments} can be replaced with $\frac{\ell_{\chi}^{c+\sigma}}{\min\{(\ln n)^{c+\sigma},n^{1/4}\}}$.
  This refinement makes a significant difference when $c+\sigma=o(\log n/\log\log n)$. Let us recall that the factor ${\ell_{\chi}\choose c+\sigma}$ appears due to the choice of the partition $f$ in step (5) of the reconstruction procedure. If we skip this particular step, then, after implementing all the other steps, we get that $F$ is reconstructed, up to the choice of the tuple $(q_1,\ldots,q_{a'})$ of lengths (i.e., the number of vertices) of closed subgraphs that were cut from $H(F)$. For each $q_i$, we first decide, whether $q_i\leq\sqrt{n}/\ln n$. There are at most $2^{c(S)+\sigma(S)}$ ways to make this choice. For each $q_i$ that should not exceed $\sqrt{n}/\ln n$, we simply choose its value in at most $\sqrt{n}/\ln n<\ell_{\chi}/\ln n$ ways. Without loss of generality, let $q_1,\ldots,q_{a''}$ be the unspecified lengths exceeding the threshold $\sqrt{n}/\ln n$. Each evaluation of these lengths identifies an $F\in\mathcal{F}_n$. We only count evaluations such that $F\in\Sigma_W$. In particular, it means that $(F,W)$ is separating (i.e. $\mathcal{Q}_{F,W}$ has property P3), and thus, for any $F\in\Sigma_W$, the probability that $H(F)$ contains a specific subset of inclusion-maximal closed subgraphs --- defined by the specified $q_1,\ldots,q_{a''}$ --- is at most $n^{-1/3}$. The random variable $\eta$ that counts the number of evaluations of $q_1,\ldots,q_{a''}$ such that $H(F=F(q_1,\ldots,q_{a''}))$ contains the respective subset of inclusion-maximal closed subgraphs is stochastically dominated by $\mathrm{Bin}\left(\ell_{\chi}^{a''},n^{-1/3}\right)$, since $\mathcal{Q}$ is a product measure. Since $\ell_{\chi}^{a''}>\sqrt{n}$, we get
$$
\mathbb{P}\left(\eta>\frac{\ell_{\chi}^{a''}}{n^{1/4}}\right)=e^{-\Omega(n^{1/6})}.
$$
Recalling that $\eta$ is exactly the number of possible choices of $(q_1,\ldots,q_{a''})$, we get the desired bound and complete the proof.
\end{proof}

\subsection{Day 1: getting fragments of size at most $n^{0.3}$}
\label{sc:day_1}

We are now ready for the second fragmentation round. %From the previous round, we have a sufficiently large multiset of smoothed fragments. 
  Let us recall that we sampled $\mathbf{F}'\sim\mathcal{Q}_{F,W}$ and got a fragment $H(F,W)=F\cap\mathbf{F}'$ for every pair $(F,W)$.  Then we defined a smoothed version $H'(F,W)$ of $H(F,W)$ for every $W$ and every $F\in\Sigma_W$. Due to Claim~\ref{cl:pre-images-size}, \eqref{eq:S_bound_large_fragments} holds for all $S$ and~\eqref{eq:S_bound_small_fragments} holds for all $S$ with $\ell(S)\leq n^{0.1}$, $\mathcal{Q}$-whp. The latter inequality is due to the union bound over all such $S$ and the probability bound from~\eqref{eq:S_bound_small_fragments}. Therefore, we can evaluate all $\mathbf{F}'=F'$ and treat them as deterministic graphs such that 
  \begin{itemize}
  \item $|\Sigma_{\mathbf{W}}|=(1-o(1))n!$ whp;
  \item \eqref{eq:S_bound_large_fragments} holds for all $S$;
  \item  \eqref{eq:S_bound_small_fragments} holds for all $S$ with $\ell(S)\leq n^{0.1}$ deterministically.  
  \end{itemize}
Let us consider the (random) multiset $\mathcal{H}^{(1)}$ of all $H'(F)$ over $F\in\Sigma_{\mathbf{W}}$. Then 
$$
|\mathcal{H}^{(1)}|=|\Sigma_{\mathbf{W}}|=(1-o(1))n!
$$ 
whp. Clearly, we may assume that every member of this family has size exactly $\ell_0+5\chi$.

Next, we expose $\mathbf{W}^{(1)}\sim G(n,m')$ with $m'=\left\lfloor\varepsilon{n\choose 2}/\sqrt{n}\right\rfloor$. For $H\in\mathcal{H}^{(1)}$ and an $m'$-set $W\subset{[n]\choose 2}$, let $\mathcal{M}^*(H,W)$ be the set of all $H'\in\mathcal{H}^{(1)}$ such that $H'\subset H\cup W$. For $\ell\in\mathbb{N}$, let $\mathcal{M}^*_{\ell}(H,W)$ be the set of all $H'\in\mathcal{M}^*(F,W)$ such that $|H'\cap H|>\ell$. Let $\delta(n)=o(1)$ be a slowly decreasing function. Let us call $(H,W)$ {\it $\ell$-bad}, if $|\mathcal{M}^*_\ell(H,W)|>\delta_n|\mathcal{M}^*(H,W)|$. Let $\mathbf{H}$ be a uniformly random element of $\mathcal{H}^{(1)}$ sampled independently of $\mathbf{W}^{(1)}$. We shall prove that $(\mathbf{H},\mathbf{W}^{(1)})$ is not $\lceil n^{0.3}\rceil$-bad whp. That would mean that whp we may replace most of $H\in\mathcal{H}^{(1)}$ with fragments of sizes at most $\lceil n^{0.3}\rceil$.

\begin{claim}
\label{cl:not_bad_2}
Whp $(\mathbf{H},\mathbf{W}^{(1)})$ is not $\lceil n^{0.3}\rceil$-bad.
\end{claim}

\begin{proof}
For $\ell\in\mathbb{N}$ and $H\in\mathcal{H}^{(1)}$, let  $\Pi^H_{\ell}:=\Prob(H\cap \mathbf{H}\text { has }\ell\text{ edges})$. We let $f=\ell_0+5\chi$ and $\mathcal{B}$ be the set of all graphs on $n$ vertices with more than $\lceil n^{0.3}\rceil$ edges. We then apply Lemma~\ref{lm:not_bad} with these $f$ and $\mathcal{B}$. We get that, in order to prove Claim~\ref{cl:not_bad_2}, it is sufficient to show that, uniformly in $H\in\mathcal{H}^{(1)}$,
\begin{equation}
 \sum_{\ell\geq \lceil n^{0.3}\rceil+1}\Pi^H_{\ell}\left((1+o(1))\frac{N}{m'}\right)^{\ell}=o\left(\frac{1}{\ln n}\right).
\label{eq:second_power_fragmentation_step}
\end{equation}
Fix non-negative integers $c$ and $x$. Let us denote by $\mathcal{J}^H_{\ell,x,c}$ the set of all subgraphs $J\subset H$ with $\ell$ edges, $x$ non-isolated vertices, and $c$ connected components (excluding isolated vertices). Denote by $p^H(\ell,x,c)$ the probability that $\mathbf{H}\cap H\in\mathcal{J}^H_{\ell,x,c}$. Due to Claim~\ref{cl:pre-images-size}, letting $\ell_{\chi}:=\ell_0+5\chi$ and $\sigma:=2x-3c-\ell$, we get
$$
p^H(\ell,x,c)\leq\frac{|\mathcal{J}^H_{\ell,x,c}|\cdot 64^{\ell}(n-x+c)!\max\left\{n^{c/6}{\ell_{\chi}\choose \min\{\lfloor\ell_{\chi}/2\rfloor,\sigma\}},{\ell_{\chi}\choose \min\{\lfloor\ell_{\chi}/2\rfloor,c+\sigma\}}\right\}}{|\mathcal{H}^{(1)}|}.
$$
Moreover, due to Claim~\ref{cl:J_upper_bound}, $|\mathcal{J}^H_{\ell,x,c}|\leq {\ell_{\chi}\choose c} 2^{6\ell}.$ Therefore,
\begin{align*}
 \Pi^H_{\ell}&=\sum_{c=1}^{\ell}\sum_{\ell/2+3c/2\leq x\leq\ell+c}p_{\ell,x,c}\\
 &\leq \sum_{c,x}\frac{{\ell_{\chi}\choose c}2^{12\ell}\left(n-x+c\right)!\max\left\{n^{c/6}{\ell_{\chi}\choose \min\{\lfloor\ell_{\chi}/2\rfloor,\sigma\}},{\ell_{\chi}\choose \min\{\lfloor\ell_{\chi}/2\rfloor,c+\sigma\}}\right\}}{(1-o(1))n!}.
\end{align*}
In the usual way, we get
\begin{align}
\frac{(n-x+c)!}{n!}&\leq\frac{1}{(n-x+c)^{x-c}}=\frac{1}{n^{x-c}(1-\frac{x-c}{n})^{x-c}}
&=O\left(n^{-x+c}\right)
=O\left(n^{-\ell/2-c/2-\sigma/2}\right).
\label{eq:n-x+c}
\end{align}
Therefore,
\begin{align}
\label{eq:day1_pi^H_ell}
 \Pi^H_{\ell}=O\left(\sum_{c,x}\frac{{\ell_{\chi}\choose c}2^{12\ell}\max\left\{n^{c/6}{\ell_{\chi}\choose \min\{\lfloor\ell_{\chi}/2\rfloor,\sigma\}},{\ell_{\chi}\choose \min\{\lfloor\ell_{\chi}/2\rfloor,c+\sigma\}}\right\}}{n^{\ell/2+c/2+\sigma/2}}\right).
\end{align}

Now, notice that, for a graph $J\in\mathcal{J}^H_{\ell,x,c}$ consisting of $c$ components that have sizes $\ell_1,\ldots,\ell_c$, we have $\ell_i=2x_i-3-\sigma_i$, where 
$$
\sigma_i\geq \frac{4C\ell_i}{\ln n}\cdot I\left(\ell_i\geq\frac{\ln n}{100C}\right)
$$
and $C$ is a large positive constant. Indeed, since every closed subgraph has size at most $4\mu$, we get that every set of $x'\geq \frac{\ln n}{200C}$ vertices contributes at most $2x'-4$ edges. Therefore, for any admissible triple $(\ell,x,c)$, we get
\begin{equation}
\sigma=\min_{J\in\mathcal{J}^H_{\ell,x,c}}\sum_{i=1}^c\sigma_i\geq\min_{J\in\mathcal{J}^H_{\ell,x,c}}\sum_{i:\,\ell_i\geq\frac{\ln n}{100C}}\frac{4C\ell_i}{\ln n}
>\left(\ell-c\cdot\frac{\ln n}{100C}\right)\frac{4C}{\ln n}=\frac{4C}{\ln n}\cdot\ell-\frac{c}{25}.
 \label{eq:sigma_lower_2d-power}
\end{equation}

Assume first 
$$
n^{c/6}{\ell_{\chi}\choose \min\{\lfloor\ell_{\chi}/2\rfloor,\sigma\}}\geq{\ell_{\chi}\choose \min\{\lfloor\ell_{\chi}/2\rfloor,c+\sigma\}}.
$$
If, also, $\lfloor\ell_{\chi}/2\rfloor\leq\sigma$, then, due to~\eqref{eq:sigma_lower_2d-power} and \eqref{eq:day1_pi^H_ell},
\begin{align}
\Pi^H_{\ell}\left((1+o(1))\frac{N}{m'}\right)^{\ell} &\leq
\frac{{\ell_{\chi}\choose c}2^{12\ell}n^{c/6}2^{2\sigma}}{n^{\ell/2+c/2+\sigma/2}}\left((1+o(1))\frac{N}{m'}\right)^{\ell}\\
&\leq\frac{\left(\frac{e\ell_{\chi}}{c}\right)^c (2^{12}/\varepsilon+o(1))^{\ell} 2^{2\sigma} }{n^{c/3}n^{\sigma/2}}\notag\\
&\leq\left(\frac{3Cn^{1/6+1/50}}{c}\right)^c \left(\frac{2^{24}}{\varepsilon}+o(1)\right)^{\ell}n^{-2C\ell/\ln n}<
\left(\frac{n^{0.2}}{c}\right)^c e^{-C\ell}
\label{eq:day1_case_large_c-sigma}
\end{align}
if $C=C(\varepsilon)$ is large enough. If $\lfloor\ell_{\chi}/2\rfloor>\sigma$, then consider separately two cases: $c<50C\ell/\ln n$ and $c\geq 50C\ell/\ln n$. In the first case, we have that $\sigma\geq 2C\ell/\ln n$ due to~\eqref{eq:sigma_lower_2d-power}. Therefore, for large enough $C$, due to~\eqref{eq:day1_pi^H_ell},
\begin{align*}
\Pi^H_{\ell}\left((1+o(1))\frac{N}{m'}\right)^{\ell} &\leq
\frac{{\ell_{\chi}\choose c}2^{12\ell}n^{c/6}(e\ell_{\chi}/\sigma)^{\sigma}}{n^{\ell/2+c/2+\sigma/2}}\left((1+o(1))\frac{N}{m'}\right)^{\ell}\\
&\leq\left(\frac{3Cn^{1/6}}{c}\right)^c\left(\frac{3C}{\sigma}\right)^{\sigma} \left(\frac{2^{12}}{\varepsilon}+o(1)\right)^{\ell}
<\left(\frac{n^{0.2}}{c}\right)^c e^{-C\ell/2}.
\end{align*}
In the latter case, since $c=\Omega(\ell/\ln n)$ is large, we get the bound
\begin{align*}
\frac{{\ell_{\chi}\choose c}2^{12\ell}n^{c/6}(e\ell_{\chi}/\sigma)^{\sigma}}{n^{\ell/2+c/2+\sigma/2}}\left((1+o(1))\frac{N}{m'}\right)^{\ell}
&\leq\left(\frac{3Cn^{1/6}}{c}\right)^c\left(\frac{3C}{\sigma}\right)^{\sigma} \left(\frac{2^{12}}{\varepsilon}+o(1)\right)^{\ell}\\
&<n^{-0.1c} \left(\frac{2^{12}}{\varepsilon}+o(1)\right)^{\ell}<e^{-C\ell}.
\end{align*}

Finally, let $n^{c/6}{\ell_{\chi}\choose \min\{\lfloor\ell_{\chi}/2\rfloor,\sigma\}}<{\ell_{\chi}\choose \min\{\lfloor\ell_{\chi}/2\rfloor,c+\sigma\}}$. If $\lfloor\ell_{\chi}/2\rfloor\leq c+\sigma$, we get exactly the same bound as in~\eqref{eq:day1_case_large_c-sigma}. We then assume that $\lfloor\ell_{\chi}/2\rfloor> c+\sigma$.
%In the case when maximum is achieved at ${\ell_{\chi}\choose c+\sigma}$, we can clearly assume that $c+\sigma<\ell_{\chi}/2$ since otherwise 
%$$
%\frac{{\ell_{\chi}\choose c}2^{12\ell}{\ell_{\chi}\choose c+\sigma}}{n^{\ell/2+c/2+\sigma/2}}\left((1+o(1))\frac{N}{m'}\right)^{\ell}=n^{-\Omega(\sqrt{n})}.
%$$
 With this assumption, we get that ${\ell_{\chi}\choose c}\leq{\ell_{\chi}\choose c+\sigma}$. Therefore, %the expression on the right-hand side in~\eqref{eq:day1_pi^H_ell} is at most
\begin{align*}
\Pi^H_{\ell}\left((1+o(1))\frac{N}{m'}\right)^{\ell} &\leq
\frac{{\ell_{\chi}\choose c}2^{12\ell}{\ell_{\chi}\choose c+\sigma}}{n^{\ell/2+c/2+\sigma/2}}\left((1+o(1))\frac{N}{m'}\right)^{\ell}
\leq\frac{\left(\frac{e\ell_{\chi}}{c+\sigma}\right)^{2(c+\sigma)} (2^{12}/\varepsilon+o(1))^{\ell} }{n^{c/2+\sigma/2}}\\
&\stackrel{\eqref{eq:sigma_lower_2d-power}}\leq\left(\frac{9C^2n^{1/2}}{(c+\sigma)^2}\right)^{c+\sigma} \left(\frac{2^{12}}{\varepsilon}+o(1)\right)^{\frac{\sigma\ln n}{4C}+\frac{c\ln n}{100C}}<
\left(\frac{n^{0.55}}{(c+\sigma)^2}\right)^{c+\sigma}.
\end{align*}
%\begin{align*}
%\sum_{\ell\geq \lceil n^{1/5}\rceil+1}\Pi^H_{\ell}\left((1+o(1))\frac{N}{m'}\right)^{\ell} &\leq \sum_{c,\ell,x}\frac{{2\ell_0+5\chi\choose c}2^{10\ell}n^{c/6}}{n^{\ell/2+c/2+\sigma/2}}\left((1+o(1))\frac{N}{m'}\right)^{\ell}\\
%&\leq\sum_{c,\ell,x}\frac{(e(2\ell_0+5\chi)/c)^c \left((1/\varepsilon+o(1))2^{10}\right)^{\ell} }{n^{c/3+\sigma/2}}\\
%&\leq\sum_{c,\ell,x}(8Cn^{1/2-1/3+1/50}/c)^c \left((1/\varepsilon+o(1))2^{10}\right)^{\ell}n^{-2C\ell/\ln n}.
%\end{align*}
%We can choose $C=C(\varepsilon)$ so large that 
%$$
%\left((1/\varepsilon+o(1))2^{10}\right)^{\ell}n^{-2C\ell/\ln n}\leq e^{-C\ell}.
%$$
Due to~\eqref{eq:sigma_lower_2d-power}, $c+\sigma\geq \frac{4C}{\ln n}\cdot\ell$. We finally get
\begin{align*}
\sum_{\ell\geq \lceil n^{0.3}\rceil+1}\Pi^H_{\ell}\left((1+o(1))\frac{N}{m'}\right)^{\ell}
&\leq\sum_{c,\ell,x}\left[\left(\frac{n^{0.2}}{c}\right)^c e^{-C\ell/2}+\left(\frac{n^{0.55}}{(c+\sigma)^2}\right)^{c+\sigma}\right]\\
&\leq\sum_{\ell\geq n^{0.3}}\left(e^{-\ell}+\sum_x\left(\sum_{c\leq n^{0.25}}n^{0.2c} e^{-C\ell/2}+\sum_{c>n^{0.25}}n^{-0.05c}\right)\right)\\
&\leq n^3 e^{-\Omega(n^{0.3})}+n^3 e^{-\Omega(n^{0.25})}=e^{-\Omega(n^{0.25})}.
\end{align*}
This completes the proof. % due to~\eqref{eq:bad_to_X_day1} and \eqref{eq:X_to_X'_day1}.
\end{proof}

Let $Y^{(1)}$ be the number of $H\in\mathcal{H}^{(1)}$ (with multiplicities) such that $(H,\mathbf{W}^{(1)})$ is not $\lceil n^{0.3}\rceil$-bad. Due to Markov's inequality, $Y^{(1)}=(1-o(1))n!$ whp. % \MZ{Be careful with this division by $2n$ --- make consistent everywhere.}
  For every such $H$ we choose $H'\subset H\cup\mathbf{W}^{(1)}$ such that $H'\in\mathcal{H}^{(1)}$ and the fragment $H\cap H'$ has size at most $\lceil n^{0.3}\rceil$ and put this fragment into the new multiset $\mathcal{H}^{(2)}$. Without loss of generality, we may assume that each graph in this multiset has exactly $\lceil n^{0.3}\rceil$ edges.
  
\subsection{Day 2: getting fragments of size $o(\log n)$}
\label{sc:day_2}

Here, we consider sufficiently sparser random graphs. Let 
$$
 m'=\left\lceil\frac{\varepsilon}{(\ln\ln n)\sqrt{n}}\cdot N\right\rceil,\quad R:=\left\lceil\log_{3/2}(\ln n/\ln\ln n)\right\rceil+1.
$$
Expose $R-1$ independent copies $\mathbf{W}^{(2)},\ldots,\mathbf{W}^{(R)}\sim G(n,m')$. We make $R-1$ fragmentation steps and argue by induction. Assume that, for $i\in\{2,\ldots,R\}$, we have a multiset $\mathcal{H}^{(i)}$ of $(1-o(1))n!$ fragments of size 
$$
f_i:=\left\lceil n^{0.3\cdot (2/3)^{i-2}}\right\rceil,
$$
each is a subgraph of a smoothed fragment of some $F$ --- one fragment per $F$. We also assume that the union of each fragment with $\mathbf{W}\cup\mathbf{W}^{(1)}\cup\ldots\cup\mathbf{W}^{(i-1)}$ contains some $F\in\mathcal{F}_n$. Let $\mathbf{H}$ be a uniformly random element of $\mathcal{H}^{(i)}$. Observe that $(\ln n)^{0.2}<f_{R+1}\leq (\ln n)^{0.3}$ and $f_i>f_{R+1}$, for every $i\in\{2,\ldots,R\}$.
\begin{claim}
\label{cl:day_2}
With probability $1-o((\ln n)^{-1/3})$, there exists $H'\subset \mathbf{H}\cup\mathbf{W}^{(i)}$, where $H'\in\mathcal{H}^{(i)}$, such that $|\mathbf{H}\cap H'|\leq f_{i+1}$.
\end{claim}
\begin{proof}
For $\ell\in\mathbb{N}$ and $H\in\mathcal{H}^{(i)}$, let  $\Pi^H_{\ell}:=\Prob(H\cap \mathbf{H}\text { has }\ell\text{ edges})$. We let $f=f_i$ and $\mathcal{B}$ be the set of all graphs on $n$ vertices with more than $f_{i+1}$ edges. By Lemma~\ref{lm:not_bad}, it is sufficient to show~\eqref{eq:second_power_fragmentation_step} with $\lceil n^{0.3}\rceil$ replaced by $f_{i+1}$. For non-negative integers $c$ and $x$, we denote, as above, the probability that $\mathbf{H}\cap H\text{ has }\ell\text{ edges, }x\text{ non-isolated vertices, and }c\text{ components}$ by $p^H(\ell,x,c)$. Here,  letting $\sigma:=2x-3c-\ell$ and $\ell_{\chi}:=\ell_0+5\chi$, we have that $c+\sigma=2(x-c)\ell\leq\ell<\ell_{\chi}/2$ and also 
\begin{align*}
n^{c/6}{\ell_{\chi}\choose\sigma}
=n^{c/6}\frac{(c+\sigma)!(\ell_{\chi}-c-\sigma)!}{\sigma!(\ell_{\chi}-\sigma)!}{\ell_{\chi}\choose c+\sigma}
&<\left(n^{1/6}\cdot\frac{c+\sigma}{\ell_{\chi}-c-\sigma}\right)^c{\ell_{\chi}\choose c+\sigma}\\
&<\left(\frac{n^{1/6}\ell}{\ell_{\chi}-\ell}\right)^c{\ell_{\chi}\choose c+\sigma}
<{\ell_{\chi}\choose c+\sigma}.
\end{align*}
Therefore, due to Claims~\ref{cl:J_upper_bound},~\ref{cl:pre-images-size}, we get
\begin{equation}
p^H(\ell,x,c)\leq\frac{{f_i\choose c} 2^{12\ell} (n-x+c)!{\ell_{\chi}\choose c+\sigma}}{(1-o(1))n!}.
\label{eq:p^H_second_power_fragmentation_step}
\end{equation}
Note that we have not yet specified the value of $w$. The only restriction that we currently have is $1\ll w\ll \log n$. We, therefore, may further assume that $\frac{\log n}{\log\log\log n}\ll w\ll\log n$. In particular, it implies the following refinement of the inequality~\eqref{eq:sigma_lower_2d-power}:
\begin{equation}
 \ell\leq (c+\sigma)\ln\ln\ln n.
\label{eq:sigma_lower_2d-power_refinement_lnlnln}
\end{equation}
Thus, applying~\eqref{eq:n-x+c},~\eqref{eq:p^H_second_power_fragmentation_step}, and the inequality $c+\sigma=2(x-c)-\ell\leq\ell$, we get that, for sufficiently large $C=C(\varepsilon)$,
\begin{align*}
 \Pi^H_{\ell}\left((1+o(1))\frac{N}{m'}\right)^{\ell}&
 \leq \sum_{c,x}\frac{{f_i\choose c} 2^{12\ell} {\ell_{\chi}\choose c+\sigma}((1/\varepsilon+o(1))\ln\ln n)^{\ell}}{n^{c/2+\sigma/2}}\\
 &\leq \sum_{c,x}\frac{{f_i\choose c}C^{c+\sigma}\left(\frac{C\ln\ln n}{\varepsilon}\right)^{\ell}}{(c+\sigma)^{c+\sigma}}\\
 &\stackrel{(*)}\leq \sum_{c,x}\frac{\left(\frac{e f_i}{c+\sigma}\right)^{c+\sigma}\left(C^3\ln\ln n\right)^{\ell}}{(c+\sigma)^{c+\sigma}}
 \leq
 \sum_{c,x}\left(\frac{f_i\cdot e^{(\ln\ln\ln n)^3}}{(c+\sigma)^2}\right)^{c+\sigma},
\end{align*} 
where the inequality (*) holds within the assumption that $c+\sigma\leq f_i/2$, and then ${f_i\choose c}\leq{f_i\choose c+\sigma}$. We can indeed make such an assumption since otherwise we get the bound $f_i^{-\Omega(f_i)}$ immediately.
Since
$$
e^{(\ln\ln\ln n)^3}<(\ln n)^{0.1}<\frac{f_i^{1/3}}{2(\ln\ln\ln n)^2}\quad\text{ and }\quad f_{i+1}+1\geq (f_i-1)^{2/3}+1\geq f_i^{2/3},
$$ 
and due to~\eqref{eq:sigma_lower_2d-power_refinement_lnlnln}, we finally get
\begin{align*}
\sum_{\ell\geq f_{i+1}+1}\Pi^H_{\ell}\left((1+o(1))\frac{N}{m'}\right)^{\ell}
&\leq\sum_{\ell\geq f_{i+1}+1}\sum_{c,x}\left(\frac{(f_i)^{4/3}/(\ln\ln\ln n)^2}{2(\ell/\ln\ln\ln n)^2}\right)^{\ell/\ln\ln\ln n}\\
&\leq\sum_{\ell\geq f_{i+1}}\sum_{c,x} 2^{-\ell/\ln\ln\ln n}\leq \sum_{\ell\geq f_{i+1}}\ell^2 2^{-\ell/\ln\ln\ln n}\\
&=2^{-\Omega((\ln n)^{0.2}/\ln\ln\ln n)},
\end{align*}
completing the proof.
\end{proof}

Due to Markov's inequality and the union bound, after the last step, we get a multiset $\mathcal{H}^{(R+1)}$ of $(1-o(1))n!$ fragments of size at most $\left\lceil(\ln n)^{0.3}\right\rceil$. Without loss of generality, we may assume that each graph in this multiset has exactly $\left\lceil(\ln n)^{0.3}\right\rceil$ edges.

\subsection{Day 3: covering a fragment}
\label{sc:day_3}

%It is well known that increasing properties that hold whp in the uniform model hold whp in the respective binomial model as well (see, e.g.,~\cite[Corollary 1.16]{Janson}). In particular, letting $p=(1+10\varepsilon)\sqrt{e/n}$, we get that, for $\mathbf{G}\sim G(n,p)$ and for every $H\in\mathcal{H}^{(R+1)}=\mathcal{H}^{(R+1)}(\mathbf{G})$, whp $H\cup\mathbf{G}$ contains some $F\in\mathcal{F}_n$. 

Next, we expose 
$$
\mathbf{G}'\sim G(n,p') ,\quad \text{ where}\quad p'=\varepsilon \cdot n^{-1/2}.
$$
Actually it remains to prove that whp there exists $H\in\mathcal{H}^{(R+1)}$ such that $H\subset\mathbf{G}'$. Let $Y^{(R+1)}$ be the number of such $H$. We get
$$
 \mathbb{E}Y^{(R+1)}=(p')^{\lceil (\ln n)^{0.3}\rceil}\cdot|\mathcal{H}^{(R+1)}|=(1-o(1))(p')^{\lceil (\ln n)^{0.3}\rceil}\cdot n!=\omega(1).
$$
On the other hand, recalling that $\mathcal{J}^{H}_{\ell,x,c}$ is defined in Section~\ref{sc:day_1}, due to the second part of Claim~\ref{cl:pre-images-size} (recall that in the beginning of Section~\ref{sc:day_1} we chose the fragments satisfying the bound~\eqref{eq:S_bound_small_fragments} deterministically), letting $\ell_{\chi}:=\ell_0+5\chi$ and $\sigma:=2x-3c-\ell$, we get
$$
 \mathrm{Var}Y^{(R+1)}\leq|\mathcal{H}^{(R+1)}|\sum_{\ell,x,c}\max_{H\in\mathcal{H}^{(R+1)}}
 |\mathcal{J}^{H}_{\ell,x,c}| 64^{\ell}(n-x+c)!\frac{(2\ell_\chi)^{c+\sigma}}{(\ln n)^{c+\sigma}}(p')^{2\lceil (\ln n)^{0.3}\rceil-\ell}-(\mathbb{E}Y^{(R+1)})^2.
$$
This implies
$$
 \frac{\mathrm{Var}Y^{(R+1)}}{(\mathbb{E}Y^{(R+1)})^2}\leq\sum_{\ell,x,c}\max_{H\in\mathcal{H}^{(R+1)}}
 \frac{|\mathcal{J}^{H}_{\ell,x,c}|\cdot 64^{\ell}(n-x+c)! (2\ell_\chi)^{c+\sigma}}{(1-o(1))n! p'^{\ell}  (\ln n)^{c+\sigma}} -1.
$$
%\MZ{Take care of the division by $n$ in the denominator}
 Due to Claim~\ref{cl:J_upper_bound}, estimates~\eqref{eq:n-x+c},~\eqref{eq:sigma_lower_2d-power_refinement_lnlnln}, and recalling that $c+\sigma\leq\ell$, we get
\begin{align*}
 \frac{\mathrm{Var}Y^{(R+1)}}{(\mathbb{E}Y^{(R+1)})^2} &\leq \sum_{\ell\geq 0,c,x}\frac{{\lceil (\ln n)^{0.3}\rceil\choose c}2^{12\ell}(4C\sqrt{n})^{c+\sigma}}{n^{\ell/2+c/2+\sigma/2}(\varepsilon\cdot n^{-1/2})^{\ell}(\ln n)^{c+\sigma}}-1\\
&\leq\sum_{\ell\geq 1,c,x}\frac{(e\lceil (\ln n)^{0.3}\rceil)^{c} \left(2^{14}C/\varepsilon\right)^{\ell} }{(\ln n)^{c+\sigma}}
\leq\sum_{\ell\geq 1,c,x}(\ln n)^{-(c+\sigma)/2}\\
&\leq\sum_{\ell\geq 1,c,x}(\ln n)^{-\ell/(2\ln\ln\ln n)}\leq\sum_{\ell\geq 1}\ell^2(\ln n)^{-\ell/(2\ln\ln\ln n)}=e^{-\Omega(\ln\ln n/\ln\ln\ln n)},
\end{align*}
completing the proof that an $H\subset\mathbf{G}'$ whp exists, due to Chebyshev's inequality.

%It is well known that increasing properties that hold whp in the uniform model hold whp in the respective binomial model as well, and vice versa (see, e.g.,~\cite[Corollary 1.16]{Janson}). 

Now, let $\mathbf{W}'\sim G(n,m')$ , where $m'=\lceil p'N\rceil$. Then whp there exists $H\in\mathcal{H}^{(R+1)}$ such that $H\subset\mathbf{W}'$. We conclude that whp $\mathbf{W}\cup\mathbf{W}^{(1)}\cup\ldots\cup\mathbf{W}^{(R)}\cup\mathbf{W}'$ contains the square of a Hamilton cycle. Then, we apply ~\cite[Corollary 1.16]{Janson} in the opposite direction and get that whp $\mathbf{G}\sim G(n,(1+10\varepsilon)\sqrt{e/n})$ contains the square of a Hamilton cycle, as well.

\subsection{Proof of Claim~\ref{cl:lm:main}}
\label{sc:lm:main_proof}

Let $\ell\geq\ell_0+1$ be an integer. We first note that the square of a path with $\ell$ edges has $\frac{\ell+3}{2}$ vertices. Therefore, a disjoint union of $c$ squares on $\ell$ edges has $\frac{\ell+3c}{2}$ vertices. We then fix integers $c\in[\ell]$ and $x\in\left[\frac{\ell}{2}+\frac{3}{2}c,\ell+c\right]$. For a subgraph $H\subset F\in\mathcal{F}_n$ with $x$ vertices, $\ell$ edges, and $c$ components, 
\begin{equation}
 \sigma:=2x-\ell-3c
\label{eq:sigma-definition_squareH}
\end{equation} 
is the difference between the number of edges of a disjoint union of squares of a path with $x$ vertices and $c$ components and the number of edges of $H$.

Now, let $F\in\mathcal{F}^*_n(\overrightarrow{D})$ and let $\mathbf{F}$ be a uniformly random element of $\mathcal{F}^*_n(\overrightarrow{D})$. Let $p(\ell,x,c)$ be the probability that $(\mathbf{F}\cap F)\cup(D_1\cup\ldots\cup D_{\chi})$ is a graph on $x+4\chi$ vertices with $\ell+5\chi$ edges, and $c+\chi$ connected components (as usual, we think about graphs as about sets of their edges, so there are no isolated vertices in $(\mathbf{F}\cap F)\cup(D_1\cup\ldots\cup D_{\chi})$). % and without closed subgraphs of size $K$. We may assume that $\sigma\geq x/K-c$ since otherwise $p(\ell,x,c)=0$.
 The following claim immediately implies that, for sufficiently large $A_1,A_2>0$,
\begin{equation}
 p(\ell,x,c)\leq\frac{{n\choose c}{x\choose c+\chi}\left(\frac{4x+2\chi}{\chi}\right)^{\chi} e^{A_1c+A_2\sigma}\left(\max_{o\leq 8\sigma}{x\choose o}\right)^2(n-x-4\chi+c)!}{|\mathcal{F}_n(\overrightarrow{D})|}.
\label{eq:p_ell_x_c_bound}
\end{equation}

\begin{claim}
\label{cl:main_square}
There exist $A^1_{\alpha},A^2_{\alpha}>0$ such that the number of subgraphs $H\subset F$ such that $H\cup(D_1\cup\ldots\cup D_{\chi})$ is a graph without isolated vertices with $x+4\chi$ vertices, $\ell+5\chi$ edges, and $c+\chi$ components is
\begin{equation}
\alpha(\ell,x,c)\leq{n\choose c}{x\choose c+\chi}\left(\frac{4x+2\chi}{\chi}\right)^{\chi}e^{A^1_{\alpha} c+A^2_{\alpha}\sigma} \max_{o\leq 8\sigma}{x\choose o}.
\label{eq:count_subgraphs_general}
\end{equation}
There exist $A^1_{\beta},A^2_{\beta}>0$ such that, given $H\subset F$ such that $H\cup(D_1\cup\ldots\cup D_{\chi})$ has $x+4\chi$ vertices, $\ell+5\chi$ edges, and $c+\chi$ components, the number of ways to extend $H$ to a graph from $\mathcal{F}^*_n(\overrightarrow{D})$ is at most
\begin{equation}
 \beta(\ell,x,c)=(n-x-4\chi+c)!e^{A^1_{\beta}c+A^2_{\beta}\sigma}\max_{o\leq 8\sigma}{x\choose o}.
\label{eq:embed_subgraphs_general}
\end{equation}
\end{claim}

\begin{proof}
There are at most ${x\choose c+\chi}{\sigma+c+\chi\choose c+\chi}$ ways to choose positive integers $\ell_1,\ldots,\ell_{c+\chi}$ as well as $x_1,\ldots,x_{c+\chi}$  such that 
\begin{itemize}
\item $\frac{\ell_i}{2}+\frac{3}{2}\leq x_i\leq\ell_i+1$ for all $i\in[c]$ and $\frac{\ell_i}{2}\leq x_i\leq\ell_i$ for all $i\in\{c+1,\ldots,c+\chi\}$,
\item $\sum_{i=1}^{c+\chi}\ell_i=\ell$, and $\sum_{i=1}^{c+\chi} x_i=x$.
\end{itemize} 
Fix such $x_i$ and $\ell_i$, $i\in[c+\chi]$. For $i\in[c]$, set $\sigma_i=2x_i-\ell_i-3$, while, for $i\in[c+\chi]\setminus[c]$, set $\sigma_i=2x_i-\ell_i$.

Let us compute the number of ways to choose connected vertex-disjoint subgraphs $R_1,\ldots,R_{c+\chi}$ from $F$ such that, for every $i\in[c]$, $R_i$ has $\ell_i$ edges and $x_i$ vertices, and, for every $i\in[c+\chi]\setminus[c]$, $R_i\cup D_i$ is a connected graph with $\ell_i+5$ edges and $x_i+4$ vertices. We first choose closed subgraphs of $F$ that correspond to $\sigma_i=0$ one by one. The number of ways to choose the first closed subgraph is at most $2n$ (the heaviest case is $x_i=2$ that corresponds to an edge). In the same way, if some set of $\tilde n$ vertices is already included in the subgraph under construction, then the next closed component can be chosen in at most $2(n-\tilde n)$ ways. On the other hand, if the $i$-th component contains one of the diamonds (i.e. $i\geq c+1$), then the number of ways to choose it is at most $x_i+1$.  %Assuming that $\tilde c+\tilde d$ is the total number of closed components ($\tilde d$ components that contain a diamond), we get that there are at most $2^{\tilde c}\frac{n!}{(n-\tilde c)!}(x/d)^d$ choices of such components.

We now switch to not closed components. Let us call a vertex {\it free} in a subgraph of $F$, if its degree is less than 4. Note that the $i$-th component has $o_i\leq 6+2\sigma_i$ free vertices. Assume that $\tilde n$ vertices have been already included in the subgraph and we now describe the procedure of choosing the $i$-th component:
\begin{enumerate}
\item choose the number of free vertices $o_i\leq 6+2\sigma_i\leq 8\sigma_i\leq 8^{\sigma_i}$ (recall that $\sigma_i\geq 1$);
\item choose a set $\mathcal{O}\in{[x_i]\choose o_i}$ that comprises the labels of free vertices in the $i$-th component; 
\item if $i\leq c$, choose a vertex $w$ outside of the set of the remaining $n-\tilde n$ vertices and, if $i\geq c+1$, choose $w$ such that $0\leq \pi_F(u_{i-c})-\pi_F(w)\leq 2x_i$, where $u_{i-c}$ is the minimum vertex of the diamond $D_{i-c}$ that this component has to contain;  {\it activate} $w$ --- we treat this vertex as the $\pi_F$-minimum vertex in the component under construction;
\item at every step $j\geq 1$, consider the $\pi_F$-minimum vertex $v_j$ among active vertices:
\begin{itemize}
\item if $j\in\mathcal{O}$ (i.e. $v_j$ should be free), then add to the component some of the neighbours of $v_j$ (in at most $2^4$ ways), deactivate $v_j$, and {\it activate} all its chosen neighbours,
\item if $j\notin\mathcal{O}$, then add to the component all the neighbours of $v_j$, deactivate $v_j$, and {\it activate} all its neighbours.
\end{itemize}
\end{enumerate}
We get that, if $i\leq c$, then the number of ways to choose the component is at most $(n-\tilde n)8^{\sigma_i}\max\limits_{o_i\leq 8\sigma_i}{x_i\choose o_i}2^{4o_i}$, and at most $(2x_i+1)8^{\sigma_i}\max\limits_{o_i\leq 8\sigma_i}{x_i\choose o_i}2^{4o_i}$, otherwise.

Eventually we get that the number of {\it ordered} choices of components with parameters $\ell_i,x_i$, $i\in[c+\chi]$, in $F$ is at most 
\begin{align*}
 c!{n\choose c}2^{c}\left(\frac{2x+\chi}{\chi}\right)^{\chi}\prod_{i=1}^{c+\chi}8^{\sigma_i}\max_{o_i\leq 8\sigma_i}{x_i\choose o_i}2^{4o_i} \leq c!{n\choose c}\left(\frac{2x+\chi}{\chi}\right)^{\chi}2^{c+35\sigma} \max_{o\leq 8\sigma}{x\choose o}.
\end{align*} 
Note that this bound does not depend on the order of the choice of components that do not contain diamonds, thus
\begin{align*}
 \alpha(\ell,x,c) & \leq{n\choose c}{x\choose c+\chi}{\sigma+c+\chi\choose c+\chi}\left(\frac{2x+\chi}{\chi}\right)^{\chi}2^{c+35\sigma} \max_{o\leq 8\sigma}{x\choose o}\\
 &\leq{n\choose c}{x\choose c+\chi}\left(\frac{4x+2\chi}{\chi}\right)^{\chi}2^{2c+36\sigma} \max_{o\leq 8\sigma}{x\choose o}
\end{align*}
as needed.

Let us now fix $H\subset F$ such that $H\cup(D_1\cup\ldots\cup D_{\chi})$ has $x+4\chi$ vertices, $\ell+5\chi$ edges, and $c+\chi$ components. Let us bound the number of ways to extend $H$ to an $F'\in\mathcal{F}^*_n(\overrightarrow{D})$. We construct such an extension in the following way.

We forget the labels of the vertices from $V(H)\setminus V(D_1\cup\ldots\cup D_{\chi})$ and assume, without loss of generality, that the desired $F'\in\mathcal{F}^*_n(\overrightarrow{D})$ is such that $F'\cup D_1\cup\ldots\cup D_{\chi}$ is the square of the cycle $(12\ldots n)$. Let $\mathcal{Z}$ be the union of the set of the $c$ connected components of $H$ that do not contain diamonds with the set of remaining $n-x-4\chi$ isolated vertices. Let $\mathcal{Z}^*$ be the union of components of $H$ that contain some $D_i$.  Then the desired number of extensions is exactly the number of ways to embed the elements of $\mathcal{Z}\sqcup\mathcal{Z}^*$ in $F'$ disjointly, where in each component in $\mathcal{Z}^*$ only the vertices of diamonds are labelled.

Let $z_1,\ldots,z_{n-x-4\chi+c}$ be an ordering of $\mathcal{Z}$ (there are $(n-x-4\chi+c)!$ ways to order the elements of $\mathcal{Z}$). We embed sequentially each $z_i$ in $F'$ in a way such that all vertices of $z_i$ are bigger than the minimum vertices of $z_1,\ldots,z_{i-1}$.  At every step $i=1,\ldots,n-x-4\chi+c$, consider the minimum vertex $\kappa_i$ of $F'$ such that none of the embedded elements of $\mathcal{Z}$ in $F'$ contain this vertex. If $z_i$ is a single vertex, then we assign $\kappa_i$ with $z_i$ and proceed to the next step. Otherwise, we distinguish between the following cases. %We let $z_i=R_1$ without loss of generality and for the sake of simplicity of notations. 

First, we assume that $z_i$ is closed. If $|V(z_i)|=2$, then there are $2$ ways to choose this edge (two edges to smaller vertices are prohibited) and at most two ways to place it (two rotations). Thus there are 4 ways to embed $z_i$. If $|V(z_i)|=3$, then there are 6 ways to choose a (labelled) copy of $z_i$ in $F'$ with the minimum vertex $\kappa_i$. If $|V(z_i)|=4$, then the number of ways equals 4. If $|V(z_i)|>4$, then there are 2 ways to choose a copy of $z_i$ in $F'$ with the minimum vertex $\kappa_i$.

Second, let $z_i$ be not closed with $o_i$ free vertices. Choose a set $\mathcal{O}_i$ from ${[|V(z_i)|]\choose o_i}$ that comprises the labels of free vertices in the embedding of $z_i$ in $F'$ (we forget the isomorphism class of $z_i$ at this point and estimate the number of ways to expose a connected subgraph of $F'$ with $|V(z_i)|$ vertices, $o_i$ free vertices, starting from $\kappa_i$). {\it Activate} $\kappa_i$. At every step $j\geq 1$, choose the $\pi_{F'}$-minimum active vertex $v_j$ in $F'$ and
\begin{itemize}
\item if $j\in\mathcal{O}_i$, then add to the image of $z_i$ under construction some neighbours of $v_j$ (in at most $2^3$ ways), deactivate $v_j$, and {\it activate} all its added neighbours,
\item if $j\notin\mathcal{O}_i$, then add all the neighbours of $v_j$, that have not been added yet, deactivate the vertex, and {\it activate} all its added neighbours.
\end{itemize}
The image is constructed. However, we have not yet mapped the vertices of $z_i$ to the vertices of the image. Clearly, $z_i$ is union of a `path' $P$ of inclusion-maximum 2-connected graphs, joined by (usual) paths, with `fringe' trees, growing from $P$. There are at most 4 automorphisms of $P$, unless $P$ is a cycle. Since the closed 1-neighbourhood of every vertex in a square of a cycle does not have an independent set of size 3, every non-root vertex of a rooted `fringe' tree $T$ has degree at most 3 --- so it is free. In particular, $T$ has at most $2^{|V(T)|-1}$ automorphisms since it is a subtree of a perfect binary tree. %Moreover, the number of free non-root vertices in $T$ is at least $|V(T)|/2$.
 We conclude that, in total, there are at most $4\cdot 2^{o_i}=2^{o_i+2}\leq 2^{2o_i}$ automorphisms of $z_i$ and this bound holds true even when $P$ is a cycle. We get that the number of ways to construct the image of $z_i$ is at most ${|V(z_i)|\choose o_i}2^{5o_i}$.

We finally switch to embedding components from $\mathcal{Z}^*$, which is an easier task, since they are ``rooted'' in advance. Indeed, it is easy to see that, for a component of $H$ that contains some $D_i$ and has $o_i$ free vertices, there are at most $6^{o_i}$ ways to embed it in $F'$. Recalling the bound $o_i\leq 6+2\sigma_i\leq 8\sigma_i\1_{\sigma_i\geq 1}$, we conclude that there are at most 
$$
(n-x-4\chi+c)!6^{c}\times\prod_{i\in[c]:\,\sigma_i\geq 1}{x_i\choose o_i}2^{5o_i}\times\prod_{i\in[c+\chi]\setminus[c]:\,\sigma_i\geq 1}6^{o_i}\leq
(n-x-4\chi+c)!6^{c}2^{40\sigma}\max_{o\leq 8\sigma}{x\choose o}
$$
ways to expose $F'$ as needed.
\end{proof}

Let us now finish the proof of Claim~\ref{cl:lm:main}. Recalling that $|\mathcal{F}_n(\overrightarrow{D})|=(n-4\chi)!$ and that~\eqref{eq:sigma-definition_squareH} implies
$$
x-c=((x+4\chi)-(c+\chi))-3\chi=((\ell+5\chi)/2+(c+\chi)/2+\sigma/2)-3\chi=\ell/2+c/2+\sigma/2,
$$
we get
\begin{align*}
 \frac{(n-x-4\chi+c)!}{|\mathcal{F}_n(\overrightarrow{D})|}  =
 \frac{(n-\frac{\ell+c+\sigma}{2}-4\chi)!}{(n-4\chi)!}
 &\leq \sqrt{2\pi n}\cdot\frac{e^{\frac{\ell+c+\sigma}{2}}}{(n-4\chi)^{\frac{\ell+c+\sigma}{2}}}\left(\frac{n-\frac{\ell+c+\sigma}{2}-4\chi}{n-4\chi}\right)^{n-4\chi-\frac{\ell+c+\sigma}{2}}\\
 &\leq \sqrt{2\pi n}\cdot\frac{e^{(\ell+c+\sigma)^2/(4(n-4\chi))}}{(n-4\chi)^{(\ell+c+\sigma)/2}}
 =e^{o(\ell)}\cdot\frac{e^{(x-c)^2/n}}{n^{(\ell+c+\sigma)/2}}.
\end{align*}
We also notice that $\ln(x/\chi)=o(x/\chi)$ since $x/\chi=\Omega(\ln n/w)=\omega(1)$. Therefore, 
$$
 \left(\frac{4x+2\chi}{\chi}\right)^{\chi}=e^{\chi\ln \left(\frac{x}{\chi}\cdot\frac{4x+2\chi}{x}\right)}=e^{\chi\ln(x/\chi)+\chi\ln(4+o(1))}=e^{o(x)}=e^{o(\ell)}.
$$

We are now ready to derive~\eqref{eq:lm}. Let us choose $0<\varepsilon'\ll\delta$ small enough. From (\ref{eq:p_ell_x_c_bound}) and the bound ${n\choose c}\leq\left(\frac{en}{c}\right)^c$, we get % and the equality $x-c=\ell/2+c/2+\sigma/2$, we get
\begin{equation}
 p(\ell,x,c)\leq
 e^{o(\ell)}\cdot\frac{\left(\frac{e^{A_1+1}\sqrt{n}}{c}\right)^c{x\choose c+\chi} e^{A_2\sigma}{x\choose o}^2 e^{(x-c)^2/n}}{n^{\ell/2+\sigma/2}},
\label{eq:lm_proof_general_bound_2d-power}
\end{equation}
where $o=o(x)\leq 8\sigma$ is chosen in such a way that ${x\choose o}$ achieves its maximum. Since $\left(\frac{e^{A_1+1}\sqrt{n}}{c}\right)^c\leq e^{e^{A_1}\sqrt{n}}$, we get that
$$
 \left(\frac{e^{A_1+1}\sqrt{n}}{c}\right)^c \leq e^{e^{A_1}\sqrt{n}}\leq e^{(e^{A_1}/C)\ell}.
$$
Let us recall that $\ell>\ell_0$, and thus $x>\ell_0/2$. Here $\ell_0=\lfloor C\sqrt{n}\rfloor$, and $C\gg\max\{\frac{1}{\varepsilon'},e^{A_1}/\delta\}$ can be chosen arbitrarily large.

We further consider separately several different cases. 

{\bf 1.} If $\sigma>\varepsilon' x$, then 
$$
{x\choose c+\chi}{x\choose o}^2e^{A_2\sigma}e^{(x-c)^2/n}< 8^xe^{x}e^{A_2\sigma}<n^{\sigma/2}.
$$
Therefore,
\begin{equation}
 p(\ell,x,c)\leq e^{\delta\ell}/n^{\ell/2}.
\label{eq:small_p_ell_x_c_2d-power}
\end{equation}

{\bf 2.} Let $\sigma\leq\varepsilon' x$. 

{\bf 2.1.} If $c<\varepsilon' x$ and $x<\varepsilon' n$, then
$$
  {x\choose c+\chi}{x\choose o}^2e^{A_2\sigma}e^{(x-c)^2/n}<e^{(\delta/2)\ell}
$$
implying (\ref{eq:small_p_ell_x_c_2d-power}) as well. 

{\bf 2.2.} If $c<\varepsilon' x$ and $x\geq \varepsilon' n$, then 
$$
 \left(\frac{e^{A_1+1}\sqrt{n}}{c}\right)^c e^{A_2\sigma} {x\choose c+\chi}{x\choose o}^2<e^{(\delta/2)\ell}.
$$
Thus, (\ref{eq:lm_proof_general_bound_2d-power}) implies
$$
 p(\ell,x,c)\leq \frac{e^{(x-c)^2/n}}{n^{\ell/2}}e^{(\delta/2+o(1))\ell}\leq (e^{1+\delta}/n)^{\ell/2}
$$
since $\ell/2=x-3c/2-\sigma/2\geq x(1-2\varepsilon')$. 

{\bf 2.3.} Finally, let $c\geq\varepsilon' x$. Let us choose $C$ so large that $\left(\frac{e^{A_1+1}\sqrt{n}}{c}\right)^c\leq e^{-10x}$ (taking any $C>10e^{A_1+1}e^{1/\varepsilon'}/\varepsilon'$ is enough for that). Since $e^{A_2\sigma}\leq n^{\sigma/2}$, we get
$$
 p(\ell,x,c)\leq e^{-10x+o(\ell)}\frac{{x\choose c+\chi}{x\choose o}^2 e^{(x-c)^2/n}}{n^{\ell/2}}
\leq e^{-10x+x+o(\ell)}8^x n^{-\ell/2}\leq n^{-\ell/2}.
$$

Summing up,
\begin{align*}
\sum_{\ell\geq \ell_0+1}\Pi_{\ell}\left((1+o(1))\frac{N}{m}\right)^{\ell}& \leq
\sum_{\ell\geq \ell_0+1}\sum_{x,c}p(\ell,x,c)\left(\frac{1+o(1)}{1+\varepsilon}\cdot\sqrt{\frac{n}{e}}\right)^{\ell}\\
&\leq n^2\sum_{\ell\geq \ell_0+1}\left(\frac{e^{\delta/2}+o(1)}{1+\varepsilon}\right)^{\ell}=e^{-\Omega(\sqrt{n})},
\end{align*}
completing the proof of Claim~\ref{cl:lm:main}.

\section{Generalisation: proof of Theorem~\ref{th:second_power_generalisation}}
\label{sc:generalisation_square_sketch}

Let $\mathcal{F}_n$ be the family of all isomorphic copies of $F$ on $[n]$ rooted in the first $r$ vertices. Let $\lambda$ be the number of edges in $F_0:=F[[r]]$. We have that $|\mathcal{F}_n|=n!$ and each graph from $\mathcal{F}_n$ induces a linear order on $[n]$. 

We then proceed in the same way as in the proof of Theorem~\ref{th:second_power}. Here, $w=o(\ln n/r)$ and $\chi=\left\lfloor\frac{w n^{(2d-\Delta)/d}}{\ln n}\right\rfloor$. We replace diamonds with $(2r)$-subgraphs isomorphic to a union of two consecutive isomorphic copies of $F_0$, that is, $F_0^*\cong F[[2r]]$. Note that $F_0^*$ has $\lambda^*\geq2\lambda+\Delta/2$ edges. In the same way as in Section~\ref{sc:KNP_conjecture_resolution}, we define the family $\mathcal{F}_n(\overrightarrow{D})\subset\mathcal{F}_n$, for a $d$-tuple $\overrightarrow{D}$ of disjoint copies of $F_0^*$. We then fix $C>0$ large enough and sample an $m$-subset $\mathbf{W}\subset{[n]\choose 2}\setminus(D_1\cup\ldots\cup D_{\chi})$ uniformly at random, where $m=\lfloor(1+\varepsilon)(e/n)^{2/d}\cdot N\rfloor$. We also let $\mathcal{F}^*_n(\overrightarrow{D})$ be the set of all $F\setminus(D_1\cup\ldots\cup D_{\chi})$, $F\in\mathcal{F}_n(\overrightarrow{D})$. In the same way, for $F\in\mathcal{F}^*_n(\overrightarrow{D})$ and an $m$-set $W\subset{[n]\choose 2}\setminus(D_1\cup\ldots\cup D_{\chi})$, we say that the pair $(F,W)$ is $\ell$-bad, if at least $n^{-1/2}$-fraction of $F'\subset F\cup W$ that belong to $\mathcal{F}^*_n(\overrightarrow{D})$ have $|F\cap F'|>\ell$. We have that Claim~\ref{cl:not_bad} holds for these $\overrightarrow{D}$, $\mathcal{F}^*_n(\overrightarrow{D})$, $m$, and $\ell_0=\lfloor Cn^{(2d-\Delta)/d}\rfloor$, as well. Indeed, the proof of Claim~\ref{cl:many_planted_copies} remains unchanged. Claim~\ref{cl:not_bad} follows then from 1) the fact that, for every $t\in\{0,1,\ldots,dn/2-\lambda^*\chi\}$,
$$
\frac{\E X'}{M(t)}  
\leq\sum_{\ell\geq \ell_0+1}\Pi_{\ell}e^{-\frac{(dn/2)^2}{m}(1-o(1))}\left((1+o(1))\frac{N }{m}\right)^{\ell},
$$ 
where 
\begin{itemize}
\item $\Pi_{\ell}=\Prob(F\cap \mathbf{F}\text { has }\ell\text{ edges})$ and $\mathbf{F}\in\mathcal{F}^*_n(\overrightarrow{D})$ is uniformly random;
\item $X'$ is the number of $F'\subset F\cup \mathbf{W}'$, $F'\in\mathcal{F}^*_n(\overrightarrow{D})$, such that $|F'\cap F|\geq\ell_0+1$, and $\mathbf{W}'$ is a uniformly random $(m-t)$-subset of ${[n]\choose 2}\setminus (F\cup D_1\cup\ldots\cup D_{\chi})$;
\item $M(t)$ is the expected number of $F'\in\mathcal{F}^*_n(\overrightarrow{D})$ such that $F'$ belongs to a uniformly random subset of ${[n]\choose 2}\setminus(D_1\cup\ldots\cup D_{\chi})$ of size $m+(dn/2-\lambda^*\chi)-t$.
\end{itemize}
and 2) the following analogue of Claim~\ref{cl:lm:main}:
\begin{claim} We have that
$$
\sum_{\ell\geq \ell_0+1}\Pi_{\ell}e^{-\frac{(dn/2)^2}{m}(1-o(1))}\left((1+o(1))\frac{N }{m}\right)^{\ell}=\exp(-\Omega(n^{2/d})).
$$
\end{claim}
The proof of the latter claim follows the same reasoning as the proof of Claim~\ref{cl:lm:main}; therefore, we omit it for brevity and to avoid repetitions. We only note that the bound~\eqref{eq:p_ell_x_c_bound} becomes
$$
 p(\ell,x,c)\leq\frac{{n\choose c}{x\choose c+\chi}\left(\frac{dx+2\chi}{\chi}\right)^{\chi} e^{A_1c+A_2\sigma}\left(\max_{o\leq (d+4)\sigma}{x\choose o}\right)^2(n-x-2r\chi+c)!}{|\mathcal{F}_n(\overrightarrow{D})|}.
$$
In particular, the bound on the number of extensions of a subgraph $H\subset F$ carries over to this setting, as every subgraph $H\subset F$ that contains $D_1$ has no non-trivial automorphisms that fix the vertices of $D_1$ as well as the boundary vertices of $H$. Since $|\mathcal{F}_n(\overrightarrow{D})|=(n-2r\chi)!$ and $x-c=\frac{2\ell}{d}+\frac{c(\Delta-d)}{d}+\frac{2\sigma}{d}$, the rest of the proof of Claim~\ref{cl:lm:main} remains unchanged.

In the same way as in Section~\ref{sc:distribution_of_fragments}, for every pair $(F\in\mathcal{F}_n,W\in{{[n]\choose 2}\choose m})$, we define the distribution $\mathcal{Q}_{F,W}$ over the set of $F'\in\mathcal{F}_n(\overrightarrow{D}(F))$ such that $F'\subset F\cup W$. We fix any positive $\delta<1/d$ and call the pair $(F,W)$ {\it separating}, if for every set $\mathcal{X}$ of disjoint closed subgraphs of $F$ on at least $n^{(2d-\Delta)/d}/\ln n$ vertices, the probability that the set of inclusion-maximal closed subgraphs of $\mathbf{F}'\cap F$, where $\mathbf{F}'\sim\mathcal{Q}_{F,W}$, contains $\mathcal{X}$ is at most $\left\lfloor\frac{\ln^2 n}{n^{\delta}}\right\rfloor$. We then generate $\mathbf{F}'(F,W)\sim\mathcal{Q}_{F,W}$ for every pair $(F,W)$ and get fragments $F\cap\mathbf{F}'(F,W)$.

The key novel ingredient in the proof of Theorem~\ref{th:second_power} --- the improvement of fragments, described in Section~\ref{sc:improvement} --- applies here as well. The main difference is that diamonds are replaced here with isomorphic copies of $F^*_0$. Let $\mu:=\left\lfloor\frac{10C\ln n}{r w}\right\rfloor$. We also consider all inclusion-maximal closed subgraphs of a fragment $H=H(F,W)$ of size at most $\ell_0+\lambda^*\chi$ with at least $r\mu$ vertices, for $F\in\mathcal{F}_n(\overrightarrow{D})$. Here we cannot cut a closed subgraph in an arbitrary place since we have to keep pieces of size $r$ entirely. For each considered closed subgraph $P_i=(v^i_1 v^i_2\ldots)$, we find the minimum $r'$ such that $H[\{v^i_{r'+1}\ldots v^i_{r'+r}\}]\cong F_0$, the bijection preserving the order of vertices is an isomorphism, and $v^i_{r'+1},\ldots, v^i_{r'+r}\notin\partial_v(H)$. Then we remove from $H$ all the consecutive pieces $(v^i_{r'+1}\ldots v^i_{r'+r\mu})$, $(v^i_{r'+r\mu+1}\ldots v^i_{r'+2r\mu})$, $\ldots$, of size $r\mu$. We also have to make sure that all vertices of the last removed piece are not boundary. Finally, we glue the two remaining pieces of every $P_i$, and insert each removed piece between the two isomorphic copies of $F_0$ of the {\it respective} $D_i\cong F_0^*$. 

As in the proof of Theorem~\ref{th:second_power}, we match every piece $P_i$ with $D_{j_i}$ randomly: Consider a binomial random bipartite graph $\mathbf{B}$ with parts $V=[\chi]$ and $U=[n]$ and edge probability $\beta:=n^{(\Delta-d)/(3d)-(2d-\Delta)/d}/w$; insert every $P_i$ into some $D_{j_i}$ so that 
\begin{itemize}
\item there exists an edge in $\mathbf{B}$ between $j_i\in V$ and the position of the first vertex of $P_i$ in $F$,
\item the component containing $D_{j_i}$ does not have closed subgraphs of size at least $r\mu$, i.e. $j_i\in V'$.
\end{itemize}
This is typically possible due to Claim~\ref{cl:random-bipartite}, that holds in this setting as well. For a bipartite graph $B$ on $V\times U$ with typical degrees, we consider the (random) set $\Sigma=\Sigma(B)$ of all separating pairs $(F\in\mathcal{F}_n,W\in{{[n]\choose 2}\choose m})$ such that $\mathbf{F}'=\mathbf{F}'(F,W)\sim\mathcal{Q}_{F,W}$ satisfies the following
\begin{itemize}
\item[(1)] $\mathbf{F}'\subset F\cup W$,
\item[(2)] $|\mathbf{F}'\cap F|\leq\ell_0+\lambda^*\chi$,
\item[(3)] $B[V'\times U]$ has a matching that covers all vertices in $U$ that represent positions of first vertices of the removed pieces of the fragment $H=F\cap\mathbf{F}'$.
\end{itemize}
Due to symmetry and linearity of expectation, we get that there exists $B$ such that whp the uniformly random vector $(\mathbf{F},\mathbf{W})$ belongs to $\Sigma$, and fix such $B$. % and realisations of all $\mathbf{F}'(F,W)$ such that (asymptotically) almost all pairs $(F,W)$ belong to $\Sigma$. We fix such $B$ and $F'=\mathbf{F}'(F,W)$
 In order to complete analysis of smoothed fragments, we state the following analogue of Claim~\ref{cl:pre-images-size}. We fix $W$, let $H(F):=F\cap \mathbf{F}'$, and let $H'(F)$ be the smoothed version of $H(F)$, according to $B$.
%Let $W\in{{[n]\choose 2}\choose m}$.
% For $W\in{{[n]\choose 2}\choose m}$, let $\Sigma_W$ be the set of all $F\in\mathcal{F}_n$ such that $(F,W)\in\Sigma$ is separating. For every $F\in\Sigma_W$, we pick $\mathbf{F}'\sim\mathcal{Q}_{F,W}$ independently. Let $H(F):=F\cap \mathbf{F}'$ and $H'(F)$ be the smoothed version of $H(F)$.
\begin{claim}
\label{cl:improvement_general}
 For any graph $S$, let $\mathcal{S}(S)$ be the set of all graphs $F$ such that $(F,W)\in\Sigma$ and $H'(F)$ contains $S$ as a subgraph. Let $c:=c(S)$, $\ell:=\ell(S)$, $x:=x(S)$, $\sigma:=\sigma(S)=\frac{d}{2}x(S)-\frac{\Delta}{2}c(S)-\ell(S)$, and $\ell_{\chi}:=\ell_0+\lambda^*\chi$. Then 
$$
 |\mathcal{S}(S)|\leq r^{c}(16d)^{\ell}(n-x+c)!\max\left\{n^{c(\Delta-d)/(3d)}{\ell_{\chi}\choose \min\{\sigma,\lfloor\ell_{\chi}/2\rfloor\}},{\ell_{\chi}\choose \min\{c+\sigma,\lfloor\ell_{\chi}/2\rfloor\}}\right\}
$$
almost surely. Moreover, the bound
$$ 
|\mathcal{S}(S)|\leq r^{c}(16d)^{\ell}(n-x+c)!\frac{(2\ell_{\chi})^{c+\sigma}}{\min\left\{(\ln n)^{c+\sigma},n^{1/(2d)}\right\}}
$$
holds with probability $1-e^{-\Omega(n^{1/(2d)})}$.
\end{claim}
We omit the proof since it is almost identical with the proof of Claim~\ref{cl:pre-images-size}.
Actually, the only difference is that when we perform the step (1) and reconstruct the order of vertices in every connected component of $S$, we choose a vertex of $F_0$ that corresponds to the root of the connected component ($r$ choices for every component).

Then, as in Section~\ref{sc:day_1}, we  perform one fragmentation step and reduce the size of a fragment by the factor of $\Omega\left(n^{\frac{\Delta-d}{2d}-\delta}\right)$, for any $\delta>0$. Indeed, let $\mathcal{H}^{(1)}$ be the set of smoothed fragments, let $\mathbf{H}$ be a uniformly random element of $\mathcal{H}^{(1)}$, and let $\mathbf{W}^{(1)}\sim G(n,m')$, $m'=\lfloor\varepsilon n^{-2/d}\cdot N\rfloor$, be sampled independently.
\begin{claim}
Let $\delta>0$. Whp there exists $H\subset \mathbf{H}\cup\mathbf{W}^{(1)}$ such that $H\in\mathcal{H}^{(1)}$ and 
$$
|\mathbf{H}\cap H|\leq n^{\frac{2d-\Delta}{d}-\frac{\Delta-d}{2d}+\delta}.
$$
\end{claim}
In order to prove this claim, we apply  Lemma~\ref{lm:not_bad} with $f=\ell_0+\lambda^*\chi$ and $\mathcal{B}$ being the set of all graphs on $n$ vertices with more than $f^{(1)}:=n^{\frac{(2d-\Delta)}{d}-\frac{\Delta-d}{2d}+\delta}$ edges. It is then sufficient to show that, uniformly in $H\in\mathcal{H}^{(1)}$,
$$
 \sum_{\ell\geq f^{(1)}+1}\Pi^H_{\ell}\left((1+o(1))\frac{N}{m'}\right)^{\ell}=o\left(\frac{1}{\log n}\right),
$$
where $\Pi^H_{\ell}=\Prob(H\cap \mathbf{H}\text { has }\ell\text{ edges})$. The rest of the proof is similar to the proof of Claim~\ref{cl:not_bad_2} and follows from Claim~\ref{cl:improvement_general}. Therefore, we only outline the differences. We first notice that the bound~\eqref{eq:n-x+c} here becomes
$$
\frac{(n-x+c)!}{n!}\leq\frac{1}{n^{x-c}(1-\frac{x-c}{n})^{x-c}}=n^{-x+c}\cdot e^{(1-o(1))\frac{(x-c)^2}{n}}.
$$
Therefore, the upper bound on $\Pi^H_{\ell}$ contains an extra factor $r^c e^{(1-o(1))\frac{(x-c)^2}{n}}\leq r^c e^{\ell}$. We choose $C'$ large enough and use the bound~\eqref{eq:sigma_lower_2d-power} that holds here as well. Since $r^c=n^{o(c)}$, we finally get the required bound
\begin{align*}
\sum_{\ell\geq f^{(1)}+1}\Pi^H_{\ell}\left((1+o(1))\frac{N}{m'}\right)^{\ell} 
&\leq \sum_{c,\ell,x}\left[\frac{{\ell_0+\lambda^*\chi\choose c}n^{(\Delta-d)c/(3d)}}{n^{(\Delta-d)c/d-\delta c}} e^{-C'\ell}+\left(\frac{(\ell_0+\lambda^*\chi)^2n^{\delta/2}}{n^{(\Delta-d)/d}(c+\sigma)^2}\right)^{c+\sigma}\right]\\
&=e^{-\Omega(f^{(1)})}.
\end{align*}

We get a multiset of fragments $\mathcal{H}^{(2)}$. Next, as in Section~\ref{sc:day_2}, we make $R-1=O(\log\log n)$ fragmentation steps by sampling $\mathbf{W}^{(2)},\ldots,\mathbf{W}^{(R)}\sim G(n,m')$, where $m'=\left\lceil \frac{\varepsilon}{n^{2/d}\ln\ln n}N\right\rceil$. An argument identical to the proof of Claim~\ref{cl:day_2} implies that, after the last fragmentation round, we get a family $\mathcal{H}^{(R+1)}$ of fragments of size $\lfloor (\ln n)^{1-\Theta(1)}\rfloor$ whp. It remains to cover at least one fragment of size $\lfloor (\ln n)^{1-\Theta(1)}\rfloor$ from the multiset $\mathcal{H}^{(R+1)}$ by the last sample $\mathbf{G}'\sim G(n,p)$, $p=\varepsilon \cdot n^{-2/d}$, whp. The proof follows from Claim~\ref{cl:improvement_general}, Chebyshev's inequality, and is verbatim as the argument in Section~\ref{sc:day_3}. Therefore, we omit it.

\section{Remaining challenges}
\label{sc:challenges}

Although the expectation threshold conjecture of Kahn and Kalai has been resolved~\cite{PP_KK}, the asymptotics of $p_c(\mathcal{Q})$ remain unknown --- even up to a constant factor --- for an arbitrary increasing property $\mathcal{Q}$, given $p_e(\mathcal{Q})$. In particular, it would be very interesting to determine the order of growth of $p_c(\mathcal{Q})$ when $\mathcal{Q}=\mathcal{Q}_F$ is generated by all isomorphic copies of any fixed $d$-regular graph $F$ on $[n]$. Theorem~\ref{th:main1} gives only a partial answer and the problem becomes significantly more challenging when considering graphs $F$ that contain many subgraphs with smaller edge boundaries.

We suspect that the requirements in Theorem~\ref{th:main2} are far from optimal. In particular, Theorem~\ref{th:second_power_generalisation} asserts that~\eqref{eq:sharp} is true for a wide class of $d$-regular graphs $F$ that is not covered by Theorem~\ref{th:main2}. We actually believe that the following is true.
\begin{conjecture}
Let $d\geq 3$ and let $F=F(n)$ be a sequence of $d$-regular graphs on $[n]$, $n\in\mathbb{N}$, such that $|\partial_e(\tilde F)|\geq d+1$, for every $\tilde F\subset F$ with $3\leq|V(\tilde F)|\leq n-3$.
 Then $p_c(F)=(1+o(1))p_e(F)$. 
\end{conjecture}
A notable special case of graphs $F$ covered by this conjecture is a hexagonal lattice. Although the hexagonal lattice is not a regular graph, it can be extended to a 3-regular graph by adding missing edges between the boundary vertices of degree 2. Since the boundary has $O(\sqrt{n})$ vertices, it does not contribute to the threshold asymptotically. This specific $F$ was mentioned by Riordan in~\cite{Riordan}, where he considered it as a special case to which his methods do not apply. It is worth noting that Fernandez de la Vega and Manoussakis~\cite{FdVM} proved that $p_c(F)=O((\log n/n)^{2/3})$ and that Theorem~\ref{th:main1}, as well as~\cite[Corollary 1.5]{CHL}, implies~\eqref{eq:p_c=p_e} for such $F$. % (as well as~\cite[Corollary 1.5]{CHL}), Theorem~\ref{th:main2} is not applicable in this case. Other 

%In particular, for hexagonal lattice $F$, $p_c(F)=\Theta(n^{-2/3})$. However, Theorem~\ref{th:main2} ca Fernandez de la Vega and Manoussakis~\cite{FdVM} proved that $p_c(F)=O((\log n/n)^{2/3})$ and Riordan noticed in
%Theorem~\ref{th:main1} implies thresholds for 3-regular graphs, in contrast to the result of Riordan (Fernandez de la Vega and Manoussakis~\cite{FdVM} proved that $p_c(F)=O((\log n/n)^{2/3})$ and Riordan noticed in~\cite{Riordan} that his methods are not applied to 3-regular graphs). In particular, for hexagonal lattice $F$, $p_c(F)=\Theta(n^{-2/3})$. However, Theorem~\ref{th:main2} ca Fernandez de la Vega and Manoussakis~\cite{FdVM} proved that $p_c(F)=O((\log n/n)^{2/3})$ and Riordan noticed in~\cite{Riordan} that his methods are not applied to 3-regular graphs.

% Getting sharp thresholds for all $d$-regular graphs with edge boundaries at least $d+1$. A special challenge is $d=3$. In particular, for hexagonal grid. Give references. State a conjecture (is it enough that there are at most $e^{o(n)}$ automorphisms? Maybe state only that $p_c=(1+o(1))p_e$?)

It is also interesting to establish stronger versions of Theorems~\ref{th:main2},~\ref{th:second_power},~and~\ref{th:second_power_generalisation} for $d$-regular graphs $F$ that they cover. Although it is plausible to refine the asymptotics of $p_c(F)$ to some extent using similar methods, achieving results as precise as those for Hamilton cycles seems very challenging. Let us recall that, for an $n$-cycle $F$, $p_c(F)=\frac{\ln n+\ln\ln n+O(1)}{n}$~\cite{Ham1} and that the following hitting time result is known~\cite{AKS,Ham3}: in the random graph process $\mathbf{W}=\mathbf{W}_m$, $m=0,1,\ldots,N$, the hitting time for $\mathcal{Q}_F$ coincides with the hitting time for $\delta(\mathbf{W})\geq 2$ whp. It is clear that for denser regular graphs $F$ hitting time versions of Theorems~\ref{th:main2},~\ref{th:second_power},~and~\ref{th:second_power_generalisation} do not hold: for example, it is possible to show that, for any $k=(1-\Theta(1))n$, whp the hitting time for containing the square of a Hamilton cycle is strictly greater than the hitting time for the property that every vertex belongs to the square of a path of length at least $k$. On the other hand, our proof method suggests that the following question may have a positive answer: Is there a (small) constant $k\in\mathbb{N}$ such that $p_c(F)\leq p_e(F)+p_c(\mathcal{Q}_k)$, where $F$ is the square of a cycle on $[n]$ and $\mathcal{Q}_k$ denotes the property that every vertex belongs to the square of a path of length $k$? In particular, is it true for $k=2$, i.e. when $\mathcal{Q}_k$ is the property that there are no isolated vertices?

%It would be very interesting to prove hitting time versions of Theorems~\ref{th:main2},~\ref{th:second_power},~and~\ref{th:second_power_generalisation}. It is clear that for denser regular graphs $F$ hitting times for such local properties, that involve only neighbourhoods of bounded radii, cannot coincide with the hitting times for $\mathcal{Q}_F$. Nevertheless, one may ask the following question: is there $k=o(n)$ such that the hitting time for containing the second power of a Hamilton cycle coincides with the hitting time for the property that every vertex belongs to the second power of a path of length $k$ whp? 

\section*{Acknowledgements}

This work originated during the author's visit to Tel Aviv University in the summer of 2022. The author is grateful to Wojciech Samotij for his warm hospitality and for helpful discussions throughout the visit. The author would like to thank Asaf Cohen Antonir, Sahar Diskin, Ilay Hoshen, and Michael Krivelevich  for valuable feedback on the presentation of the paper. The author would also like to thank J\'{o}zsef Balogh and Robert A. Krueger for very useful comments on last chapters of the paper and for helpful discussions at the Banff International Research Station during the workshop ``Bootstrap Percolation and its Applications'', as well as the organisers of this event. 

An early version of this paper, uploaded to arXiv in January 2023, contained a crucial error that invalidated the proof of Theorem~\ref{th:second_power}. The author sincerely thanks Ashwin Sah and Mehtaab Sawhney for identifying the mistake and for subsequent discussions and valuable comments.

\appendix

\section{Coarse bound: proof of Theorem~\ref{th:main1}}
\label{sc:theorem_coarse_proof}

Let $F$ be as in the statement of Theorem~\ref{th:main1}. Let $B>0$ be a large enough constant. Consider $d$ independent samples 
$$
\mathbf{W}_1,\mathbf{W}_2,\ldots,\mathbf{W}_d\sim G(n,m'),\quad\text{ where }
\quad m'=\lfloor B\cdot n^{-2/d}\cdot N\rfloor.
$$
Let $\mathcal{F}_n$ be the family of all isomorphic copies of $F$ on $[n]$. %\MZ{$\mathbf{W}$ for uniform model everywhere?}

\begin{claim}
\label{cl:coarse_day_1}
Whp there exists $F'\subset F\cup\mathbf{W}_1$, $F'\in\mathcal{F}_n$, such that $|F\cap F'|\leq n^{1-1/d}$.
\end{claim}

\begin{proof}
Let $f=\frac{dn}{2}$, $m=m'$, and $\mathcal{B}$ be the property of graphs on $n$ vertices to have more than $n^{1-1/d}$ edges. Due to Lemma~\ref{lm:not_bad} and due to symmetry, it is sufficient to prove~\eqref{eq:sufficient_first_fragmentation}. 

Fix non-negative integers $c$ and $x$. Let us denote by $\mathcal{J}_{\ell,x,c}$ the set of all subgraphs $J\subset F$ with $\ell$ edges, $x$ non-isolated vertices, and $c$ connected components (excluding isolated vertices). Denote by $p(\ell,x,c)$ the probability that $\mathbf{F}\cap F\in\mathcal{J}_{\ell,x,c}$. Due to Claim~\ref{cl:J_upper_bound} and Claim~\ref{cl:simple_embeddings}, % and since 
%\begin{equation}
%\left(\frac{x}{c}\right)^c=e^{c\ln (x/c)}=e^{x\cdot\frac{\ln(x/c)}{x/c}}\leq e^x,
%\label{eq:x/c}
%\end{equation}
 for some constant $A>0$,
$$
p(\ell,x,c)\leq\frac{{n\choose c}e^{A\ell}(n-x+c)!/|\mathrm{Aut}(F)|}{|\mathcal{F}_n|}.
$$

Let $\ell>n^{1-1/d}$. Recalling that $x=\frac{2}{d}\ell+\frac{c\Delta}{d}+\frac{2}{d}\sigma$, assuming $n-x+c\geq 1$, we get
\begin{align}
 \frac{(n-x+c)!/|\mathrm{Aut}(F)|}{|\mathcal{F}_n|} & =
 \frac{(n-\frac{2\ell+c(\Delta-d)+2\sigma}{d})!}{n!}\notag\\
 &\leq (1+o(1))\frac{e^{\frac{2\ell+c(\Delta-d)+2\sigma}{d}}}{n^{\frac{2\ell+c(\Delta-d)+2\sigma}{d}}}\left(\frac{n-\frac{2\ell+c(\Delta-d)+2\sigma}{d}}{n}\right)^{n-\frac{2\ell+c(\Delta-d)+2\sigma}{d}}\notag\\
 &\leq (1+o(1))\frac{e^{\frac{2\ell+c(\Delta-d)+2\sigma)^2}{d^2n}}}{n^{\frac{2\ell+c(\Delta-d)+2\sigma}{d}}}
 \leq (1+o(1))\cdot\frac{e^{\frac{(x-c)^2}{n}}}{n^{\frac{2\ell+c+2\sigma}{d}}}.
 \label{eq:factorials_fraction}
\end{align}
Clearly, the same bound holds when $n-x+c=0$ as well. Therefore, since $x-c\leq \min\{\ell,n\}$,
$$
p(\ell,x,c)\leq\frac{\left(\frac{en^{1-1/d}}{c}\right)^c e^{(A+2)\ell}}{n^{\frac{2\ell+2\sigma}{d}}}\leq\frac{e^{n^{1-1/d}} e^{(A+2)\ell}}{n^{\frac{2\ell+2\sigma}{d}}}.
$$
The requirement on the edge boundary from the statement of Theorem~\ref{th:main1} implies $\sigma\geq -\frac{2d\ell}{\varepsilon\ln n}$. Indeed, if a component of $J\in\mathcal{J}_{\ell,x,c}$ has size $x_i\leq\varepsilon\ln n$, then the respective $\sigma_i:=\frac{d}{2}x_i-\ell_i-\frac{\Delta}{2}$ is non-negative. If $x_i>\varepsilon\ln n$, then its number of edges $\ell_i$ satisfies $2\ell_i\leq dx_i$. Thus, recalling that $x_i\leq 2\ell_i$, we get
$$
\sigma_i=\frac{d}{2}x_i-\ell_i-\frac{\Delta}{2}\geq-\frac{\Delta}{2}\geq -\frac{\Delta x_i}{2 \varepsilon\ln n}\geq -\frac{d x_i}{\varepsilon\ln n}\geq -\frac{2d \ell_i}{\varepsilon\ln n}.
$$
So,
\begin{align*}
\sum_{\ell}\Pi^F_{\mathcal{B}_{\ell}}\left(\left(1+\frac{3f}{m}\right)\frac{N}{m}\right)^{\ell}e^{-f^2/m+f^3/(3m^2)}& \leq
\sum_{\ell>n^{1-1/d}}\sum_{x,c}p(\ell,x,c)\left(\frac{1+o(1)}{B}\cdot n^{2/d}\right)^{\ell}\\
&\leq n^2\sum_{\ell>n^{1-1/d}}e^{n^{1-1/d}} \left(\frac{e^{A+2+4/\varepsilon}+o(1)}{B}\right)^{\ell}\\
&=o\left(\frac{1}{n^2}\right),
\end{align*}
completing the proof of the claim.
\end{proof}

For every $F\in\mathcal{F}_n$ such that there exists $F'$ as in the statement of Claim~\ref{cl:coarse_day_1}, we choose one such $F'$ and put $F\cap F'$ into a multiset $\mathcal{F}^{(1)}_n$. Due to Claim~\ref{cl:coarse_day_1} and Markov's inequality, whp $|\mathcal{F}^{(1)}_n|=(1-o(1))|\mathcal{F}_n|$. We then proceed by induction. Assume that, for $i\in[d-1]$, whp we have a multiset $\mathcal{F}^{(i)}_n$ of graphs with exactly $\lfloor n^{1-i/d}\rfloor$ edges that is obtained by adding a single $H\subset F$ from almost every $F\in\mathcal{F}_n$ such that the graph $H\cup\mathbf{W}_1\cup\ldots\cup\mathbf{W}_i$ contains some $F'\in\mathcal{F}_n$. We have that $|\mathcal{F}^{(i)}_n|=(1-o(1))|\mathcal{F}_n|$ whp. Let $\mathbf{H}$ be a uniformly random element of $\mathcal{F}^{(i)}_n$.

\begin{claim}
\label{cl:coarse_induction}
Whp there exists $H'\subset \mathbf{H}\cup\mathbf{W}_{i+1}$ such that $H'\in\mathcal{F}^{(i)}_n$ and 
$$
|\mathbf{H}\cap H'|\leq \max\left\{n^{1-(i+1)/d},\ln n\right\}.
$$
\end{claim}

\begin{proof}
Let $f=\lfloor n^{1-i/d}\rfloor$, $m=m'$, and $\mathcal{B}$ be the property of graphs on $n$ vertices to have more than $n^{1-(i+1)/d}$ edges. Due to Lemma~\ref{lm:not_bad}, it is sufficient to prove~\eqref{eq:main_lm}.

Fix non-negative integers $c$ and $x$. For $H\in \mathcal{F}^{(i)}_n$, let us recall that $\mathcal{J}^H_{\ell,x,c}$ is the set of all subgraphs $J\subset H$ with $\ell$ edges, $x$ non-isolated vertices, and $c$ connected components (excluding isolated vertices). Denote by $p^H(\ell,x,c)$ the probability that $\mathbf{H}\cap H\in\mathcal{J}^H_{\ell,x,c}$. % We shall prove the following easy estimate on $|\mathcal{J}^H_{\ell,x,c}|$.
 Due to Claim~\ref{cl:J_upper_bound}, Claim~\ref{cl:simple_embeddings}, and estimate~\eqref{eq:factorials_fraction}, for some constant $A>0$,
\begin{align*}
p^H(\ell,x,c) & \leq\frac{{\lfloor n^{1-i/d}\rfloor\choose c}e^{A\ell}(n-x+c)!/|\mathrm{Aut}(F)|}{|\mathcal{F}^{(i)}_n|}\\
&\leq(1-o(1))\frac{\left(\frac{en^{1-(i+1)/d}}{c}\right)^{c}e^{(A+2)\ell}}{n^{2(\ell+\sigma)/d}}\leq(1-o(1))\frac{e^{n^{1-(i+1)/d}}e^{(A+2)\ell}}{n^{2(\ell+\sigma)/d}}.
\end{align*}
Recalling that $\sigma\geq-\frac{C\ell}{2\ln n}$, for every $H\in\mathcal{F}_n^{(i)}$,
\begin{align*}
\sum_{\ell}\Pi^H_{\mathcal{B}_{\ell}} & \left(\left(1+\frac{3f}{m}\right)\frac{N}{m}\right)^{\ell}e^{-f^2/m+f^3/(3m^2)}\\
& \leq
\sum_{\ell>\max\{n^{1-(i+1)/d},\ln n\}}\sum_{x,c}p^H(\ell,x,c)\left(\frac{1+o(1)}{B}\cdot n^{2/d}\right)^{\ell}\\
&\leq n^2\sum_{\ell>\max\{n^{1-(i+1)/d},\ln n\}}e^{n^{1-(i+1)/d}}\left(\frac{e^{A+2+4/\varepsilon }+o(1)}{B}\right)^{\ell}=o\left(\frac{1}{n^2}\right),
\end{align*}
as desired.
\end{proof}

By induction, whp we get a multiset $\mathcal{F}^{(d)}_n=\mathcal{F}^{(d)}_n(\mathbf{W}_1\cup\ldots\cup\mathbf{W}_d)$ of graphs of size $\lfloor \ln n\rfloor$ comprising a single $H\subset F$ from almost every $F\in\mathcal{F}_n$ such that the graph $H\cup\mathbf{W}_1\cup\ldots\cup\mathbf{W}_d$ contains some $F'\in\mathcal{F}_n$. %It is well known that increasing properties that hold whp in the uniform model hold whp in the respective binomial model as well (see, e.g.,~\cite[Corollary 1.16]{Janson}). In particular,
  Let
$$
 \mathbf{G}\sim G(n,p),\quad\text{ where }\quad p=(1+\varepsilon)dBn^{-2/d}.
$$
By~\cite[Corollary 1.16]{Janson}, whp there exists a multiset $\mathcal{F}^{(d)}_n=\mathcal{F}^{(d)}_n(\mathbf{G})$ of graphs of size $\lfloor\ln n\rfloor$ comprising a single subgraph $H$ of almost every $F\in\mathcal{F}_n$ so that $H\cup\mathbf{G}$ contains some $F'\in\mathcal{F}_n$. 

Let $X$ be the number of $H\in\mathcal{F}^{(d)}_n$ that belong to $\mathbf{G}'\sim G(n,p'=Bn^{-2/d})$, sampled independently of $\mathbf{G}$. We get 
$$
\mathbb{E}X=|\mathcal{F}_n^{(d)}|p'^{\lfloor\ln n\rfloor }=(1-o(1))|\mathcal{F}_n|p'^{\lfloor\ln n\rfloor }=\omega(1).
$$
Let $\mathcal{B}$ be the set of non-empty graphs. Due to the definition~\eqref{eq:Pi_def} of $\Pi^H_{\mathcal{B}_{\ell}}=\Pi^H_{\mathcal{B}_{\ell}}(\mathcal{F}^{(d)}_n)$,
$$
 \frac{\mathrm{Var}X}{(\mathbb{E}X)^2} \leq\frac{\max_{H\in\mathcal{F}_n^{(d)}}\sum_{\ell\geq 1}\Pi^H_{\mathcal{B}_{\ell}}|\mathcal{F}_n^{(d)}|p'^{\lfloor\ln n\rfloor-\ell}}{\mathbb{E}X}=\max_{H\in\mathcal{F}_n^{(d)}}\sum_{\ell\geq 1}\Pi^H_{\mathcal{B}_{\ell}}p'^{-\ell}.
$$
Therefore, due to Claim~\ref{cl:J_upper_bound}, Claim~\ref{cl:simple_embeddings}, and estimate~\eqref{eq:factorials_fraction}, for some constant $A>0$, 
\begin{align*}
 \frac{\mathrm{Var}X}{(\mathbb{E}X)^2} & \leq
\sum_{\ell\geq 1}\sum_{x,c}\frac{{\lfloor\ln n\rfloor\choose c}e^{(A+2)\ell}}{n^{2\ell/d+c/d+2\sigma/d}}\left(\frac{1}{B}\cdot n^{2/d}\right)^{\ell}\\
&\leq \sum_{\ell\geq 1}\sum_{c\geq 1} O(\ell) \left(\frac{e\ln n}{cn^{1/d}}\right)^c\left(\frac{e^{A+2+4/\varepsilon}}{B}\right)^{\ell}=O\left(\frac{\ln n}{n^{1/d}}\right),
\end{align*}
completing the proof of Theorem~\ref{th:main1}, due to Chebyshev's inequality.

\section{Proofs of Claims~\ref{cl:many_planted_copies},~\ref{cl:not_bad}}
\label{sc:appendix_claims}

\begin{claim}
Let $F\in\mathcal{F}^*_n(\overrightarrow{D})$ and let $\mathbf{W}$ be a uniformly random $m$-element subset of ${[n]\choose 2}\setminus(D_1\cup\ldots\cup D_d)$. Let $\delta(n)\to 0$ as $n\to\infty$. Then, for every $t\in\{0,1,\ldots,2n-5d\}$,
$$
 \mathbb{P}\left(|\mathcal{M}(F,\mathbf{W})|<\delta(n)M \mid | F\cap\mathbf{W}|=t\right)\leq\delta(n).
$$ 
\end{claim}

\begin{proof}
Sample a uniformly random $\mathbf{F}\in\mathcal{F}^*_n(\overrightarrow{D})$ independently of $\mathbf{W}$.
Then
\begin{multline*}
 \mathbb{P}\left(|\mathcal{M}(\mathbf{F},\mathbf{W})|<\delta(n) M\mid|\mathbf{F}\cap\mathbf{W}|=t\right)=\\
 \sum_{F\in\mathcal{F}^*_n(\overrightarrow{D})}\mathbb{P}\left(\mathbf{F}=F,|\mathcal{M}(F,\mathbf{W})|<\delta(n) M\mid|\mathbf{F}\cap\mathbf{W}|=t\right)\\
 =\sum_{F\in\mathcal{F}^*_n(\overrightarrow{D})}\mathbb{P}\left(|\mathcal{M}(F,\mathbf{W})|<\delta(n) M\mid \mathbf{F}=F,|\mathbf{F}\cap\mathbf{W}|=t\right)\\
 \times\mathbb{P}(\mathbf{F}=F\mid |\mathbf{F}\cap\mathbf{W}|=t)\\
  =\sum_{F\in\mathcal{F}^*_n(\overrightarrow{D})}\mathbb{P}\left(|\mathcal{M}(F,\mathbf{W})|<\delta(n) M\mid |F\cap\mathbf{W}|=t\right)\cdot\mathbb{P}(\mathbf{F}=F\mid |\mathbf{F}\cap\mathbf{W}|=t).
\end{multline*}
By symmetry, $\mathbb{P}(|F\cap\mathbf{W}|=t)$ and $\mathbb{P}\left(|\mathcal{M}(F,\mathbf{W})|<\delta(n) M, \, |F\cap\mathbf{W}|=t\right)$ do not depend on $F\in\mathcal{F}^*_n(\overrightarrow{D})$. 

Therefore,
\begin{align*}
\mathbb{P}(\mathbf{F}=F\mid |\mathbf{F}\cap\mathbf{W}|=t) &=
\frac{\mathbb{P}(\mathbf{F}=F,\,|F\cap\mathbf{W}|=t)}{\mathbb{P}(|\mathbf{F}\cap\mathbf{W}|=t)}
=
\frac{\mathbb{P}(\mathbf{F}=F)\Prob(|F\cap\mathbf{W}|=t)}{\sum_{\tilde F\in\mathcal{F}^*_n(\overrightarrow{D})}\mathbb{P}(\mathbf{F}=\tilde F)\Prob(|\tilde F\cap\mathbf{W}|=t)}\\
&=
\frac{\Prob(|F\cap\mathbf{W}|=t)}{\sum_{\tilde F\in\mathcal{F}^*_n(\overrightarrow{D})}\Prob(|\tilde F\cap\mathbf{W}|=t)}=
\frac{1}{|\mathcal{F}^*_n(\overrightarrow{D})|}
\end{align*} 
and 
$$
 \mathbb{P}\left(|\mathcal{M}(F,\mathbf{W})|<\delta(n) M\mid |F\cap\mathbf{W}|=t\right)=\mathbb{P}\left(|\mathcal{M}(F',\mathbf{W})|<\delta(n) M\mid |F'\cap\mathbf{W}|=t\right)
$$
for all $F'\in\mathcal{F}^*_n(\overrightarrow{D})$. We also notice that each pair $\{F\in\mathcal{F}^*_n(\overrightarrow{D}),W\in{E(K_n)\setminus(D_1\cup\ldots\cup D_d)\choose m}\}$ such that $|\mathcal{M}(F,W)|<\delta(n) M$ and $|F\cap W|=t$  can be obtained by
\begin{itemize}
\item first, choosing a set of edges $A\subset[n]$ of size $m+(2n-5d)-t$, on the role of $F\cup W$, that contains less than $\delta(n) M$ graphs $F\in\mathcal{F}^*_n(\overrightarrow{D})$ --- in at most ${N-5d\choose m+(2n-5d)-t}$ ways,
\item then, choosing an $F\in\mathcal{F}^*_n(\overrightarrow{D})$ such that $F$ is a subgraph of $A$ --- in less than $\delta(n) M$ ways, and, 
\item finally, choosing the intersection $F\cap W$ --- in ${2n-5d\choose t}$ ways.
\end{itemize}  
We conclude that
\begin{align*}
 \mathbb{P}\left(|\mathcal{M}(F,\mathbf{W})|<\delta(n) M\mid |F\cap\mathbf{W}|=t\right) & =
 \mathbb{P}\left(|\mathcal{M}(\mathbf{F},\mathbf{W})|<\delta(n) M\mid|\mathbf{F}\cap\mathbf{W}|=t\right)\\
 &\leq\frac{\delta(n) M{N-5d\choose m+(2n-5d)-t}{2n-5d\choose t}}{|\mathcal{F}^*_n(\overrightarrow{D})|{N-2n\choose m-t}{2n-5d\choose t}}=\delta(n),
\end{align*}
completing the proof.
\end{proof}

\begin{claim}
Let $F\in\mathcal{F}^*_n(\overrightarrow{D})$ and let $\mathbf{W}\subset{[n]\choose 2}\setminus(D_1\cup\ldots\cup D_d)$ be a uniformly random $m$-subset. Then $(F,\mathbf{W})$ is not $\ell_0$-bad with probability at least $1-\exp(-\Omega(\sqrt{n}))$.
\end{claim}

\begin{proof}
Fix $\varepsilon'>0$ small enough and let $\delta(n)=\exp(-\varepsilon'\sqrt{n})$. Fix $t\in\{0,\ldots,2n-5d\}$. Due to Claim~\ref{cl:many_planted_copies}, we have
$$
 \Prob((F,\mathbf{W})\text{ is $\ell_0$-bad},\, |\mathcal{M}(F,\mathbf{W})|<\delta(n)M(t)
 \mid|F\cap\mathbf{W}|=t)=\exp(-\Omega(\sqrt{n})).
$$ 
Thus, it is sufficient to prove that, uniformly over $t$,
$$
 \Prob((F,\mathbf{W})\text{ is $\ell_0$-bad},\, |\mathcal{M}(F,\mathbf{W})|\geq\delta(n)M(t)
 \mid|F\cap\mathbf{W}|=t)=\exp(-\Omega(\sqrt{n})).
$$
The latter probability equals
\begin{multline*}
\Prob\left(|\mathcal{M}_{\ell_0}(F,\mathbf{W})|>\frac{1}{\sqrt{n}}|\mathcal{M}(F,\mathbf{W})|,\, |\mathcal{M}(F,\mathbf{W})|\geq\delta(n)M(t)
 \mid|F\cap\mathbf{W}|=t\right)\\
 \leq\Prob\left(|\mathcal{M}_{\ell_0}(F,\mathbf{W})|>\frac{\delta(n)}{\sqrt{n}}M(t)
 \mid|F\cap\mathbf{W}|=t\right)
 \leq\frac{\sqrt{n}\cdot\mathbb{E}(X \mid| F\cap\mathbf{W}|=t)}{\delta(n)M(t)},
\end{multline*}
where $X$ counts the number of  $F'\in\mathcal{M}(F,\mathbf{W})$ such that $|F'\cap F|\geq\ell_0+1$. Let $\mathbf{W}'$ be a uniformly random $(m-t)$-subset of ${[n]\choose 2}\setminus (F\cup D_1\cup\ldots\cup D_d)$, and let $X'$ be the number of $F'\in\mathcal{M}(F,\mathbf{W}')$ such that $|F'\cap F|\geq\ell_0+1$. Clearly, 
$$
 \mathbb{E}(X \mid|F\cap\mathbf{W}|=t)=\mathbb{E}X'.
$$
Then
\begin{equation}
 \E X'=\sum_{\ell\geq \ell_0+1}|\mathcal{F}^*_n(\overrightarrow{D})|\Pi_{\ell}
 {m-t\choose (2n-5d)-\ell}/{N-2n\choose (2n-5d)-\ell}.
\label{eq:expectation_general}
\end{equation}
We have
\begin{align}
\frac{\E X'}{M}  
=\sum_{\ell\geq \ell_0+1}|\mathcal{F}^*_n(\overrightarrow{D})|\cdot\frac{\Pi_{\ell}}{M}\cdot
 \frac{{m-t\choose (2n-5d)-\ell}}{{N-2n\choose (2n-5d)-\ell}}
 =\sum_{\ell\geq \ell_0+1}\Pi_{\ell}\cdot\frac{{m-t\choose (2n-5d)-\ell}/{N-2n\choose (2n-5d)-\ell}}{{N-2n\choose m-t}/{N-5d\choose m-t+(2n-5d)}}.
\label{eq:expectation_binomial}
\end{align}
By Stirling's approximation,
\begin{multline*}
 \frac{{m-t\choose (2n-5d)-\ell}/{N-2n\choose (2n-5d)-\ell}}{{N-2n\choose m-t}/{N-5d\choose m-t+(2n-5d)}}=\frac{{m-t\choose (2n-5d)-\ell}{N-5d\choose m-t+(2n-5d)}}{{N-2n\choose m-t}{N-2n\choose (2n-5d)-\ell}}\\
 \sim 
\frac{(m-t)^{2(m-t)}(N-4n+5d+\ell)^{N-4n+5d+\ell}(N-5d)^{N-5d}}{(m-t-2n+5d+\ell)^{m-t-2n+5d+\ell}(m-t+2n-5d)^{m-t+2n-5d}(N-2n)^{2N-4n}}.
\end{multline*}
Therefore, for $\ell\geq\ell_0+1$
\begin{multline*}
 \frac{{m-t\choose (2n-5d)-\ell}/{N-2n\choose (2n-5d)-\ell}}{{N-2n\choose m-t}/{N-5d\choose m-t+(2n-5d)}}\sim\left(\frac{N-4n+5d+\ell}{m-t-2n+5d+\ell}\right)^{\ell}\\
 \times
\frac{\left(1+\frac{2n-5d-\ell}{m-t-2n+5d+\ell}\right)^{m-t-2n+5d}\left(1-\frac{2n-5d}{m-t+2n-5d}\right)^{m-t+2n-5d}}
{\left(1-\frac{2n-5d}{N-5d}\right)^{N-5d}\left(1+\frac{2n-5d-\ell}{N-4n+5d+\ell}\right)^{N-4n+5d}}\\
\quad<\left((1+o(1))\frac{N}{m}\right)^{\ell} e^{2n-5d-\ell-\frac{\ell(2n-5d-\ell)}{m}-\frac{(2n-5d-\ell)^2}{2m}-(2n-5d)-\frac{(2n-5d)^2}{2m}+2n-5d-(2n-5d-\ell)}\\
\quad=e^{-\frac{4n^2}{m}}\left((1+o(1))\frac{N }{m}\right)^{\ell}.
\end{multline*}
Claim~\ref{cl:lm:main} completes the proof.
\end{proof}

\end{document}